%% file: main.tex
\documentclass[10pt, reqno]{amsart}
\usepackage{amssymb}
\usepackage{amsmath}
\usepackage{enumitem}
\usepackage{comment}
\usepackage{breqn}
\usepackage{geometry}
\usepackage{cleveref}
\usepackage{textcase}
\usepackage{esint}
\usepackage[giveninits=true,isbn=false, url=false, doi=false, backend=biber,style=numeric,maxnames=99,minnames=99]{biblatex}
\DeclareFieldFormat[article,incollection]{title}{\mkbibemph{#1}}
\renewbibmacro{in:}{%
  \ifentrytype{article}{}{\printtext{\bibstring{in}\intitlepunct}}%
}
\newtheorem{thm}{Theorem}[section]

\newtheorem{lem}[thm]{Lemma}

\theoremstyle{definition}
\newtheorem{definition}[thm]{Definition}

\newtheorem{claim}{Claim}

\theoremstyle{remark}
\newtheorem{remark}[thm]{Remark}
\newtheorem*{remark*}{Remark}
\newenvironment{claimproof}{\proof[Proof of Claim]}{\endproof}

\newcommand{\sd}{d}
\newcommand{\R}{\mathbb{R}}
\newcommand{\N}{\mathbb{N}}
\newcommand{\I}{\mathcal{I}}

\newcommand{\D}{\mathcal{D}}

\newcommand{\Hold}[1]{C^{#1,\frac{#1}{2}}(\bar{Q}_T)}
\newcommand{\Camp}[2]{L^{#1,#2}_C(Q_T)}
\newcommand{\Mor}[2]{L^{#1,#2}_M(Q_T)}
\newcommand{\nd}{\nabla_{\nu}}

\newcommand{\appuc}{u_C^{(\varepsilon)}}
\newcommand{\appun}{u_N^{(\varepsilon)}}
\numberwithin{equation}{section}
\numberwithin{claim}{section}
\usepackage{xcolor}
\usepackage[foot]{amsaddr}

\title[on a cross-diffusion hybrid model]{On a cross-diffusion hybrid model: Cancer Invasion Tissue with Normal Cell Involved}
\author[Guanjun Pan and Hong-Ming Yin]{Guanjun Pan and Hong-Ming Yin\\Department of Mathematics and Statistics\\ Washington State University\\ Pullman, WA 99164, USA.}

\email{guanjun.pan@wsu.edu and hyin@wsu.edu}

\begin{document}

\maketitle

\begin{abstract}
    \input{Abstract.tex}
\end{abstract}


\input{Introduction.tex}

\input{Main_Result.tex}

\input{Estimate.tex}

\input{Existence.tex}

\input{3d.tex}

\input{Conclusion.tex}
\newpage
\input{Appendix.tex}
\newpage
\printbibliography{}

\end{document}

%% file: Abstract.tex
In this paper, we study a well-posedness problem on a new mathematical model for cancer invasion within the plasminogen activation system, which explicitly incorporates cooperation with host normal cells. Key biological mechanisms—including chemotaxis, haptotaxis, recruitment, logistic growth, and natural degradation of normal cells—along with other primary components (cancer cells, vitronectin, uPA, uPAI-1 and plasmin) are modeled via a continuum framework of cancer cell invasion of the extracellular matrix. The resulting model constitutes a strongly coupled, cross-diffusion hybrid system of differential equations. The primary mathematical challenges arise from the strongly coupled cross-diffusion terms, the parabolic operators of divergence form, and the interaction between the cross-diffusion fluxes and the ODE components. We address these by deriving several a priori estimates for dimensions $\sd\leqslant 3$. Subsequently, we employ a decoupling strategy to split the system into proper sub-problems, establishing the existence (and uniqueness) for each subsystem. Finally, we demonstrate the global existence and uniqueness of the solution for dimensions $\sd\leqslant2$ and the global existence of a solution for dimension $\sd = 3$.

%% file: Introduction.tex
\section{Introduction}
According to the World Health Organization (WHO), cancer is one of the most deadly diseases. Meanwhile, around 90 percent of the death caused by cancer is due to the metastasis, see~\cite{chaffer_perspective_2011}, which refers to spreading process of cancer cells from the primary tumor site to other parts of human body. Hence, it is important to understand the mechanism of metastasis. Before cancer cells move through the blood or lymphatic vessels, they invade into the surrounding connective tissues first. An early critical step in cancer cell invasion is the remodeling of the extracellular matrix of the surrounding tissues by the over-expression of proteolytic enzymes. Andasari--Gerisch--Lolas--South--Chaplain had built up a mathematical model of an interacting system of one type of proteolytic enzymes, known as urokinase-type plasminogen activator (uPA) system~\cite{andasari_mathematical_2011}. The major components in uPA system are cancer cells acting as urokinase plasminogen activator receptors (uPAR), plasmin, uPA, vitronectin (VN, one kind of extracellular matrix protein) and plasminogen activator inhibitor type-1 (PAI-1). Meanwhile, they introduced a diffusion-taxis-reaction model to depict the interactions between those components. For more biological details, one can find them in ~\cite{chaplain_mathematical_2005}. %\footnote{We focus on the cancer cell invasion due to plasminogen activation system.}

The model in this paper is an extension of the diffusion-taxis-reaction model of five components uPA system in \cite{andasari_mathematical_2011} by adding interaction of normal cells with other components. Under this new system, normal cells interact with cancer cells, such as leukocytes, stromal cells~\cite{burger2006cxcr4}, macrophages (tumor-associated macrophages~\cite{pan2020tumor}) and fibroblast (cancer-associated fibroblast~\cite{yang_cancer-associated_2023}), etc. We consider an advection of normal cells towards high concentration of cancer cells~\cite{hanahan_accessories_2012} and an advection of cancer cells towards normal cells~\cite{pittet_clinical_2022}. Normal cells perform haptotaxis\footnote{Haptotaxis is the movement of cancer or normal cells driven by the proteins in ECM.} by moving along a vitronectin adhesion gradient~\cite{carter_haptotaxis_1967}. Although there is lack of experimental evidences on the existence of uPAR on a healthy tissue~\cite{baart2020molecular}, for mathematical modelling symmetry, we assume normal cells perform chemotaxis\footnote{Chemotaxis is the movement of cancer or normal cells driven by chemical molecules} by moving along a concentration gradient of uPA and PAI-1 with rates of $\chi_{22}$ and $\chi_{23}$ with no restriction on whether these two coefficients can be zero. An increment of cancer cells density~\cite{kamdje2020tumor} and a decrement of normal cell density is considered due to contact with normal cells~\cite{gatenby1996reaction}. To describe the mathematical model, we will name the variables of each component distinguished by the sub-index:
\begin{align*}
    u_C(x,t)&=\text{cancer cell density at location $x$ and time $t$};\\
    u_N(x,t)&=\text{normal cell density at location $x$ and time $t$};\\
    u_V(x,t)&=\text{vitronectin concentration at location $x$ and time $t$};\\
    u_A(x,t)&=\text{uPA concentration at location $x$ and time $t$};\\
    u_I(x,t)&=\text{PAI-1 concentration at location $x$ and time $t$};\\
    u_P(x,t)&=\text{plasmin concentration at location $x$ and time $t$}.
\end{align*}

Due to the biological facts above and the five components model in \cite{andasari_mathematical_2011}, we extend the model into the following system of hybrid reaction-diffusion and cross diffusion equations:
\begin{multline}
    \label{equ:1}
    \partial_t u_C = \nabla\cdot(d_C(x,t)\nabla u_C)-\nabla\cdot [B(u_C)\left(\chi_{11}\nabla u_N +\chi_{12}\nabla u_A + \chi_{13}\nabla u_I + \chi_{14}\nabla u_V\right)] 
    \\+ \alpha_{11}u_N u_C + \mu_C u_C (1-\frac{u_C}{K_C})-\delta_C u_C;
\end{multline}
\begin{multline}
    \label{equ:2} 
    \partial_t u_N = \nabla\cdot(d_N(x,t)\nabla u_N)-\nabla\cdot [B(u_N)\left(\chi_{21}\nabla u_C + \chi_{22}\nabla u_A + \chi_{23}\nabla u_I + \chi_{24}\nabla u_V\right)]\\ -\alpha_{21} u_N u_C + \mu_N u_N (1-\frac{u_N}{K_N})-\delta_N u_N ;
\end{multline}
\begin{align}
    \label{equ:3}
    \partial_t u_V& = - \alpha_{31}u_V u_P - \alpha_{32}u_V u_I + \alpha_{33}u_A u_I + \mu_V u_V(1-\frac{u_V}{K_V}); \\ 
    \label{equ:4}
    \partial_t u_A& = \nabla\cdot(d_A(x,t)\nabla u_A) - \alpha_{41}u_A u_I -\alpha_{42}u_A u_C + \mu_A u_C;\\ 
    \label{equ:5}
    \partial_t u_I& = \nabla\cdot(d_I(x,t)\nabla u_I) - \alpha_{51}u_I u_A - \alpha_{52}u_I u_V + \mu_I u_P ;\\
    \label{equ:6}
    \partial_t u_P& = \nabla\cdot(d_P(x,t)\nabla u_P) +\alpha_{61} u_C u_A+ \alpha_{62}u_V u_I - \delta_P u_P,
\end{align}
where $\partial_t$ stands for time partial derivative, $\nabla:=(\frac{\partial}{\partial x_1},\frac{\partial}{\partial x_2},\cdots,\frac{\partial}{\partial x_d})$ stands for $d$-dimensional spatial partial derivative and the function $B(u)$ is a bounded function of $u$ with detailed definition in~\eqref{hypo:bound}.

For the biological meaning of some coefficients, one can find detailed illustration in \cite{andasari_mathematical_2011,chaplain_mathematical_2005}, and we will briefly introduce them as following: 
\begin{enumerate}
    \item $d_i(x,t),\forall i \in \{C,N,A,I,P\}$, stands for the diffusion coefficients.
    \item $\chi_{1i},\chi_{2i},\forall i \in \{1,\cdots, 4\}$, stands for the drifting coefficients, which means absolute drifting effect.
    \item $\alpha_{ij}$ stands for the interaction reaction coefficients between two components.
    \item $\mu_i,\forall i\in \{C,N,V,A,I\}$, stands for the proliferation rate.
    \item $K_i,\forall i\in \{C,N,V\}$, stands for the carrying capacity of cancer cells, normal cells and vitronectin, respectively.
    \item $\delta_i,\forall i\in \{C,N,P\}$, stands for the degradation rate.
\end{enumerate}

To complete the whole mathematical model, we assume the system \eqref{equ:1}-\eqref{equ:6} holds in the spatio-temporal domain, $Q_T:=\{(x,t):x\in \Omega \subset \R^\sd, t\in (0,T)\}$, for any fixed $T>0$, where $\Omega$ is a bounded domain of $\R^\sd$ with $C^2$-boundary $\partial \Omega$ and $\sd\in \N$ is the Euclidean dimension. Also, we assume the model equipping with zero-flux boundary conditions:
\begin{equation}\label{bc}
     \nabla_{\vec{\nu}} u_i(x,t) = 0,\;\forall (x,t)\in\partial \Omega \times (0,T), \forall i \in \{C,N,A,I,P\},
\end{equation}
where $\vec{\nu}$ stands for the unit outward normal vector along the boundary of the domain, $\partial \Omega$; and known-data initial conditions:
\begin{multline}\label{ic}
    (u_C(x,0),u_N(x,0),u_V(x,0),u_A(x,0),u_I(x,0),u_P(x,0))\\
    = (u_{C_0}(x),u_{N_0}(x),u_{V_0}(x),u_{A_0}(x),u_{I_0}(x),u_{P_0}(x)),\;\forall x \in \Omega.
\end{multline}

More hypotheses and assumptions on known data, including coefficients and initial conditions will be stated in Section 2.

In the past decades, the mathematical modeling and analysis for cancer evolution, invasion, and metastasis has attracted many researchers~\cite{chaplain_mathematical_2005,franssen2019mathematical,hillen2024modelling,lowengrub2009nonlinear, suzuki2017mathematical}. We provide a short review on those known well-posedness results for models closely related to our model. Tao, in 2011, studied a 2D model featuring ECM remodeling and a single haptotactic cross-diffusion term, proving conditional global existence and uniqueness~\cite{tao2011global}. Giesselmann--Kolbe--Lukacova-Medvidova--Sfakianakis, in 2018, investigated a 2D haptotaxis model with multiple cell components, also establishing a conditional global existence of classical solutions~\cite{giesselmann2017existence}. Subsequently, Jin, in 2020, studied a similar 2D-haptotaxis model that notably excluded the ECM self-remodeling effect, successfully proving the global existence and uniqueness of a uniformly bounded classical solution~\cite{jin2020global}. Desvillettes--Giunta, in 2021, proved the existence and uniqueness of a global weak solution of a 1D chemotaxis model and a global classical solution under smoothness assumption of initial data~\cite{desvillettes2021existence}. For a more general case, Choquet--Rosier--Rosier, in 2021, established a conditional existence for spatial dimension $\sd\leqslant 3$ and uniqueness for spatial dimension $\sd\leqslant 2$ of a general cross-diffusion system under a non-degenerated self-cross diffusion assumption~\cite{choquet_well_2021}. However, due to non-constant diffusion coefficients and coupled cross-diffusion on multiple components, these existing results do not cover the system~\eqref{equ:1} -~\eqref{equ:6}. This motivates us to consider the global existence and uniqueness problem. A crucial tool used in this paper is the estimate in Morrey-John-Nirenberg-Campanato space, which is developed by Yin in~\cite{yin1992l2}. Various apriori estimates are also the keys for the existence proof.

This paper is organized as following: In section 2, we recall some function spaces which are used in subsequence analysis, illustrated the hypothesis on the known data and state the main results; In section 3, we derive some useful apriori estimates. In section 4, we prove theorem \ref{exist} by two important lemmas. In section 5, we prove~\cref{exist_3d} showing the existence of a solution when dimension $\sd=3$. In section 6, we present the conclusion and discuss some open questions.

Throughout the paper, we shall use $C$ to denote various constants, which is only dependent on the known data and terminal time marker $T$, and whose value might differ at each occurrence. The dependence of $C$ on known data and $T$ will be shown by the end.

%% file: Main_Result.tex
\section{Preliminaries and Main Results}
For reader's convenience, we recall some function spaces which is standard for studying elliptic and parabolic partial differential equations (see~\cite{evans,Ladyzen}).

\input{def}

Before we state the main result, we shall start with the definition of a weak solution.
\begin{definition}\label{weak}
    A non-negative $(u_C,u_N,u_V,u_A,u_I,u_P)$ is called a non-negative weak solution to the system~\eqref{equ:1} -~\eqref{equ:6} on $Q_T:=\Omega\times (0,T)$ with boundary conditions~\eqref{bc} and initial conditions~\eqref{ic} if for any $(u_C,u_N,u_V,u_A,u_I,u_P)$,
    \[(u_C,u_N,u_V,u_A,u_I,u_P)\in (V_2(Q_T)\cap L^3(Q_T))^2 \times {(V_2(Q_T)\cap L^{\infty}(Q_T))}^4;\]
     and for all test functions $\phi\in L^2(0,T;H^1(\Omega))\cap L^3(Q_T)$ with $\phi(x,T)=0$, and denote the drifting vector $\vec{W}_C = \chi_{11} \nabla u_N +\chi_{12}\nabla u_A + \chi_{13}\nabla u_I + \chi_{14}\nabla u_V$ and $\vec{W}_N = \chi_{21}\nabla u_C + \chi_{22}\nabla u_A + \chi_{23}\nabla u_I + \chi_{24}\nabla u_V$ for notation simplicity, it satisfies following integral identities:
    \begin{multline}
    \iint_{Q_T}  - u_C \partial_t\phi \,dxdt -\int_{\Omega}u_{C_0}\phi(x,0)\,dx + \iint_{Q_T} d_C(x,t)\nabla u_C \cdot \nabla \phi \,dxdt \\= \iint_{Q_T}B(u_C)\vec{W}_C  \cdot \nabla\phi \,dxdt
    +\iint_{Q_T}\alpha_{11}u_N u_C \phi \,dxdt + \iint_{Q_T} \mu_C u_C \phi \,dx \\- \iint_{Q_T}\frac{\mu_C}{K_C}u_C^2\phi \,dxdt - \iint_{Q_T}\delta_C u_C\phi \,dxdt
    \end{multline}
    \begin{multline}
        \iint_{Q_T}  -u_N \partial_t \phi \,dx dt -\int_{\Omega}u_{N_0}\phi(x,0)\,dx + \iint_{Q_T} d_N(x,t)\nabla u_N \cdot \nabla \phi \,dxdt \\= \iint_{Q_T} B(u_N)\vec{W}_N  \cdot \nabla\phi \,dxdt
        -\iint_{Q_T}\alpha_{21}u_N u_C \phi \,dxdt + \iint_{Q_T} \mu_N u_N \phi \,dx\\ - \iint_{Q_T}\frac{\mu_N}{K_N}u_N^2\phi \,dxdt - \iint_{Q_T}\delta_N u_N\phi \,dxdt
    \end{multline}
    \begin{multline}
        \iint_{Q_T}- u_V \partial_t \phi \,dx dt -\int_{\Omega}u_{V_0}\phi(x,0)\,dx \\= \iint_{Q_T} (- \alpha_{31}u_V u_P -\alpha_{32} u_V u_I + \alpha_{33}u_A u_I + \mu_V u_V(1-\frac{u_V}{K_V}) )\phi \,dxdt
    \end{multline}
    \begin{multline}
        \iint_{Q_T} -u_A  \partial_t \phi \,dxdt -\int_{\Omega}u_{A_0}\phi(x,0)\,dx+\iint_{Q_T}d_A(x,t)\nabla u_A \cdot \nabla \phi \,dxdt\\
        = \iint_{Q_T}(-\alpha_{41} u_A u_I -\alpha_{42}u_A u_C + \mu_A u_C)\phi \,dxdt
    \end{multline}
    \begin{multline}
        \iint_{Q_T} - u_I \partial_t \phi \,dxdt -\int_{\Omega}u_{I_0}\phi(x,0)\,dx +\iint_{Q_T}d_I(x,t)\nabla u_I \cdot \nabla \phi \,dxdt\\
        = \iint_{Q_T} (- \alpha_{51}u_I u_A - \alpha_{52}u_I u_V + \mu_I u_P)\phi \,dxdt
    \end{multline}
    \begin{multline}
        \iint_{Q_T} - u_P \partial_t\phi \,dxdt -\int_{\Omega}u_{P_0}\phi(x,0)\,dx +\iint_{Q_T}d_P(x,t)\nabla u_P \cdot \nabla \phi \,dxdt\\
        =\iint_{Q_T}(\alpha_{61}u_C u_A+\alpha_{62}u_V u_I - \delta_P u_P)\phi \,dxdt
    \end{multline}
    \begin{remark*}
        For spatial dimension $\sd\leqslant 2$, one can have weak solution, $(u_C,u_N,u_V,u_A,u_I,u_P)\in {(V_2(Q_T)\cap L^{\infty}(Q_T))}^6$ in the same weak solution form with test function $\phi\in C(0,T;H^1(\Omega))\cap H^1(0,T;L^2(\Omega))$ with $\phi(x,T)=0$.
    \end{remark*}    
\end{definition}

Here we make some assumptions for coefficients and known data:
\begin{enumerate}[label= (H\arabic*),ref=H\arabic*]
    \item\label{h1}\label{hypo:diffusion} Diffusion coefficient, $d_i(x,t)$ for $i\in\{C,N,A,I,P\}$, follows the ellipticity conditions, i.e. there exists some positive $d_i^{(0)},d_i^{(1)}\in \R^{+}$,
    \[d_i^{(0)} \leqslant d_i(x,t)\leqslant d_i^{(1)},\;\forall (x,t)\in Q_T, i\in\{C,N,A,I,P\}.\]
    \item\label{h2}\label{hypo:init} The initial conditions are assumed to be non-negative, i.e.
     \[ u_{i_0}(x)\geqslant 0,\;\forall x\in \Omega,\; i \in \{C,N,V,A,I,P\}\]
     and moreover,
     \[u_{i_0}\in L^{\infty}(\Omega),\,\forall i \in \{C,N,A,I,P\}\]
     \[u_{V_0}\in \{u| \|u\|_{C^{\alpha}(\bar{\Omega})}<\infty\,\text{and}\, \|\nabla u_{V_0}\|_{L^{2,n-2+2\alpha}_C(\Omega)}<\infty\} \,\text{with}\, \nabla_{\vec{\nu}}u_{V_0}(x) = 0,\,\text{for}\,x\in \partial \Omega,\] for some $\alpha \in (0,1)$.
     \item\label{h3}\label{hypo:bound} $B(u)$ is a bounded function, motivated by~\cite{choquet_well_2021}, satisfying the following definition:
        \[B(u):= \max\{0,\min\{u,1\}\},\,\text{for}\, u\in\{u_C,u_N\}.\]
        \begin{remark*}
           An example one can think of is a saturation behavior of population density:
           \[\hat{B}(u)=\frac{u}{1+u}.\]
        \end{remark*}
     \item\label{h4}\label{hypo:reaction} The constants $\alpha_{11}, \alpha_{21},\mu_C,K_C$ satisfy the following inequality:
     \[\alpha_{11}^2 < 4\alpha_{21} \frac{\mu_C}{K_C}.\]
\end{enumerate}

Now, we state the following existence and uniqueness theorem:
\begin{thm}\label{exist}
   (Global Existence and Uniqueness of the Weak Solution $\sd\leqslant 2$) 
For arbitrary time $T>0$, spatial dimension $\sd\leqslant 2$, under~\eqref{h1} -~\eqref{h4}, there exists a unique non-negative weak solution of the form, definition~\eqref{weak}, for the system~\eqref{equ:1} -~\eqref{equ:6}, in ${(V_2(Q_T)\cap L^{\infty}(Q_T))}^6$, for some constants $\chi_{11},\,\chi_{21}\in \R$.
\end{thm}

\begin{thm}\label{exist_3d}
(Global Existence of the weaker solution $\sd =3$) For arbitrary time $T>0$, spatial dimension $\sd=3$, under~\eqref{h1} -~\eqref{h4}, there exists a non-negative weaker solution of the form, definition~\eqref{weak}, for the system~\eqref{equ:1} -~\eqref{equ:6}, in $ (V_2(Q_T)\cap L^3(Q_T))^2 \times {(V_2(Q_T)\cap L^{\infty}(Q_T))}^4$, for some constants $\chi_{11},\,\chi_{21}\in \R$, satisfying $(\chi_{11}+\chi_{21})^2< 4 d_C^{(0)}d_N^{(0)}$.
\end{thm}

%% file: def.tex
For $T=\infty$, we define $Q_T:=\Omega\times(0,T)$, as $Q := \Omega \times (0,\infty)$. And $\bar{Q}_T$ means the closure of $Q_T$.

We recall several Banach spaces which will be frequently used later:

For $p\in [1,\infty)\cup \{\infty\}$ and some subset $D$ of $\Omega$ or $Q_T$, $L^p(D)$ is a Banach space of measurable functions on $D$, where its norm is defined as: for all $u\in L^p(D)$,
\[\|u\|_{L^p(D)} =\begin{cases}\displaystyle
    \big(\int_{D}u^p dx \big)^{\frac{1}{p}},\;\forall p \in[1,\infty);\\
    ess.\sup_{(x,t)\in D}\{|u(x,t)|\},\; p = \infty,
\end{cases} \]
$L^p(D):=\{u:\|u\|_{L^p(D)}< \infty\}$. When $p=2$, another notation, $H(D)$ will be used for $L^2(D)$.

We denote $H^1(\Omega)$ as the Sobolev space, $W^{1,2}(\Omega)$, whose norm is defined as:
\[\|u\|_{H^1(\Omega)} = (\int_{\Omega} |\nabla u|^2 dx)^{\frac{1}{2}}.\]
We also denote its dual space as $H^*(\Omega)$ and $<\cdot,\cdot>_H$ as the dual pairing between $H^*(\Omega)$ and $H^1(\Omega)$ and define its norm as:
\[\|u\|_{H^*(\Omega)}:=\sup\{<u,f>_H:f\in H^1(\Omega),\|f\|_{H^1(\Omega)}\leqslant 1\}.\]

$C^{\alpha,\frac{\alpha}{2}}(Q_T)$ with $\alpha\in (0,1)$ is the $\alpha$-H\"older continuous function space on $Q_T$, where the H\"older seminorm is defined as following:
\[[u]_{\alpha}:= \sup_{\substack{(x_1,t_1),(x_2,t_2)\in Q_T \\ (x_1,t_1)\neq(x_2,t_2)}}\frac{|u(x_1,t_1) -u(x_2,t_2)|}{(x_1-x_2)^{\alpha}+(t_1-t_2)^{\frac{\alpha}{2}}}\]
which is also a Banach space equipped with the norm as: 
\[\|u\|_{C^{\alpha,\frac{\alpha}{2}}(Q_T)} = \|u\|_{L^{\infty}(Q_T)} + [u]_{\alpha}, \;\forall u \in C^{\alpha,\frac{\alpha}{2}}(Q_T).\]

$L^2(0,T;H^1(\Omega))$ is a Banach space of measurable mappings $u:(0,T)\rightarrow H^1(\Omega)$, where its norm is defined as: $\forall u \in {L^2(0,T;H^1(\Omega))}$,
\[\|u\|_{L^2(0,T;H^1(\Omega))} = \big\| \|u(\cdot,t)\|_{H^1(\Omega)}\big\|_{L^2(0,T)} = \Big(\int_0^T \|u(\cdot,t)\|^2_{H^1(\Omega)}dt\Big)^{\frac{1}{2}}.\] 

$V_2(Q_T)$ has the same definition as it in~\cite{Ladyzen}, which is a Banach space consists of all elements of $L^2(0,T;H^1(\Omega))$ with finite norm:\[\|u\|_{V_2(Q_T)}:=\sup_{0< t <T}\|u\|_{L^2(\Omega)} + \|\nabla u\|_{L^2(Q_T)},\;\forall u\in V_2(Q_T).\]

We define the distance between two points $z_1 := (x_1,t_1)$ and $z_2 :=(x_2,t_2)$ in $Q_T$ by
\[d(z_1,z_2) := \max\{|x_1-x_2|,|t_1-t_2|^{\frac{1}{2}}\}.\]

Denote $B_r(x_0) := \{x\in \R^n:|x-x_0|<r\}$ and $Q_r(x_0,t_0) := B_r(x_0)\times (t_0-r^2,t_0].$

For a measurable set $A\subset \R^n\times[0,T]$
\[\fint_{A} u\, dxdt: = \frac{1}{|A|} \int_{A} u\, dxdt\]

In particular, when $A = {Q_T\cap Q_r(x_0,t_0)}$,
\[[u]_{(x_0,t_0)}:= \fint_{Q_T\cap Q_r(x_0,t_0)} u\, dxdt\] which is measurable and whose measure is denoted by $|A|$ and for $\mu>0$,
\[[u]^2_{2,\mu,Q_T}:= \sup_{(x_0,t_0)\in Q_T,r>0}r^{-\mu}\int_{A}\left|u-[u]_{(x_0,t_0)}\right|^2 dxdt\]

The Camponato spaces, denoted $\Camp{2}{\mu}$, is the space consisting of all functions in $L^2(Q_T)$ such that a seminorm $[u]_{2,\mu,Q_T}<\infty$ and its norm is define by 
\[\|u\|_{\Camp{2}{\mu}} = \left(\|u\|^2_{L^2(Q_T)}+ [u]^2_{2,\mu,Q_T}\right)^{\frac{1}{2}}.\]
\begin{remark*}
    A well-known fact is that Campanato space is equivalent to Morrey space if $0\leqslant \mu < d+2$. We denote $\Mor{2}{\mu}$ for Morrey space with its corresponding norm denoted as $\|\cdot\|_{\Mor{2}{\mu}}$.
\end{remark*}

For notation simplicity, we denote the product space by the exponential fashion, e.g. suppose $X$ is a Banach space, then $X^p$ is a $p$ dimensional product space.

%% file: Estimate.tex
\section{Apriori Estimate}
In this section, we shall show some useful apriori estimates of the system~\eqref{equ:1} -~\eqref{equ:6}. For the notation simplicity on both claiming and proof process, we denote an index set, $\I$, as a collection of subindex for each variable, i.e. $\I:=\{C,N,V,A,I,P\}$.

\begin{lem}\label{ap:pos} (Non-negativity) Under the assumption~\eqref{h1}-\eqref{h2}, a solution $(u_i)_{i\in\I}$ to the system~\eqref{equ:1}-~\eqref{equ:6} must have a non-negative apriori estimate: $\forall i \in \I$,
    \[u_i(x,t)\geqslant 0,\; \forall (x,t)\in Q:=\Omega\times (0,\infty).\]  
\end{lem}
\begin{proof}
    One can apply the maximum principle to each equation of the system~\eqref{equ:1} -~\eqref{equ:6} by assuming at least one component reaching a zero minimum in $Q_T$ and leading to a contradiction. The proof will be similar to~\cite{lou_global_1998, pao, protter2012maximum, yin_cross-diffusion_2017}.
\end{proof}

\begin{lem}\label{bc:v}
    (Induced Boundary Condition) With boundary conditions~\eqref{bc}, equation~\eqref{equ:3}, and under assumption~\eqref{hypo:init},
    \[\nabla_{\vec{\nu}} u_V(x,t) = 0,\,\forall (x,t) \in \partial \Omega \times (0,T),\, \text{in the sense of trace}.\]
\end{lem}
\begin{proof}
    Since $\partial \Omega$ is assumed to be at least $C^2$, we may use the $C^2$-function to approximate a $V_2(Q_T)$ function. Without loss of generality, we may assume $u_i\in C^2(\bar{Q}_T)$ in the classical sense. Taking gradient on both sides of equation~\eqref{equ:3} and dot product with unit normal outward vector $\nu$, we have the following equation: $\forall (x,t) \in \partial \Omega \times (0,T),$
    \begin{flalign}\label{uv:boundary}
        \partial_t \nd u_V =& -\alpha_{31}u_P \nabla_{\nu} u_V-\alpha_{31}u_V \nabla_{\nu} u_P-\alpha_{32}u_I \nabla_{\nu} u_V-\alpha_{32}u_V \nabla_{\nu} u_I \notag\\
        &+ \alpha_{33} u_A\nd u_I + \alpha_{33} u_I\nd u_A + \mu_V \nd u_V - 2 \frac{\mu_V}{K_V}u_V \nd u_V \notag \\
        =& -\alpha_{31}u_P \nabla_{\nu} u_V -\alpha_{32}u_I \nabla_{\nu} u_V + \mu_V \nd u_V - 2 \frac{\mu_V}{K_V}u_V \nd u_V
    \end{flalign}
    where the last equation is due to boundary condition~\eqref{bc}.

    Denote $y(t): = \nd u_V(\cdot,t),\,\forall t \in (0,T)$, then equation~\eqref{uv:boundary} can be rewritten as a linear ordinary differential equation of $y(t)$:
    \begin{flalign}
        &\frac{d}{dt} y = (\mu_V -\alpha_{31}u_P - \alpha_{32}u_I -2\frac{\mu_V}{K_V}u_V) y, \notag\\
        \implies& y e^{-\int_{0}^t(\mu_V -\alpha_{31}u_P - \alpha_{32}u_I -2\frac{\mu_V}{K_V}u_V) ds}= y(0) = 0\notag\\
        \implies& y(t) = 0,\,\forall t\in (0,T). \notag
    \end{flalign}

    This completes the proof.
\end{proof}

\begin{lem} (Theorem 1, 2 in \cite{yin1992l2})\label{camponato}
    Let $Q_T = \Omega \times (0,T)$ and the functions $f_0(x,t)\in \Camp{2}{(\mu-2)^+}$ and $f_i(x,t)\in \Camp{2}{\mu}$ for $i=1,2,\cdots,n$. Let $u(x,t)$ be the weak solution of the following euqtion:
    \begin{equation}\label{camp:grad}
    u_t -\sum_{i,j = 1}^n\frac{\partial}{\partial x_i}(a_{ij}(x,t)u_{x_j}) = \sum_{i = 1}^n\frac{\partial}{\partial x_i}f_i + f_0,
    \end{equation}
    with Neumann's type boundary condition 
    \[\nabla_{\vec{\nu}} u = 0,\;\forall (x,t)\in \partial \Omega \times (0,T), \]then
    \[ \| \nabla u \|_{\Camp{2}{\mu}}^2\leqslant C \left( \|u\|_{L^2(Q_T)}^2 +\|f_0\|_{\Camp{2}{(\mu-2)^+}}^2 + \sum_{i=1}^n \|f_i\|_{\Camp{2}{\mu}}^2 \right)\]
    where $\mu<\mu_0 = \sd +2\delta_0, \text{ for some }\delta_0\in(0,1)$. Moreover, by the imbedding, it yields 
    \[u\in \Camp{2}{\mu+2}.\]
    In particular, if $\mu > \sd$, then 
    \[u\in C^{\alpha,\frac{\alpha}{2}}(\bar{Q}_T),\text{ where }\alpha = \frac{\mu-\sd}{2}.\]
\end{lem}
\begin{proof}
    Proof can be found in~\cite{yin1992l2}.
\end{proof}

\begin{lem}\label{ap:v2} 
    (Apriori Estimate) Under the assumption~\eqref{h1}-\eqref{h4} and for some $\chi_{11},\,\chi_{21}$, satisfying $(\chi_{11}+\chi_{21})^2<4d_C^{(0)}d_N^{(0)}$, for any fixed $T>0$, dimension $\sd\leqslant 3$, a solution $(u_i)_{i\in\I}$ to the system~\eqref{equ:1}-\eqref{equ:6} must satisfy the following a priori estimate: 
    \[\|u_i\|_{V_2(Q_T)} + \|u_i\|_{L^{\infty}(Q_T)}\leqslant C,\; \forall i \in \I,\]
    \[\|\partial_t u_i\|_{L^2(0,T;H^*(\Omega))}\leqslant C,\; \forall i \in \I,\]
    moreover, 
    \[\|u_i\|_{L^{\infty}(Q_T)}\leqslant C, \forall i\in\I\backslash\{C,N\},\]
    where $C$ depends only on $T$ and known data.
\end{lem}
\input{v2.tex}

\begin{lem}\label{holderlemma}
    Let $Q_T:= \Omega \times (0,T)$ is a bounded domain. Assume $f,g\in C^{\alpha,\frac{\alpha}{2}}(\bar{Q}_T)$ and $u_0\in C^{\alpha}(\bar{\Omega})$, then a (weak) solution, $u \in L^{\infty}(Q_T)$,  to the equation:
    \begin{align}
        &\partial_t u = a_1 f u + a_2 g - a_3 u^2 \label{holdereq}\\
        &u(x,0) = u_0(x),\;\forall x \in \Omega
    \end{align} 
    is an $\alpha$-H\"older continuous function on $\bar{Q}_{T}:=\bar{\Omega} \times [0,T]$, where $a_i$ are positive constant for $i=1,2,3$.
\end{lem}
\begin{proof}
Proof can be found in Appendix.    
\end{proof}
\begin{remark*}
    One important technique, bootstrap argument, can be found in~\cref{bootstrap:arg} of the proof of~\cref{holderlemma}.
\end{remark*}
\begin{lem}\label{max:ucun}
    Under the assumption~\eqref{h1}-\eqref{h4}, for any fixed $T>0$, dimension $\sd \leqslant 2$ and some constants $\chi_{11},\chi_{21}\in \R$, a solution $(u_i)_{i\in\{C,N\}}$ to the system~\eqref{equ:1}-\eqref{equ:6} with boundary condition~\eqref{bc} and initial condition~\eqref{ic} must satisfy the following a priori estimate: 
    \[\|u_i\|_{L^{\infty}(Q_T)}\leqslant C,\; \forall i \in \{C,N\},\]
    where $C$ depends only on $T$ and known data. 
\end{lem}
\begin{proof}
By~\cref{camponato}, we have the following estimates:
\begin{align}
     \| \nabla u_C \|_{\Camp{2}{\mu}}^2 &\leqslant \tilde{C}_1 ( \|u_C\|_{L^2(Q_T)}^2+ \|\chi_{11}B(u_C)\nabla u_N\|_{\Camp{2}{\mu}}^2 + \|\chi_{12}B(u_C)\nabla u_A\|_{\Camp{2}{\mu}}^2 \notag \\
     &+\|\chi_{13}B(u_C)\nabla u_I\|_{\Camp{2}{\mu}}^2 + \|\chi_{14}B(u_C)\nabla u_V\|_{\Camp{2}{\mu}}^2+ \|\alpha_{11}u_Nu_C\|^2_{\Camp{2}{(\mu-2)^+}}   \notag\\
     & + \|\frac{\mu_C}{K_C}u_C^2\|_{\Camp{2}{(\mu-2)^+}}^2+\|(\mu_C-\delta_C)u_C\|_{\Camp{2}{(\mu-2)^+}}^2) \label{camp:uc}\\
      \| \nabla u_N \|_{\Camp{2}{\mu}}^2 &\leqslant \tilde{C}_2 ( \|u_N\|_{L^2(Q_T)}^2+ \|\chi_{21}B(u_N)\nabla u_C\|_{\Camp{2}{\mu}}^2 + \|\chi_{22}B(u_N)\nabla u_A\|_{\Camp{2}{\mu}}^2 \notag\\
      &+\|\chi_{23}B(u_N)\nabla u_I\|_{\Camp{2}{\mu}}^2 + \|\chi_{24}B(u_N)\nabla u_V\|_{\Camp{2}{\mu}}^2+ \|\alpha_{21}u_Nu_C\|^2_{\Camp{2}{(\mu-2)^+}}  \notag\\
     &  + \|\frac{\mu_N}{K_N}u_N^2\|_{\Camp{2}{(\mu-2)^+}}^2+\|(\mu_N-\delta_N)u_N\|_{\Camp{2}{(\mu-2)^+}}^2) \label{camp:un}\\
     \|\nabla u_A\|_{\Camp{2}{\mu}}^2 &\leqslant C ( \|u_A\|_{L^2(Q_T)}^2+ \|u_Au_I\|^2_{\Camp{2}{(\mu-2)^+}}+\|u_Au_C\|^2_{\Camp{2}{(\mu-2)^+}} \notag \label{camp:ua}\\
     & +\|u_C\|_{\Camp{2}{(\mu-2)^+}}^2)\\
     \|\nabla u_I\|_{\Camp{2}{\mu}}^2 &\leqslant C ( \|u_I\|_{L^2(Q_T)}^2+ \|u_Iu_A\|^2_{\Camp{2}{(\mu-2)^+}}+\|u_Iu_V\|^2_{\Camp{2}{(\mu-2)^+}}\notag\\
     & + \|u_P\|_{\Camp{2}{(\mu-2)^+}}^2)\\
     \|\nabla u_P\|_{\Camp{2}{\mu}}^2 &\leqslant C ( \|u_P\|_{L^2(Q_T)}^2+ \|u_Cu_A\|^2_{\Camp{2}{(\mu-2)^+}}+\|u_Vu_I\|^2_{\Camp{2}{(\mu-2)^+}}\notag\\
     & + \|u_P\|_{\Camp{2}{(\mu-2)^+}}^2),
\end{align}
where $\tilde{C}_1$, $\tilde{C}_2$ and $C$ are constants obtained from~\cref{camponato}.
   
    Again, we shall show the proof of the case, dimension $\sd=2$ only, since, for $\sd=1$, it will follow the same process.
    
    With estimates in lemma~\ref{ap:v2}, for $\alpha \in (0,1)$,
    we have the following estimate:
        \begin{align}
            \|\nabla u_A\|_{\Camp{2}{2}}^2 &\leqslant C ( \|u_A\|_{L^2(Q_T)}^2+ \|u_Au_I\|^2_{\Camp{2}{0}}+\|u_Au_C\|^2_{\Camp{2}{0}} +\|u_C\|_{\Camp{2}{0}}^2) \notag\\
            &\leqslant C ( \|u_A\|_{L^2(Q_T)}^2+ \|u_Au_I\|^2_{L^2(Q_T)}+\|u_Au_C\|^2_{L^2(Q_T)} +\|u_C\|^2_{L^2(Q_T)})\leqslant C, \label{grad:ua:2}\\
            \|\nabla u_A\|_{\Camp{2}{2+2\alpha}}^2 &\leqslant C ( \|u_A\|_{L^2(Q_T)}^2+ \|u_Au_I\|^2_{\Camp{2}{2\alpha}}+\|u_Au_C\|^2_{\Camp{2}{2\alpha}} +\|u_C\|_{\Camp{2}{2\alpha}}^2) \notag\\
            &\leqslant C ( \|u_A\|_{L^2(Q_T)}^2+ \|u_Au_I\|^2_{\Mor{2}{2}}+\|u_Au_C\|^2_{\Mor{2}{2}} +\|u_C\|_{\Mor{2}{2}}^2) \notag\\
            &\leqslant C ( \|u_A\|_{L^2(Q_T)}^2+ \|u_A\|^2_{\Camp{2}{2}}+\|u_C\|^2_{\Camp{2}{2}}) \notag \\
            &\leqslant C ( \|u_A\|_{L^2(Q_T)}^2+ \| u_A\|^2_{V_2(Q_T)}+\|u_C\|^2_{V_2(Q_T)}) \notag \\
            &\leqslant C,\label{grad:ua:2+}
        \end{align}
        where the second to the last inequality is due to~\eqref{camp:embed:2d}.

        Similarly, we would have the following estimates: 
        \begin{align}
            \|\nabla u_I\|_{\Camp{2}{2+2\alpha}}&\leqslant C,\label{grad:ui:2+}\\
             \|\nabla u_P\|_{\Camp{2}{2+2\alpha}}&\leqslant C,\label{grad:up:2+}
        \end{align}
        which implies $u_A,u_I,u_P\in C^{\alpha,\frac{\alpha}{2}}(\bar{Q}_T)$. 

        Now, we claim a H\"older estimate on equation \eqref{equ:3} for $u_V$. By previous H\"older estimates on $u_A,u_I,u_P$, we could have $\|\alpha_{33}u_A u_I\|_{C^{\alpha,\frac{\alpha}{2}}(\bar{Q}_T)}\leqslant C$, $\|u_P\|_{C^{\alpha,\frac{\alpha}{2}}(\bar{Q}_T)}\leqslant C$ and $\|u_I\|_{C^{\alpha,\frac{\alpha}{2}}(\bar{Q}_T)}\leqslant C$.  By lemma \ref{holderlemma}, we have $\|u_V\|_{C^{\alpha,\frac{\alpha}{2}}(\bar{Q}_T)}$.

        Furthermore, we claim that $\|\nabla u_V\|_{\Camp{2}{2+2\alpha}}\leqslant C$:
        \begin{claimproof}
            Take gradient on equation~\eqref{equ:3}, we obtain the following:
            \begin{align}
                \partial_t \nabla u_V &= - (\alpha_{31}u_P +\alpha_{32}u_I - \mu_V+ \frac{2\mu_V}{K_V} u_V)\nabla u_V  - \alpha_{31}u_V \nabla u_P + \alpha_{33}u_I \nabla u_A + (\alpha_{33} u_A- \alpha_{32}u_V) \nabla u_I ;
            \end{align}
            This implies:
            \begin{multline}
            \nabla u_V = \nabla u_{V_0} e^{\int_{0}^t - (\alpha_{31}u_P +\alpha_{32}u_I - \mu_V+ \frac{2\mu_V}{K_V} u_V) d\tau} + e^{\int_{0}^t - (\alpha_{31}u_P +\alpha_{32}u_I - \mu_V+ \frac{2\mu_V}{K_V} u_V) d\tau}\\ \int_{0}^{t} e^{\int_{0}^{s}(\alpha_{31}u_P +\alpha_{32}u_I - \mu_V+ \frac{2\mu_V}{K_V} u_V) d\tau}(- \alpha_{31}u_V \nabla u_P + \alpha_{33}u_I \nabla u_A + (\alpha_{33} u_A- \alpha_{32}u_V) \nabla u_I )ds\label{uv:soln}
            \end{multline} 
         Therefore, apply $\Camp{2}{2+2\alpha}$ estimate on both side on~\eqref{uv:soln} by estimates~\eqref{grad:ua:2+} -~\eqref{grad:up:2+} and~\eqref{hypo:init} and estimates in lemma~\ref{ap:v2}, we obtain:

        \begin{align}
            \|\nabla u_V\|_{\Camp{2}{2+2\alpha}} &\leqslant C \|\nabla u_{V_0}\|_{L^{{2},{2\alpha}}_C(\Omega) }+ C \|\nabla u_P\|_{\Camp{2}{2+2\alpha}} + C\|\nabla u_A\|_{\Camp{2}{2+2\alpha}} + C\|\nabla u_I\|_{\Camp{2}{2+2\alpha}},\notag
        \end{align}
        which implies
        \begin{equation}
             \|\nabla u_V\|_{\Camp{2}{2+2\alpha}} \leqslant C.
        \end{equation}
        
        \end{claimproof}
        Lastly, for estimate~\eqref{camp:uc} and~\eqref{camp:un}, when $\mu = 2$, by estimates~\eqref{grad:ua:2},~\eqref{grad:ui:2+},~\eqref{grad:up:2+} and~\eqref{hypo:bound},
        \begin{flalign}
            \| \nabla u_C \|_{\Camp{2}{2}}^2 &\leqslant \tilde{C}_1 ( \|u_C\|_{L^2(Q_T)}^2+ \|\chi_{11}B(u_C)\nabla u_N\|_{\Camp{2}{2}}^2 + \|\chi_{12}B(u_C)\nabla u_A\|_{\Camp{2}{2}}^2 \notag \\
         &+\|\chi_{13}B(u_C)\nabla u_I\|_{\Camp{2}{2}}^2 + \|\chi_{14}B(u_C)\nabla u_V\|_{\Camp{2}{2}}^2 + \|\alpha_{11}u_Nu_C\|^2_{\Camp{2}{0}}  \notag\\
        & + \|\frac{\mu_C}{K_C}u_C^2\|_{\Camp{2}{0}}^2)+\|(\mu_C-\delta_C)u_C^2\|_{\Camp{2}{0}}^2)\notag\\
        &\leqslant  \tilde{C}_1 \chi_{11}^2 \|\nabla u_N\|_{\Camp{2}{2}}^2+ C( \|u_C\|_{L^2(Q_T)}^2+ \|\nabla u_A\|_{\Camp{2}{2}}^2 +\|\nabla u_I\|_{\Camp{2}{2}}^2\notag \\
         & + \|\nabla u_V\|_{\Camp{2}{2}}^2  + \|u_Nu_C\|^2_{L^2(Q_T)} + \|u_C^2\|^2_{L^2(Q_T)})\notag\\
        &\leqslant \tilde{C}_1\chi_{11}^2\|\nabla u_N\|_{\Camp{2}{2}}^2+C, \label{ucmu} &&
        \end{flalign}
        where the last equality is due to Gagliardo-Nirenberg inequality which implies $V_2(Q_T) \hookrightarrow L^4(Q_T)$ when $\sd\leqslant 2$. And similarly,
        \begin{equation}
            \| \nabla u_N \|_{\Camp{2}{2}}^2 \leqslant \tilde{C}_2\chi_{21}^2\|\nabla u_C\|_{\Camp{2}{2}}^2+C.\label{unmu}
        \end{equation}
        Therefore, there exists $\chi_{11}$ and $\chi_{21}$ such that $\tilde{C}_1 \chi_{11}^2< 1$ and $\tilde{C}_2 \chi_{12}^2< 1$, then summing up estimate \eqref{ucmu} and \eqref{unmu} yields:
            \begin{equation}
                \| \nabla u_C \|_{\Camp{2}{2}}^2 +\| \nabla u_N \|_{\Camp{2}{2}}^2 \leqslant C,\,\label{camp:small}
            \end{equation}
        and implies,
            \begin{align}
                \|u_C\|_{\Camp{2}{4}}&\leqslant C,\\
                \|u_N\|_{\Camp{2}{4}}&\leqslant C
            \end{align}
        Then we have
            \begin{flalign}
                \| \nabla u_C \|_{\Camp{2}{2+2\alpha}}^2 &\leqslant \tilde{C}_1 ( \|u_C\|_{L^2(Q_T)}^2+ \|\chi_{11}B(u_C)\nabla u_N\|_{\Camp{2}{2+2\alpha}}^2 + \|\chi_{12}B(u_C)\nabla u_A\|_{\Camp{2}{2+2\alpha}}^2 \notag && \\
                &+\|\chi_{13}B(u_C)\nabla u_I\|_{\Camp{2}{2+2\alpha}}^2 + \|\chi_{14}B(u_C)\nabla u_V\|_{\Camp{2}{2+2\alpha}}^2   \notag\\
                &+ \|u_Nu_C\|^2_{\Camp{2}{2\alpha}} + \|u_C^2\|_{\Camp{2}{2\alpha}}^2)\notag\\
                &\leqslant \tilde{C}_1 \chi_{11}^2\|\nabla u_N\|_{\Mor{2}{2+2\alpha}}^2 + C ( \|u_C\|_{L^2(Q_T)}^2+  \|\nabla u_A\|_{\Mor{2}{2+2\alpha}}^2 \notag \\
                &+\|\nabla u_I\|_{\Mor{2}{2+2\alpha}}^2 + \|\nabla u_V\|_{\Mor{2}{2+2\alpha}}^2 + \underbrace{\|u_Nu_C\|^2_{\Mor{2}{2\alpha}}}_{(\ref{epsuc}.1)} + \underbrace{\|u_C^2\|_{\Mor{2}{2\alpha}}^2)}_{(\ref{epsuc}.2)}\label{epsuc}
        \end{flalign}
        and
    \begin{flalign}
        (\ref{epsuc}.1) &\leqslant C (\|u_N\|^2_{\Mor{4}{2\alpha} } +\|u_C\|^2_{\Mor{4}{2\alpha} })^{2}\notag &&\\
        &\leqslant C (\|\nabla u_N\|_{\Mor{2}{2\alpha}} +\|u_N\|_{\Mor{2}{2\alpha}}+ \|\nabla u_C\|_{\Mor{2}{2\alpha}}+\|u_C\|_{\Mor{2}{2\alpha}})^2 \leqslant C  \notag    
    \end{flalign}
    where the second to the last inequality is due to Gagliardo-Nirenberg inequality with dimension $\sd\leqslant 2$, and the last inequality holds due to the estimate~\eqref{camp:small}. Similarly,
    \begin{equation}
        (\ref{epsuc}.2) \leqslant C (\|\nabla u_C\|_{\Mor{2}{2\alpha}}+\|u_C\|_{\Mor{2}{2\alpha}})^2\leqslant C.
    \end{equation}

    Therefore, repeating the same process on~\eqref{camp:un}, we obtain a similar estimate as~\eqref{camp:small}:   
    \begin{equation}\label{sum:camp}
        \|\nabla u_C\|^2_{\Camp{2}{2+2\alpha}}+ \|\nabla u_N\|^2_{\Camp{2}{2+2\alpha}}\leqslant C,
    \end{equation}

    In particular, we obtain,
     \begin{equation}
     \|u_C\|_{\Camp{2}{4+2\alpha}}+ \|u_N\|_{\Camp{2}{4+2\alpha}} \leqslant C,
     \end{equation}
     which implies 
     \begin{equation}
        \|u_C\|_{C^{\alpha,\frac{\alpha}{2}}(\bar{Q}_T)}+ \|u_N\|_{C^{\alpha,\frac{\alpha}{2}}(\bar{Q}_T)} \leqslant C.
     \end{equation}
     
     This completes the proof.
\end{proof}

%% file: v2.tex
\begin{proof}
    First, by lemma~\ref{ap:pos}, we have already known the non-negativity of a solution of the system~\eqref{equ:1}-~\eqref{equ:6}. Since we are deriving apriori estimate, we may assume that the solution of the system~\eqref{equ:1}-\eqref{equ:6},  $(u_i)_{i\in\I} \in C^{0,0}(\bar{Q}_T)\cap C^{2,1}(Q_T)$.

    \textbf{Part 1: To prove} $\mathbf{\|u_A\|_{V_2(Q_T)}\leqslant C}$, we claim that:
    \begin{claim}
        \begin{align}
            &\bullet\; \|u_N u_C\|_{L^1(Q_T)} \leqslant C;\label{nc:l1}\\
            &\bullet\; \|u_C\|_{L^1(\Omega)}\leqslant C;\label{c:l1s}\\
            &\bullet\; \|u_A u_C\|_{L^1(Q_T)}\leqslant C.\label{ac:l1}\\
            &\bullet\; \|u_C\|_{L^2(Q_T)}\leqslant C.\label{uc:l2}\\
            &\bullet\; \|u_N\|_{L^2(Q_T)}\leqslant C.\label{un:l2}
        \end{align}
    \end{claim}
    \begin{claimproof}
        Applying $L^1(\Omega)$ estimate to equation~\eqref{equ:2}, we obtain the following:
        \[\frac{d}{dt}\int_{\Omega}u_N dx + \int_{\Omega}\alpha_{21}u_N u_C +\frac{\mu_N}{K_N} u_N^2 dx = ( \mu_N - \delta_N)\int_{\Omega}  u_N dx.\]

        Then multiplying integrating factor, $e^{-(\mu_N-\delta_N)t}$, and integrating over time, we obtain the following inequalities:
        \begin{multline*}
            \sup_{0<t<T} \int_{\Omega}u_N dx + \alpha_{21}\min\{1,e^{(\mu_N-\delta_N)T}\}\iint_{Q_T}u_N u_C dx dt\\ + \frac{\mu_N}{K_N}\min\{1,e^{(\mu_N-\delta_N)T}\}\iint_{Q_T}u_N^2dxdt\\ \leqslant \max\{1,e^{-(\mu_N-\delta_N)T}\}\int_{\Omega}u_{N_0}dx,
        \end{multline*}
        which implies
        \begin{align}
            &\|u_N\|_{L^1(\Omega)}\leqslant\sup_{0<t<T} \int_{\Omega}u_N dx\leqslant C;\label{un:l1}\\
            &\|u_N u_C\|_{L^1(Q_T)} \leqslant C; \notag\\
            &\|u_N\|_{L^2(Q_T)} \leqslant C.\notag
        \end{align}

        Hence, we obtain estimate~\eqref{nc:l1} and~\eqref{un:l2}.

        Applying $L^1(\Omega)$ estimate to equation~\eqref{equ:1} with the same process as above,  we obtain the following:
        \begin{multline*}
            \sup_{0<t<T} \int_{\Omega}u_C dx  + \frac{\mu_C}{K_C}\min\{1,e^{(\mu_C-\delta_C)T}\}\iint_{Q_T}u_C^2dxdt\\ \leqslant \alpha_{11}\max\{1,e^{(\mu_C-\delta_C)T}\}\iint_{Q_T}u_N u_C dx dt+\max\{1,e^{-(\mu_C-\delta_C)T}\}\int_{\Omega}u_{C_0}dx,
        \end{multline*}
        then using estimate~\eqref{nc:l1}, it implies
        \begin{align}
            &\|u_C\|_{L^1(\Omega)} \leqslant \sup_{0<t<T}\int_{\Omega}u_C dx\leqslant C;\notag\\ % label{uc:l1}
            &\|u_C\|_{L^2(Q_T)}\leqslant C.\notag
        \end{align}

        Hence, we obtain estimate~\eqref{c:l1s} and~\eqref{uc:l2}.

        Applying $L^1(Q_T)$ estimate on equation~\eqref{equ:4} and with estimate~\eqref{c:l1s}, we obtain the following:

        \[\frac{d}{dt}\int_{\Omega}u_A dx + \alpha_{41}\int_{\Omega}u_A u_I dx + \alpha_{42}\int_{\Omega}u_A u_C dx = \mu_A\int_{\Omega}u_C dx\leqslant C,\]

        Integrate on time, we have:
        \[\sup_{0<t<T}\int_{\Omega}u_A dx + \alpha_{41}\iint_{Q_T}u_A u_I dxdt + \alpha_{42}\iint_{Q_T}u_A u_C dxdt \leqslant C,\]
        which implies:
        \begin{align}
            &\|u_A\|_{L^1(\Omega)}\leqslant\sup_{0<t<T} \int_{\Omega}u_A dx\leqslant C;\label{ua:l1}\\
            &\|u_A u_I\|_{L^1(Q_T)}\leqslant C;\label{uai:l1}\\
            &\|u_A u_C\|_{L^1(Q_T)}\leqslant C \notag.%\label{uac:l1},
        \end{align}

        Hence, we obtain estimate~\eqref{ac:l1}.

    \end{claimproof}
    With estimate~\eqref{ac:l1}, we can proceed the proof on $\|u_A\|_{V_2(Q_T)}$. By applying $L^2(\Omega)$ estimate on equation~\eqref{equ:4}, we have:
    \begin{equation*}
        \frac{1}{2}\frac{d}{dt}\int_{\Omega}u_A^2dx + d_A^{(0)} \int_{\Omega}|\nabla u_A|^2 dx + \alpha_{41}\int_{\Omega} u_A^2 u_I dx + \alpha_{42}\int_{\Omega} u_A^2 u_C dx \leqslant \mu_A \int_{\Omega}u_C u_A dx
    \end{equation*}

    Then integrate on time,
    \begin{multline*}
        \sup_{0<t<T}\frac{1}{2}\int_{\Omega}u_A^2dx + d_A^{(0)}\iint_{Q_T}|\nabla u_A|^2dxdt +\alpha_{41}\iint_{Q_T}u_A^2u_I dxdt + \alpha_{42}\iint_{Q_T}u_A^2u_C dxdt\\ \leqslant \mu_A\iint_{Q_T}u_C u_A dxdt + \int_{\Omega}u_{A_0}dx\leqslant C, 
    \end{multline*}
    which implies %\footnote{I.C.: $u_{A_0}\in L^1(Q_T)$ required.}
    \begin{equation}\label{v2:ua}
        \|u_A\|_{V_2(Q_T)}\leqslant C.
    \end{equation}

    \textbf{Part 2: To prove} $\mathbf{\|u_A\|_{L^{\infty}(Q_T)}\leqslant C}$ and $\mathbf{\|u_V\|_{L^{\infty}(Q_T)}\leqslant C}$.

    For $L^{\infty}(Q_T)$-estimate of $u_A$, we define $w_A := M_A - u_A$, where $M_A = \max\{ \|u_{A_0}\|_{L^{\infty}(\Omega)},\frac{\mu_A}{\alpha_{42}} \}$, then we convert the equation~\eqref{equ:4} of $u_A$ into an equivalent form of $w_A$:
        \begin{multline*}
            \partial_t w_A = \nabla\cdot(d_A(x,t) \nabla w_A) - (\alpha_{41}u_I + \alpha_{42} u_C)w_A \\ + \alpha_{41}u_I M_A + \alpha_{42} u_C M_A -\mu_A u_C,\;\forall (x,t)\in Q_T;
        \end{multline*}
        \begin{equation*}
            \frac{\partial w_A}{\partial \nu} = 0,\; \forall (x,t)\in \partial\Omega\times (0,T);\quad w_A(x,0) \geqslant 0,\;\forall x \in \Omega.
        \end{equation*}
        Then by the weak maximum principle~\cite{pao}, since $\alpha_{41}u_I + \alpha_{42} u_C\geqslant 0$, $\alpha_{41}u_I M_A + \alpha_{42} u_C M_A -\mu_A u_C\geqslant 0$ and $w_A(x,0)\geqslant 0,\;\forall x\in \Omega$, we have $w_A(x,t)\geqslant 0, \;\forall (x,t)\in Q_T$, which implies 
        \begin{equation}
            u_A \leqslant M_A,\;\forall (x,t) \in Q_T.
        \end{equation}

        Hence, we obtain
        \begin{equation}\label{max:ua} %\label{ua:linf}
            \|u_A\|_{L^{\infty}(Q_T)}\leqslant C.
        \end{equation}

        \begin{remark}
            Notice that $C$ in~\eqref{max:ua} is \textbf{independent} of $T$, so we obtain a \textbf{global} upper bound for $u_A$, i.e. $u_A(x,t)\leqslant M_A,\;\forall (x,t)\in Q$.
        \end{remark}

        For $L^{\infty}(Q_T)$-estimate of $u_V$, we define  $w_V := M_V - u_V$, where $M_V = \max\{ \|u_{V_0}\|_{L^{\infty}(\Omega)}, \frac{\alpha_{33}}{\alpha_{32}}M_A\}$, then we convert equation~\eqref{equ:3} into an equivalent form of $w_V$:
        \begin{multline*}
            \partial_t w_V = \alpha_{31} u_P M_V + (\alpha_{32} M_V - \alpha_{33}u_A )u_I + \mu_V M_V + \frac{\mu_V}{K_V} {(M_V-w_V)}^2 \\ - (\alpha_{31}u_P+\alpha_{32}u_I + \mu_V) w_V,\;\forall (x,t) \in Q_T;
        \end{multline*}
        \[\frac{\partial w_V}{\partial \nu} = 0,\; \forall (x,t)\in \partial\Omega\times (0,T);\quad w_V(x,0)\geqslant 0,\;\forall x \in \Omega.\]
        \begin{multline*}
            \implies w_V =e^{\int_{0}^{t}- (\alpha_{31}u_P+\alpha_{32}u_I + \mu_V)d\tau}( w_V(x,0) \\
            + \int_{0}^{t}\alpha_{31} u_P M_V + (\alpha_{32} M_V - \alpha_{33}u_A )u_I + \mu_V M_V + \frac{\mu_V}{K_V} (M_V-w_V)^2 d\tau)
        \end{multline*}
        Notice that $\alpha_{31} u_P M_V + (\alpha_{32} M_V - \alpha_{33}u_A )u_I + \mu_V M_V + \frac{\mu_V}{K_V}(M_V-w_V)^2\geqslant 0, \;\forall (x,t)\in Q_T$ and $w_V(x,0)\geqslant 0,\;\forall x \in \Omega$, hence, $w_V\geqslant 0,\;\forall (x,t)\in Q_T $. And it implies:
        \begin{equation}
            u_V \leqslant M_V,\;\forall (x,t) \in Q_T.
        \end{equation}
        Hence, we obtain
        \begin{equation}\label{max:uv}
          \|u_V\|_{L^{\infty}(Q_T)}\leqslant C.
        \end{equation}
        \begin{remark}
            Notice that the $C$ above is \textbf{independent} of $T$, so we obtain a \textbf{global} upper bound for $u_V$, i.e. $u_V(x,t) \leqslant M_V,\;\forall (x,t)\in Q$.
        \end{remark}

    \textbf{Part 3: To prove} $\mathbf{\|u_P\|_{V_2(Q_T)}\leqslant C}$ and $\mathbf{\|u_I\|_{V_2(Q_T)}\leqslant C}$.

    With estimates~\eqref{uc:l2},~\eqref{max:ua} and~\eqref{max:uv} and applying $L^2(\Omega)$ estimate on~\eqref{equ:5} and~\eqref{equ:6}, respectively:
    \begin{multline}\label{l2:ui}
        \frac{d}{dt}\int_{\Omega} u_I^2 dx + 2d_I^{(0)}\int_{\Omega} |\nabla u_I|^2 dx + 2\alpha_{51} \int_{\Omega} u_I^2 u_A dx + 2\alpha_{52}\int_{\Omega} u_I^2u_V dx\\ \leqslant  2\mu_I \int_{\Omega} u_P u_I dx
        \leqslant \mu_I \int_{\Omega}u_P^2 dx + \mu_I\int_{\Omega}u_I^2 dx
    \end{multline}
    where the last inequality above is due to Cauchy inequality;
    \begin{equation}
    \begin{aligned}\label{l2:up}
        \frac{d}{dt}\int_{\Omega}& u_P^2 dx + 2d_P^{(0)} \int_{\Omega} |\nabla u_P|^2 dx + 2\delta_P\int_{\Omega} u_P^2 dx\\
        &\leqslant 2\alpha_{61}\int_{\Omega} u_C u_A u_P dx + 2\alpha_{62}\int_{\Omega}u_V u_I u_P dx\\
        &\leqslant \alpha_{61} M_A (\int_{\Omega} u_C^2 dx + \int_{\Omega} u_P^2 dx) + \alpha_{62} M_V (\int_{\Omega}u_I^2 dx +\int_{\Omega} u_P^2 dx)\\
        &\leqslant \alpha_{61} M_A \int_{\Omega} u_C^2 dx + (\alpha_{61}M_A +\alpha_{62}M_V) \int_{\Omega} u_P^2 dx + \alpha_{62}M_V\int_{\Omega} u_I^2dx
        \end{aligned}
    \end{equation}
    where the second inequality above is due to inequality~\eqref{max:ua} and~\eqref{max:uv} and Cauchy inequality.

    Then adding up inequalities~\eqref{l2:ui} and~\eqref{l2:up}, by Gr\"{o}nwall's inequality with estimate~\eqref{uc:l2}, we obtain:
    \begin{multline*}
        \sup_{0<t<T} \int_{\Omega}u_I^2dx+\sup_{0<t<T} \int_{\Omega}u_P^2dx 
        +2d_I^{(0)}\iint_{Q_T} |\nabla u_I|^2 dx dt + 2d_P^{(0)} \iint_{Q_T} |\nabla u_P|^2 dx dt\\
         + 2\alpha_{51} \int_{Q_T} u_I^2 u_A dx + 2\alpha_{52}\int_{Q_T} u_I^2u_V dx \leqslant C 
    \end{multline*}

    Hence, we obtain
    \begin{align}
    &\|u_I\|_{V_2(Q_T)}\leqslant C; \label{v2:ui}\\
    &\|u_P\|_{V_2(Q_T)}\leqslant C.\label{v2:up}
    \end{align}
    \textbf{Part 4: To prove $\mathbf{\|u_I\|_{L^{\infty}(Q_T)}\leqslant C}$} for $\sd\leqslant 3$.

    We will prove the case for spatial dimension, $\sd=3$, in details. For $\sd\leqslant 2$, the proof will be similar, and we will omit them here. 
    
    Apply~\cref{camponato} on equations~\eqref{equ:6} with estimates~\eqref{uc:l2},~\eqref{max:ua},~\eqref{max:uv}, and~\eqref{v2:up}, we have the following estimates:
    \begin{flalign}
        \|\nabla u_P\|_{\Camp{2}{2}}^2 &\leqslant C(\|u_P\|_{L^2(Q_T)}^2+\|u_Cu_A\|_{\Camp{2}{0}}^2+\|u_Vu_I\|_{\Camp{2}{0}}^2+\|u_P\|_{\Camp{2}{0}}^2)\notag\\
        &\leqslant C(\|u_P\|_{L^2(Q_T)}^2+\|u_C\|_{L^2(Q_T)}^2+\|u_I\|^2_{L^2(Q_T)})\leqslant C \label{grad:up:camp3}&&\\
        \implies &\|u_P\|_{\Camp{2}{4}}^2\leqslant C\label{up:camp3}
    \end{flalign}
    Then we apply theorem~\eqref{camponato} again on equation~\eqref{equ:5} with estimates~\eqref{max:ua},~\eqref{max:uv},~\eqref{v2:ui},~\eqref{v2:up} and~\eqref{up:camp3}    
    \begin{flalign}
        \|\nabla u_I\|_{\Camp{2}{2}}^2 &\leqslant C(\|u_I\|_{L^2(Q_T)}^2+\|u_Iu_A\|_{\Camp{2}{0}}^2+\|u_Vu_I\|_{\Camp{2}{0}}^2+\|u_P\|_{\Camp{2}{0}}^2)\notag\\
        &\leqslant C(\|u_I\|_{L^2(Q_T)}^2+\|u_P\|_{L^2(Q_T)}^2)\leqslant C &&\\
        \implies &\|u_I\|_{\Camp{2}{4}}^2\leqslant C\label{ui:camp3}
    \end{flalign}
    With estimate~\eqref{ui:camp3} for spatial dimension $n=3$, we can further have the following estimate:
    \begin{flalign}
        \|\nabla u_I\|_{\Camp{2}{4}}^2 &\leqslant C(\|u_I\|_{L^2(Q_T)}^2+\|u_Iu_A\|_{\Camp{2}{2}}^2+\|u_Vu_I\|_{\Camp{2}{2}}^2+\|u_P\|_{\Camp{2}{2}}^2)\notag\\
        &\leqslant C(\|u_I\|_{L^2(Q_T)}^2+\|u_I\|_{\Camp{2}{2}}^2+\|u_P\|_{\Camp{2}{2}}^2)&&\notag\\
        &\leqslant C(\|u_I\|_{L^2(Q_T)}^2+\|u_I\|_{\Mor{2}{4}}^2+\|u_P\|_{\Mor{2}{4}}^2)\notag\\
        &\leqslant C(\|u_I\|_{L^2(Q_T)}^2+\|u_I\|_{\Camp{2}{4}}^2+\|u_P\|_{\Camp{2}{4}}^2)\notag\leqslant C\notag\\
        \implies &\|u_I\|_{\Camp{2}{6}}^2\leqslant C\label{ui:camp6} \\
        \iff& \|u_I\|_{\Hold{\alpha}}\leqslant C,\,\text{where}\,\alpha = \frac{1}{2}.
    \end{flalign}
    Hence, we obtain
    \begin{equation}\label{max:ui}
        \|u_I\|_{L^{\infty}(Q_T)}\leqslant C,
    \end{equation}
    when $\sd\leqslant 3$.
 
    \textbf{Part 5: To prove $\mathbf{\|u_V\|_{V_2(Q_T)}\leqslant C}$} for $\sd\leqslant 3$. 
    
    With estimates~\eqref{v2:ui},~\eqref{v2:up},~\eqref{max:ua} and~\eqref{max:ui}, we can proceed the proof on $\nabla u_V$.By taking gradient and dot product with $\nabla u_V$ on equation~\eqref{equ:3}:
    \begin{flalign*}
        \nabla u_V \cdot \partial_t \nabla u_V =& (-\alpha_{31}u_P - \alpha_{32}u_I + \mu_V - 2\mu_V u_V)|\nabla u_V|^2 \\
        &+ u_V \nabla (-\alpha_{31}u_P - \alpha_{32}u_I)\cdot \nabla u_V + \nabla (\alpha_{33}u_Au_I)\cdot \nabla u_V.\\
    \end{flalign*}
    
    Then integrating over $\Omega$ on spatial variable $x$,
    \begin{flalign*}
        \implies \frac{d}{dt} \int_{\Omega} |\nabla u_V|^2 dx =& 2\mu_V\int_{\Omega}|\nabla u_V|^2dx +2\int_{\Omega}(-\alpha_{31}u_P - \alpha_{32}u_I  - 2\mu_V u_V)|\nabla u_V|^2dx \\
        &-2\int_{\Omega}\alpha_{31}u_V \nabla u_P\cdot \nabla u_Vdx-2\int_{\Omega}\alpha_{32}u_V \nabla u_I\cdot \nabla u_V dx\\
        &+ 2\int_{\Omega}\alpha_{33}u_A \nabla u_I \cdot \nabla u_V dx+2\int_{\Omega}\alpha_{33}u_I \nabla u_A \cdot \nabla u_Vdx\\
        &\leqslant C(\int_{\Omega}|\nabla u_V|^2dx + \int_{\Omega} |\nabla u_P|^2 dx + \int_{\Omega} |\nabla u_I|^2 dx+ \int_{\Omega} |\nabla u_A|^2 dx),
        \end{flalign*}
    where the inequality is due to Cauchy inequality. 
    
    Hence, by Gr\"onwall inequality and estimates~\eqref{v2:ua},~\eqref{v2:ui} and~\eqref{v2:up}, we have:
    \begin{equation}
        \sup_{0<t<T} \|\nabla u_V\|_{L^2(\Omega)} \leqslant C.
    \end{equation}

    Therefore, we obtain:
    \begin{equation}\label{v2:uv}
        \|u_V\|_{V_2(Q_T)}\leqslant C.
    \end{equation}

    \textbf{Part 6: To prove $\mathbf{\|u_C\|_{V_2(Q_T)}\leqslant C}$ and $\mathbf{\|u_N\|_{V_2(Q_T)}\leqslant C}$} for $\sd\leqslant 3$.

    Applying $L^2(\Omega)$ estimate on equations~\eqref{equ:1} and~\eqref{equ:2}, then adding those up. Assuming that $(\chi_{11}+\chi_{21})^2<4d_C^{(0)}d_N^{(0)}$ and~\eqref{hypo:bound}, by Cauchy inequality, we obtain the following:
    \begin{flalign}\label{longv2}
        &\frac{1}{2}\frac{d}{dt}\int_{\Omega} (u_C^2 +u_N^2) dx + d_C^{(0)}\int_{\Omega}|\nabla u_C|^2dx +d_N^{(0)}\int_{\Omega}|\nabla u_N|^2dx + \frac{\mu_C}{K_C}\int_{\Omega}u_C^3 dx + \frac{\mu_N}{K_N}\int_{\Omega}u_N^3dx &&\notag\\
        &\leqslant \int_{\Omega}\chi_{11} \nabla u_N \cdot \nabla u_C dx+\int_{\Omega}\chi_{12} \nabla u_A \cdot \nabla u_C dx +\int_{\Omega}\chi_{13} \nabla u_I \cdot \nabla u_C dx +\int_{\Omega}\chi_{14} \nabla u_V \cdot \nabla u_C dx&&\notag\\
        & + \int_{\Omega}\chi_{21} \nabla u_C \cdot \nabla u_N dx+\int_{\Omega}\chi_{22} \nabla u_A \cdot \nabla u_N dx+\int_{\Omega}\chi_{23} \nabla u_I \cdot \nabla u_N dx +\int_{\Omega}\chi_{24} \nabla u_V \cdot \nabla u_N dx &&\notag\\
        & +\alpha_{11}\int_{\Omega} u_N u_C^2 dx+ \max\{\mu_C-\delta_C,\mu_N-\delta_N\} \int_{\Omega}(u_C^2 +u_N^2) dx - \alpha_{21}\int_{\Omega}u_N^2u_C dx &&\notag\\
        &\leqslant\frac{(\chi_{11}+\chi_{21})^2}{4\varepsilon} \int_{\Omega}|\nabla u_C|^2dx + \varepsilon\int_{\Omega}|\nabla u_N|^2 dx+ C(\xi)\int_{\Omega}|\nabla u_A|^2+|\nabla u_I|^2 + |\nabla u_V|^2dx  \notag\\
        & + \xi \int_{\Omega} |\nabla u_C|^2+ |\nabla u_N|^2 dx  + (\frac{\theta -1}{\theta})\frac{\mu_C}{K_C}\int_{\Omega}u_C^3 dx + C\int_{\Omega}u_C^2+u_N^2 dx,\notag
    \end{flalign}
    where $\theta$ is defined to satisfying $\alpha_{11}^2 \leqslant (\frac{\theta -1}{\theta})\frac{4\alpha_{21}\mu_C}{K_C}$ following by~\eqref{hypo:reaction}.

    Since we assume $(\chi_{11}+\chi_{21})^2<4d_C^{(0)}d_N^{(0)}$, there exists $\gamma >0$, such that 
    \begin{equation}
        (\chi_{11}+\chi_{21})^2 \leqslant (\frac{\gamma-1}{\gamma})^2 4 d_C^{(0)}d_N^{(0)}.\label{const:gamma}
    \end{equation} 
    Pick $\varepsilon = \frac{\gamma -1 }{\gamma} d_N^{(0)}$ and $\xi = \min\{\frac{1}{2\gamma}d_C^{(0)},\frac{1}{2\gamma}d_N^{(0)}\}$, then by Gr\"onwall's inequality, it yields the following estimate:
    \begin{equation}
    \sup_{0< t< T}\int_{\Omega}(u_C^2 + u_N^2 )dx + \iint_{Q_T} (|\nabla u_C|^2 + |\nabla u_N|^2) dxdt +\iint_{Q_T}u_C^3 + u_N^3 dxdt \leqslant C.\label{ucun:v2}
    \end{equation}
    Hence, we obtain 
    \begin{align}
    \|u_C\|_{V_2(Q_T)}\leqslant C, \label{uc:v2}\\
    \|u_N\|_{V_2(Q_T)}\leqslant C.
    \end{align}
\begin{remark}
    We also obtain some $L^3(Q_T)$ estimate from~\eqref{ucun:v2}:
    \begin{align}
    \|u_C\|_{L^3(Q_T)}\leqslant C, \label{uc:l3}\\
    \|u_N\|_{L^3(Q_T)}\leqslant C.\label{un:l3}
    \end{align}
\end{remark}

    \input{dual_estimate.tex}

    \textbf{Part 8: To prove $\mathbf{\|u_P\|_{L^{\infty}(\Omega)}\leqslant C}$} for $\sd\leqslant 3$.

    We will prove the case for spatial dimension, $\sd=3$, in details. For $\sd\leqslant 2$, the proof will be similar, and we will omit them here. By Morrey-Campanato space embedding, we have the inequality:
    \begin{equation}
        \|u\|_{L^{2,\frac{5}{3}}_C(Q_T)}\leqslant C\|u\|_{L^3(Q_T)},\,\forall u\in L^3(Q_T).
    \end{equation}

    Then with $L^3(Q_T)$-estimate~\eqref{uc:l3}, $V_2(Q_T)$-estimates~\eqref{v2:ui} and~\eqref{v2:up}, $L^{\infty}(Q_T)$-estimates~\eqref{max:ua} and~\eqref{max:uv}, and $L^{2,4}_C(Q_T)$-estimate~\eqref{up:camp3} and~\eqref{ui:camp3}, we have the following estimate: for any $\alpha\in(0,\frac{1}{3})$
    \begin{flalign}
        \|\nabla u_P\|_{\Camp{2}{3+2\alpha}}^2 &\leqslant C(\|u_P\|_{L^2(Q_T)}^2+\|u_Cu_A\|_{\Camp{2}{1+2\alpha}}^2+\|u_Vu_I\|_{\Camp{2}{1+2\alpha}}^2+\|u_P\|_{\Camp{2}{1+2\alpha}}^2)\notag\\
        &\leqslant C(\|u_P\|_{L^2(Q_T)}^2+\|u_C\|_{L^3(Q_T)}^2+\|u_I\|_{\Camp{2}{1+2\alpha}}^2+\|u_P\|_{\Camp{2}{1+2\alpha}}^2)\leqslant C &&\\
        \implies &\|u_P\|_{\Camp{2}{5+2\alpha}}^2\leqslant C\\
        \iff &\|u_P\|_{\Hold{\alpha}}^2\leqslant C.
    \end{flalign}
    Similar process works for $\sd=2$ with an embedding
    \begin{equation}
    \|u\|_{L^{2,2}_C(Q_T)}\leqslant C\|u\|_{L^4(Q_T)}\leqslant C (\|u\|_{L^2(0,T;H^1(\Omega)} + \|u\|_{L^{\infty}(0,T;L^2(\Omega))})\leqslant C\|u\|_{V_2(Q_T)},\,\text{for}\,d=2,\label{camp:embed:2d}
    \end{equation} 
    and estimate~\eqref{up:camp3} directly implies the desired estimate for $\sd=1$, so we obtain 
    \begin{equation}
        \|u_P\|_{L^{\infty}(Q_T)}\leqslant C,\,\forall\, \sd\leqslant 3.
    \end{equation}
    This completes the proof.
\end{proof}

%% file: dual_estimate.tex
\textbf{Part 7: To prove $\|\partial_t u_i\|_{L^2(0,T;H^*(\Omega))}\leqslant C,\,\forall i\in\I$ for $d\leqslant 3$.}

We shall prove the case for spatial dimension, $d=3$, in details. For $d\leqslant 2$, the proof will be similar in the virtue of an embedding $V_2(Q_T)\hookrightarrow L^4(Q_T)$, and we will omit them here. Given $v\in \{v: \|v\|_{H^1(\Omega)}\leqslant 1\}$ and recall the dual pairing between $H^*(\Omega)$ and $H^1(\Omega)$ by $\left<u,v \right>_H$, then we have:
\begin{align}
\left<\partial_t u_C,v \right>_H =& \left<\nabla\cdot (d_C(x,t) \nabla u_C),v \right>_H -\sum_{(i,j)\in \{(1,N),(2,A),(3,I),(4,V)\}}\left<\nabla \chi_{1i} \cdot B(u_C) \nabla u_j,v \right>_H\notag\\
&+ \left<\alpha_{11} u_C u_N,v \right>_H + \left<(\mu_C-\delta_C)u_C,v \right>_H - \left<\frac{\mu_C}{K_C}u_C^2,v \right>_H \notag\\
\leqslant & C( \sum_{j\in \I\backslash\{P\}}\|\nabla u_j\|_{L^2(\Omega)}\|v\|_{H^1(\Omega)} + \|u_C\|_{L^{\frac{12}{5}}(\Omega)}\|u_N\|_{L^{\frac{12}{5}}(\Omega)}\|v\|_{L^{6}(\Omega)}\notag\\
&+ \|u_C\|_{L^2(\Omega)}\|v\|_{L^2(\Omega)} +\|u_C\|^2_{L^{\frac{12}{5}}(\Omega)}\|v\|_{L^{6}(\Omega)})\notag\\
\leqslant& C \sum_{j\in \I\backslash\{P\}}\|\nabla u_j\|_{L^2(\Omega)} + \|u_C\|^2_{L^{3}(\Omega)} + \|u_C\|_{L^2(\Omega)},
\end{align}
where the second inequality is due to an embedding of $H^1(\Omega)\hookrightarrow L^6(\Omega)$ by Gagliardo-Nirenberg-Sobolev inequality for $d=3$, the third inequality is due to Cauchy inequality. Thus, we have:
\begin{equation}
    \|\partial_t u_C\|_{H^*(\Omega)} = \sup_{v\in\{v:\|v\|_{H^1(\Omega)}\leqslant 1\}} \left<\partial_t u_C,v \right>_H \leqslant  C \sum_{j\in \I\backslash\{P\}}\|\nabla u_j\|_{L^2(\Omega)} + \|u_C\|^2_{L^{3}(\Omega)} +\|u_N\|^2_{L^{3}(\Omega)} + \|u_C\|_{L^2(\Omega)},
\end{equation}

Notice that we need an important embedding to proceed as stated by the following claim:
\begin{claim}\label{v2:embedding}
    For dimension $d=3$, any $V_2(Q_T)$ function is embedding in $L^4(0,T;L^3(\Omega))$.
\end{claim}
\begin{claimproof}
    For any function $u\in V_2(Q_T)$,
    \begin{align}
        \|u\|^4_{L^4(0,T;L^3(\Omega))} &= \int_{0}^{T}\|u\|_{L^3(\Omega)}^4 dt\notag\\
        &\leqslant \int_{0}^{T} \|u\|_{L^2(\Omega)}^{2}\|u\|_{L^6(\Omega)}^{2} dt\notag\\
        &\leqslant \sup_{0<t<T}\|u\|_{L^2(\Omega)}^{2} \int_{0}^{T} C \|u\|_{H^1(\Omega)}^{2} dt \leqslant C\label{embedding:l3:v2}
    \end{align}
    where the first inequality is due to interpolation inequality, the second inequality is due to Gagliardo-Nirenberg-Sobolev inequality.
\end{claimproof}

Then we square both side and integrate over time variable, and one can obtain the following estimate:
\begin{align}
\int_{0}^T \|\partial_t u_C\|_{H^*(\Omega)}^2 dt &\leqslant  C (\sum_{j\in \I\backslash\{P\}}\| u_j\|^2_{L^2(0,T;H^1(\Omega))} + \|u_C\|^4_{L^{4}(0,T;L^3(\Omega))} +\|u_N\|^4_{L^{4}(0,T;L^3(\Omega))} + \|u_C\|^2_{L^2(Q_T)})\notag\\
&\leqslant C,
\end{align}
where the last inequality is due to Cauchy inequality, an embedding $V_2(Q_T) \hookrightarrow L^4(0,T;L^3(\Omega))$ for dimension $d=3$, and estimates in part 1-6.
Similarly, one would have the following estimates:
\begin{align}
    \|\partial_t u_N\|_{H^*(\Omega)} \leqslant  C \sum_{j\in \I\backslash\{P\}}\|\nabla u_j\|_{L^2(\Omega)} + \|u_C\|^2_{L^{3}(\Omega)} +\|u_N\|^2_{L^{3}(\Omega)} + \|u_N\|_{L^2(\Omega)},
\end{align}
which yields:
\begin{equation}
    \int_{0}^T \|\partial_t u_N\|_{H^*(\Omega)}^2 dt \leqslant C.
\end{equation}

Next, for equation~\eqref{equ:3} -~\eqref{equ:6}, one have the following estimates:

\begin{align}
    \left<\partial_t u_V,v \right>_H =& - \left<\alpha_{31} u_V u_P,v \right>_H - \left<\alpha_{32} u_V u_I,v \right>_H + \left<\alpha_{33} u_A u_I,v \right>_H +\left<\mu_V u_V,v \right>_H -\left<\frac{\mu_V}{K_V} u^2_V,v \right>_H\notag \\
    \leqslant& C(\|u_P\|_{L^2(\Omega)}+\|u_V\|_{L^2(\Omega)} + \|u_I\|_{L^2(\Omega)}), \label{dual:uv}\\
\left<\partial_t u_A,v \right>_H =& \left<\nabla\cdot (d_A(x,t) \nabla u_A),v \right>_H - \left<\alpha_{41} u_A u_I,v \right>_H - \left<\alpha_{42} u_A u_C,v \right>_H + \left<\mu_A u_C,v \right>_H\notag\\
\leqslant & C (\|\nabla u_A\|_{L^2(\Omega)} + \|u_C\|_{L^2(\Omega)}), \label{dual:ua}\\
\left<\partial_t u_I,v \right>_H =& \left<\nabla\cdot (d_I(x,t) \nabla u_I),v \right>_H - \left<\alpha_{51} u_A u_I,v \right>_H - \left<\alpha_{52} u_A u_V,v \right>_H + \left<\mu_I u_P,v \right>_H\notag\\
\leqslant & C (\|\nabla u_I\|_{L^2(\Omega)} + \|u_I\|_{L^2(\Omega)} + \|u_A\|_{L^2(\Omega)} + \|u_P\|_{L^2(\Omega)}),\label{dual:ui}\\
\left<\partial_t u_P,v \right>_H =& \left<\nabla\cdot (d_P(x,t) \nabla u_P),v \right>_H + \left<\alpha_{61} u_C u_A,v \right>_H + \left<\alpha_{62} u_V u_I,v \right>_H - \left<\delta_P u_P,v \right>_H\notag\\
\leqslant & C (\|\nabla u_A\|_{L^2(\Omega)} +\|u_C\|_{L^2(\Omega)} + \|u_I\|_{L^2(\Omega)}+ \|u_P\|_{L^2(\Omega)}). \label{dual:up}
\end{align}

Then they yield:
\begin{equation}
    \sum_{i\in\I\backslash \{C,N\}}\int_{0}^T \|\partial_t u_i\|_{H^*(\Omega)}^2 dt \leqslant C.
\end{equation}

Hence, we obtain:
\begin{equation}
    \|\partial_t u_i\|_{L^2(0,T;H^*(\Omega))} \leqslant C,\,\forall\,\sd\leqslant 3.
\end{equation}

%% file: Existence.tex
\section{Proof of theorem~\eqref{exist}}
In this section, we will show the proof of theorem~\eqref{exist} by separating the system into two existence and existence problems. For notation simplicity, we recall an index set notation, $\I:=\{C,N,V,A,I,P\}$.

\begin{lem}\label{sepa:lemma1}
    For arbitrary $T>0$, spatial dimension $\sd\leqslant 2$ and $\alpha$ satisfies~\eqref{hypo:init}, given $\hat{u}_V \in \{u:\|\nabla u\|_{\Camp{2}{\sd+2\alpha}}<\infty,\,\text{and}\,\nabla_{\vec{\nu}} u(x,t) = 0,\,\forall (x,t)\in\partial \Omega \times (0,T)\},$ the following system:
    \begin{flalign}
        \partial_t u_C - \nabla \cdot (d_C(x,t)\nabla u_C)=& -\nabla\cdot (\chi_{11} B(u_C)\nabla u_N +\chi_{12} B(u_C)\nabla u_A + \chi_{13} B(u_C)\nabla u_I\notag\\ 
        + \chi_{14} B(u_C)\nabla &\hat{u}_V) + \alpha_{11}u_N u_C + \mu_C u_C (1-\frac{u_C}{K_C})-\delta_C u_C;\label{ori:c}\\
        \partial_t u_N - \nabla \cdot (d_N(x,t)\nabla u_N)=& -\nabla\cdot (\chi_{21} B(u_N)\nabla u_C + \chi_{22} B(u_N)\nabla u_A + \chi_{23} B(u_N)\nabla u_I\notag\\ 
        + \chi_{24} B(u_N)\nabla &\hat{u}_V) -\alpha_{21} u_N u_C + \mu_N u_N (1-\frac{u_N}{K_N})-\delta_N u_N; &&\\
        \partial_t u_A - \nabla \cdot (d_A(x,t)\nabla u_A)=& - \alpha_{41}u_A u_I -\alpha_{42}u_A u_C + \mu_A u_C;\\
        \partial_t u_I - \nabla \cdot (d_I(x,t)\nabla u_I)=& - \alpha_{51}u_I u_A - \alpha_{52}u_I \hat{u}_V + \mu_I u_P; \\
        \partial_t u_P - \nabla \cdot (d_P(x,t)\nabla u_P)=& \alpha_{61} u_C u_A+ \alpha_{62}\hat{u}_V u_I - \delta_P u_P,\label{ori:p}
    \end{flalign}
    with Neumann type boundary condition~\eqref{bc}
    and initial condition~\eqref{ic} (excluding $u_{V_0}$), has a unique weak solution, for some $\chi_{11}$ and $\chi_{21}$.
\end{lem}

\begin{proof} 
     We shall prove the existence and uniqueness for the case when dimension $\sd=2$, and a similar process would be applied for $\sd=1$, which we would omit here.

    \textbf{Estimate:} Given $\hat{u}_V \in \{u: \|\nabla u\|_{\Camp{2}{\sd+2\alpha}}<\infty,\text{and}\,\nabla_{\vec{\nu}} u(x,t) = 0,\,\forall (x,t)\in\partial \Omega \times (0,T)\}$, 
    then with similar process in section 3, we obtain
    \begin{equation}\label{lemma1:estimate}
        \sum_{i\in\I\backslash \{V\}}\|\nabla u_i\|_{\Camp{2}{\sd+2\alpha}}\leqslant C \implies u_C,u_N,u_A,u_I,u_P\in  C^{\alpha,\frac{\alpha}{2}}(\bar{Q}_T) \cap V_2(Q_T).
    \end{equation}

    \textbf{Existence:} We shall use Schaefer fixed point theorem to prove the existence.

    Given $\tilde{u}:=(\tilde{u}_C,\tilde{u}_N,\tilde{u}_A,\tilde{u}_I,\tilde{u}_P)\in (L^{\infty}(Q_T))^5$, we have the following system of equations: %\footnote{tilde on $\chi_{i4}$ or not? (continuity)}
    \begin{flalign}
    \partial_t u_C - \nabla \cdot (d_C(x,t)\nabla u_C)=& -\nabla\cdot (\chi_{11}B(\tilde{u}_C)\nabla u_N +\chi_{12}B(\tilde{u}_C) \nabla u_A + \chi_{13}B(\tilde{u}_C)\nabla u_I\notag\\ 
    + \chi_{14}B(u_C)\nabla &\hat{u}_V) + \alpha_{11}\tilde{u}_N u_C + \mu_Cu_C (1-\frac{\tilde{u}_C}{K_C})-\delta_C u_C;\label{lem1:equc}\\
    \partial_t u_N - \nabla \cdot (d_N(x,t)\nabla u_N)=& -\nabla\cdot (\chi_{21}B(\tilde{u}_N)\nabla u_C + \chi_{22}B(\tilde{u}_N)\nabla u_A + \chi_{23}B(\tilde{u}_N)\nabla u_I\notag\\ 
    + \chi_{24}B(u_N)\nabla &\hat{u}_V) -\alpha_{21} u_N \tilde{u}_C + \mu_N u_N (1-\frac{\tilde{u}_N}{K_N})-\delta_N u_N;\label{lem1:equn} &&\\
    \partial_t u_A - \nabla \cdot (d_A(x,t)\nabla u_A)=& - \alpha_{41}u_A \tilde{u}_I -\alpha_{42}\tilde{u}_A u_C + \mu_A u_C;\label{lem1:equa}\\
    \partial_t u_I - \nabla \cdot (d_I(x,t)\nabla u_I)=& - \alpha_{51}u_I \tilde{u}_A - \alpha_{52}u_I \hat{u}_V + \mu_I u_P; \label{lem1:equi}\\
    \partial_t u_P - \nabla \cdot (d_P(x,t)\nabla u_P)=& \alpha_{61} u_C\tilde{u}_A + \alpha_{62}\hat{u}_V u_I - \delta_P u_P,\label{lem1:equp}
\end{flalign}
with Neumann type boundary condition~\eqref{bc} and initial condition~\eqref{ic} (excluding $u_{V_0}$).

By the existence and uniqueness theorem of linear cross-diffusion advection parabolic equation~\cite{lieberman1996second}, we define the mapping, $M:(L^{\infty}(Q_T))^5\rightarrow (L^{\infty}(Q_T))^5$, i.e. \[M[(\tilde{u}_i)_{i\in\I\backslash \{V\}}] = (u_C,u_N,u_A,u_I,u_P) =: u,\]
which is a well-defined mapping.

Since we have shown apriori estimate in section 3, i.e. we know the set $S:=\{u: u = \lambda M[u],\forall \lambda\in [0,1]\}$ is bounded in $(L^{\infty}(Q_T))^5$-norm sense. Hence, we only need to show that the operator $M$ is compact and continuous.

To show $M$ is compact, we notice that with the similar estimate process, we have 
\begin{equation}
\|M[\tilde{u}]\|_{(V_2(Q_T))^5}\leqslant C.
\end{equation}
Hence, we can apply~\cref{camponato} to obtain that 
\begin{equation}
\|M[\tilde{u}]\|_{(C^{\alpha,\frac{\alpha}{2}}(\bar{Q}_T))^5}\leqslant C.
\end{equation}
Therefore, this mapping $M$ is an embedding of $(C^{\alpha,\frac{\alpha}{2}}(\bar{Q}_T))^5$ into $(L^{\infty}(Q_T))^5$, which is compact.

To show $M$ is continuous, we begin by $L^2$-energy estimate on a convergent sequence of given function $(\tilde{u}_j)_{j\in\N}$ in the sense of $L^{\infty}(Q_T)$-norm sense, which converges to $\tilde{U}$. Then, for some $\chi_{11},\,\chi_{21}$, satisfying $(\chi_{11}+\chi_{21})^2<4d_C^{(0)}d_N^{(0)}$, we obtain the following:
\begin{equation}
    \|M[\tilde{u}_j] - M[\tilde{U}]\|_{(V_2(Q_T))^5}\leqslant C(\|\tilde{u}_j - \tilde{U}\|_{(L^{\infty}(Q_T))^5} + \|\tilde{u}_{C_j}-\tilde{u}_C\|^2_{L^{\infty}(Q_T)} + \|\tilde{u}_{N_j}-\tilde{u}_N\|^2_{L^{\infty}(Q_T)}),
\end{equation}
where $C$ is independent of $j$. Hence, $M$ is continuous.

Also, we know that the sequence $(M[\tilde{u}_j])_{j\in \N}$ is uniformly bounded in $C^{\alpha,\frac{\alpha}{2}}(\bar{Q}_T)$, which is also equicontinuous. Therefore, by Arzelà–Ascoli theorem, we obtain a convergent subsequence, $(M[\tilde{u}_{j_k}])_{k\in\N}$, which converges to $U^*$. Due to the uniqueness of the limit in $L^2(Q_T)$, $U^* \equiv M[\tilde{U}]$ and also any convergent subsequence of $(M[u_j])_{j\in \N}$ converges to $M[\tilde{U}]$ in $L^{\infty}(Q_T)$-sense, which means $(M[u_j])_{j\in \N}$ converges to $M[\tilde{U}]$ in $L^{\infty}(Q_T)$-sense, i.e.
\[\lim_{j\rightarrow \infty} M[\tilde{u}_j] = M[\tilde{U}]= M[\lim_{n\rightarrow\infty} \tilde{u}_j].\]

\textbf{Uniqueness:} Given a fixed $\hat{u}_V$, suppose there are two distinct fixed point, which denoted by $z_1:=(u_C^{(1)},u_N^{(1)}, u_A^{(1)}, u_I^{(1)},u_P^{(1)})$ and $z_2:=(u_C^{(2)},u_N^{(2)}, u_A^{(2)}, u_I^{(2)},u_P^{(2)})$. Taking the difference of system of equations of fixed point $z_1$ and $z_2$, we have the following system of equations, with denoting the difference by $\bar{z}:= z_1-z_2 = (u_C^{(1)}-u_C^{(2)},u_N^{(1)}-u_N^{(2)}, u_A^{(1)}-u_A^{(2)}, u_I^{(1)}-u_I^{(2)},u_P^{(1)}-u_P^{(2)})=:(\bar{u}_C,\bar{u}_N,\bar{u}_A,\bar{u}_I,\bar{u}_P)$, $\bar{\hat{u}}_V:= \hat{u}_V^{(1)}-\hat{u}_V^{(2)}$ and an index set $\I_p:=\{(A,2),(I,3)\}$:\\

\begin{flalign}
    \partial_t \bar{u}_C &-\nabla\cdot(d_C(x,t)\nabla \bar{u}_C) = \sum_{(i,j)\in \I_p\cup\{(N,1)\}}-\nabla\cdot(\chi_{1j}B(u_C^{(1)})\nabla \bar{u}_i + \chi_{1j}(B(u_C^{(1)})-B(u_C^{(2)}))\nabla u_i^{(2)})\notag  &&\\
    &-\nabla\cdot( \chi_{14}(B(u_C^{(1)})-B(u_C^{(2)}))\nabla \hat{u}_V)+\alpha_{11}(\bar{u}_N u_C^{(1)}+\bar{u}_Cu_N^{(2)}) + (\mu_C - \delta_C)\bar{u}_C -\frac{\mu_C}{K_C}\bar{u}_C (u_C^{(1)}+u_C^{(2)})\label{uniq:c}\\
    \partial_t \bar{u}_N &-\nabla\cdot(d_N(x,t)\nabla \bar{u}_C) = \sum_{(i,j)\in \I\cup\{(C,1)\}}-\nabla\cdot(\chi_{2j} B(u_N^{(1)})\nabla \bar{u}_i + \chi_{2j}(B(u_N^{(1)})-B(u_N^{(2)}))\nabla u_i^{(2)}) \notag&&\\
    &-\nabla\cdot( \chi_{24}(B(u_N^{(1)})-B(u_N^{(2)}))\nabla \hat{u}_V)-\alpha_{21}(\bar{u}_N u_C^{(1)}+\bar{u}_Cu_N^{(2)}) + (\mu_N - \delta_N)\bar{u}_N - \frac{\mu_N}{K_N}\bar{u}_N (u_N^{(1)}+u_N^{(2)})\label{uniq:n}\\
    \partial_t \bar{u}_A&-\nabla\cdot(d_A(x,t)\nabla \bar{u}_A) =-\alpha_{41}(\bar{u}_A u_I^{(1)}+\bar{u}_Iu_A^{(2)})-\alpha_{42}(\bar{u}_A u_C^{(1)}+\bar{u}_Cu_A^{(2)}) + \mu_A \bar{u}_C\label{uniq:a}\\
    \partial_t \bar{u}_I&-\nabla\cdot(d_I(x,t)\nabla \bar{u}_I) =-\alpha_{51}(\bar{u}_A u_I^{(1)}+\bar{u}_Iu_A^{(2)})-\alpha_{52}\bar{u}_I\hat{u}_V + \mu_I \bar{u}_P\label{uniq:i}\\
    \partial_t \bar{u}_P&-\nabla\cdot(d_P(x,t)\nabla \bar{u}_P) =\alpha_{61}(\bar{u}_A u_C^{(1)}+\bar{u}_Cu_A^{(2)})+\alpha_{62}\bar{u}_I \hat{u}_V - \delta_P \bar{u}_P\label{uniq:p}
\end{flalign}
By applying the $L^2(\Omega)$ estimate on~\eqref{uniq:c}, we have 
    \begin{flalign}
        &\frac{1}{2}\frac{d}{dt} \int_{\Omega} \bar{u}_C^2dx + d_C^{(0)}\int_{\Omega}|\nabla \bar{u}_C|^2 dx\leqslant
         \underbrace{\int_{\Omega}\chi_{11}B(u_C^{(1)})\nabla \bar{u}_N\cdot \nabla \bar{u}_C + \chi_{11}(B(u_C^{(1)})-B(u_C^{(2)}))\nabla u_N^{(2)}\cdot \nabla \bar{u}_C dx}_{(\ref{uniq:c1}.1)}\notag &&\\
         &+\underbrace{\sum_{(i,j)\in \I}\int_{\Omega}\chi_{1j} B(u_C^{(1)})\nabla \bar{u}_i\cdot \nabla \bar{u}_C + \chi_{1j}(B(u_C^{(1)})-B(u_C^{(2)}))\nabla u_i^{(2)}\cdot\nabla \bar{u}_C dx}_{(\ref{uniq:c1}.2)} &&\notag\\
         &+\underbrace{\int_{\Omega}\chi_{14}((B(u_C^{(1)})-B(u_C^{(2)})))\nabla \hat{u}_V \cdot \nabla \bar{u}_C dx}_{(\ref{uniq:c1}.3)}&&\notag\\
         &+ \underbrace{\int_{\Omega}\alpha_{11}u_C^{(1)}\bar{u}_N \bar{u}_C+\alpha_{11}u_N^{(2)}\bar{u}_C^2 dx+ \int_{\Omega}(\mu_C-\delta_C)\bar{u}_C^2 dx-  \frac{\mu_C}{K_C} \int_{\Omega} \big(u_C^{(1)}+u_C^{(2)} \big) \bar{u}_C^2 dx }_{(\ref{uniq:c1}.4)} &&\label{uniq:c1}
    \end{flalign}
    we see that
    \begin{align*}
        (\ref{uniq:c1}.1)&\leqslant \int_{\Omega}\chi_{11}\nabla \bar{u}_N\cdot \nabla \bar{u}_C dx + \chi_{11}\int_{\Omega} (B(u_C^{(1)})-B(u_C^{(2)})) \nabla u_N^{(2)}\cdot \nabla \bar{u}_C dx
    \end{align*}

Notice that with~\eqref{hypo:bound}:
\begin{equation}\label{B:lip}
    |B(u_C^{(1)})-B(u_C^{(2)})|\leqslant |\bar{u}_C|,\,\forall (x,t) \in Q_T.
\end{equation}
Then by~\eqref{B:lip} and Cauchy inequality:
    \begin{flalign*}
        \int_{\Omega} (B(u_C^{(1)})-B(u_C^{(2)})) \nabla u_N^{(2)}\cdot \nabla \bar{u}_C dx &\leqslant \|\bar{u}_C\|_{L^{\infty}(\Omega)}\|\nabla \bar{u}_C\|_{L^{2}(\Omega)}\|\nabla u_N^{(2)}\|_{L^2(\Omega)}&&\\
        &\leqslant C\|\bar{u}_C\|^2_{L^{\infty}(\Omega)}\|\nabla u_N^{(2)}\|^2_{L^2(\Omega)}+\varepsilon\|\nabla \bar{u}_C\|^2_{L^{2}(\Omega)}.
    \end{flalign*}

    Hence, we obtain:
    \begin{flalign*}
        (\ref{uniq:c1}.1)\leqslant \int_{\Omega}\chi_{11} \nabla \bar{u}_N \cdot \nabla \bar{u}_C dx + C\|\bar{u}_C\|^2_{L^{\infty}(\Omega)}\|\nabla u_N^{(2)}\|^2_{L^2(\Omega)}+\varepsilon\|\nabla \bar{u}_C\|^2_{L^{2}(\Omega)}.&&
    \end{flalign*}

    Also, one can repeat the same process to estimate (\ref{uniq:c1}.2), we have:   
\begin{flalign*}
    (\ref{uniq:c1}.2)\leqslant &C(\|\nabla \bar{u}_A \|^2_{L^2(\Omega)}+\|\nabla \bar{u}_I \|^2_{L^2(\Omega)} )+ \varepsilon \|\nabla \bar{u}_C\|_{L^2(\Omega)}^2+ \|\bar{u}_C\|^2_{L^{\infty}(\Omega)}(\|\nabla u^{(2)}_I\|^2_{L^{2}(\Omega)}+\|\nabla u_A^{(2)}\|^2_{L^2(\Omega)});&&\\
    (\ref{uniq:c1}.3)\leqslant& C\|\bar{u}_C\|^2_{L^{\infty}(\Omega)}\|\nabla \hat{u}_V\|^2_{L^2(\Omega)}+\varepsilon\|\nabla \bar{u}_C\|^2_{L^{2}(\Omega)},&&
\end{flalign*}

With $L^{\infty}(Q_T)$ bounds for $u_C^{(i)}, u_N^{(i)}$, we have:     
\begin{flalign*}
    (\ref{uniq:c1}.4) \leqslant C(\|\bar{u}_C\|^2_{L^2(\Omega)}+\|\bar{u}_N\|^2_{L^2(\Omega)}),&&
\end{flalign*}

 Therefore, the $L^2(\Omega)$ estimate gives:
\begin{flalign}
        &\frac{1}{2} \frac{d}{dt} \int_{\Omega} \bar{u}_C^2dx + d_C^{(0)}\int_{\Omega}|\nabla \bar{u}_C|^2 dx \notag\\ 
        \leqslant & \int_{\Omega}\chi_{11} \nabla \bar{u}_N \cdot \nabla \bar{u}_Cdx + C\|\bar{u}_C\|^2_{L^{\infty}(\Omega)}(\|\nabla u_N^{(2)}\|^2_{L^2(\Omega)}+\|\nabla u^{(2)}_I\|^2_{L^{2}(\Omega)}+\|\nabla u_A^{(2)}\|^2_{L^2(\Omega)}+\|\nabla \hat{u}_V\|_{L^2(\Omega)}^2) \notag\\
        &C(\|\nabla \bar{u}_A \|^2_{L^2(\Omega)}+\|\nabla \bar{u}_I \|^2_{L^2(\Omega)} )+ \varepsilon \|\nabla \bar{u}_C\|_{L^2(\Omega)}^2+C(\|\bar{u}_C\|^2_{L^2(\Omega)}+\|\bar{u}_N\|^2_{L^2(\Omega)}) &&\label{summary:c}
\end{flalign}
Similarly, applying $L^2(\Omega)$ estimate on \eqref{uniq:n}, we obtain:
\begin{flalign}
        &\frac{1}{2} \frac{d}{dt} \int_{\Omega} \bar{u}_N^2dx +  d_N^{(0)}\int_{\Omega}|\nabla \bar{u}_N|^2 dx\notag\\
        \leqslant & \int_{\Omega}\chi_{21} \nabla \bar{u}_N \cdot \nabla \bar{u}_Cdx + C\|\bar{u}_N\|^2_{L^{\infty}(\Omega)}(\|\nabla u_C^{(2)}\|^2_{L^2(\Omega)}+\|\nabla u^{(2)}_I\|^2_{L^{2}(\Omega)}+\|\nabla u_A^{(2)}\|^2_{L^2(\Omega)}+\|\nabla \hat{u}_V\|_{L^2(\Omega)}^2) \notag\\
        &C(\|\nabla \bar{u}_A \|^2_{L^2(\Omega)}+\|\nabla \bar{u}_I \|^2_{L^2(\Omega)} )+ \varepsilon \|\nabla \bar{u}_N\|_{L^2(\Omega)}^2+C(\|\bar{u}_C\|^2_{L^2(\Omega)}+\|\bar{u}_N\|^2_{L^2(\Omega)}) && \label{summary:n}
\end{flalign}
With $L^{\infty}(Q_T)$ bounds for $u_C^{(i)}, u_N^{(i)},u_A^{(i)}, u_I^{(i)}, u_P^{(i)}$, for $i=1,2$, one can easily obtain $L^2(\Omega)$ estimates for \eqref{uniq:a} - \eqref{uniq:p}:
\begin{flalign}
    &\frac{1}{2}\frac{d}{dt} \int_{\Omega} \bar{u}_A^2dx +  d_A^{(0)}\int_{\Omega}|\nabla \bar{u}_A|^2 dx \leqslant C (\|\bar{u}_A\|^2_{L^2(\Omega)}+\|\bar{u}_I\|^2_{L^2(\Omega)}+\|\bar{u}_C\|^2_{L^2(\Omega)})\label{summary:a}\\
    &\frac{1}{2}\frac{d}{dt} \int_{\Omega} \bar{u}_I^2dx +  d_I^{(0)}\int_{\Omega}|\nabla \bar{u}_I|^2 dx\leqslant C (\|\bar{u}_A\|^2_{L^2(\Omega)}+\|\bar{u}_I\|^2_{L^2(\Omega)}+\|\bar{u}_P\|^2_{L^2(\Omega)})\label{summary:i}\\
    &\frac{1}{2}\frac{d}{dt} \int_{\Omega} \bar{u}_P^2dx +  d_P^{(0)}\int_{\Omega}|\nabla \bar{u}_P|^2 dx\leqslant C (\|\bar{u}_A\|^2_{L^2(\Omega)}+\|\bar{u}_I\|^2_{L^2(\Omega)}+\|\bar{u}_C\|^2_{L^2(\Omega)}+\|\bar{u}_P\|^2_{L^2(\Omega)})\label{summary:p}
\end{flalign}

Lastly, similar to section 3, for $\chi_{11}$ and $\chi_{21}$ small enough, we sum up~\eqref{summary:c} -~\eqref{summary:p} with certain weights and obtain the following estimate:
\begin{equation}
\sum_{i\in\I\backslash \{V\}} (\frac{d}{dt}\int_{\Omega} \bar{u}_i^2 dx + \int_{\Omega}|\nabla \bar{u}_i|^2 dx )  \leqslant  C\sum_{i\in\I\backslash \{V\}} \|\bar{u}_i\|_{L^2(\Omega)}^2+C \mathcal{A}_1(\|\bar{u}_C\|^2_{L^{\infty}(\Omega)}+\|\bar{u}_N\|^2_{L^{\infty}(\Omega)}),
\end{equation} 
where we denoted $\mathcal{A}_1: = \|\nabla u_C^{(2)}\|^2_{L^2(\Omega)}+\|\nabla u_N^{(2)}\|^2_{L^2(\Omega)}+\|\nabla u^{(2)}_I\|^2_{L^{2}(\Omega)}+\|\nabla u_A^{(2)}\|^2_{L^2(\Omega)}+\|\nabla \hat{u}_V\|_{L^2(\Omega)}^2$.

Therefore, by Gr\"onwall inequality and the initial condition, $\bar{u}_i(x,0) = 0,\,\forall x \in \Omega,i\in\I\backslash\{V\}$,
we obtain
\begin{equation}
    \sum_{i\in\I\backslash \{V\}} (\sup_{0<t<T}\|\bar{u}_i\|^2_{L^2(\Omega)}+\|\nabla \bar{u}_i\|^2_{L^2(Q_T)}) \leqslant C (\|\bar{u}_C\|^2_{L^{\infty}(Q_T)}+\|\bar{u}_N\|^2_{L^{\infty}(Q_T)}) \int_{0}^{T}\mathcal{A}_1\,dt  .
\end{equation}

Apply~\cref{camponato} on equations~\eqref{uniq:c} -~\eqref{uniq:p}, with similar process in lemma~\ref{max:ucun}, we obtain the following estimates:
\begin{flalign}
    \|\nabla\bar{u}_C\|^2_{\Camp{2}{2}}\leqslant & C(\|\bar{u}_C\|_{L^2(Q_T)}^2 + \|\bar{u}_N + \bar{u}_C\|_{\Camp{2}{0}}^2 + \|\bar{u}_C\|_{L^{\infty}(Q_T)}^2 \mathcal{A}_{2,N}(T))\notag\\
    &+ \tilde{C}_1 \chi_{11}^2\|\nabla \bar{u}_N\|^2_{\Camp{2}{2}} + C\|\nabla \bar{u}_A\|^2_{\Camp{2}{2}}+ C\|\nabla \bar{u}_I\|^2_{\Camp{2}{2}}\notag &&\\
    \leqslant& C(\|\bar{u}_C\|_{L^2(Q_T)}^2+ \|\bar{u}_N\|_{L^2(Q_T)}^2) + C\|\bar{u}_C\|_{L^{\infty}(Q_T)}^2 \mathcal{A}_{2,N}(T)\notag\\
    &+ \tilde{C}_1 \chi_{11}^2\|\nabla \bar{u}_N\|^2_{\Camp{2}{2}} + C\|\nabla \bar{u}_A\|^2_{\Camp{2}{2}}+ C\|\nabla \bar{u}_I\|^2_{\Camp{2}{2}}\label{uniq:campuc}
\end{flalign}
where we denoted $\mathcal{A}_{2,N}(T): =\|\nabla u_N^{(2)}\|_{\Camp{2}{2}}^2+|\nabla u_A^{(2)}\|_{\Camp{2}{2}}^2+|\nabla u_I^{(2)}\|_{\Camp{2}{2}}^2+ |\nabla \hat{u}_V\|_{\Camp{2}{2}}^2.$

Similarly, we have:
\begin{flalign}
    \|\nabla\bar{u}_N\|^2_{\Camp{2}{2}} \leqslant& C(\|\bar{u}_C\|_{L^2(Q_T)}^2+ \|\bar{u}_N\|_{L^2(Q_T)}^2) + C\|\bar{u}_N\|_{L^{\infty}(Q_T)}^2 \mathcal{A}_{2,C}(T)\notag&&\\
    &+ \tilde{C}_2 \chi_{21}^2\|\nabla \bar{u}_C\|^2_{\Camp{2}{2}} + C\|\nabla \bar{u}_A\|^2_{\Camp{2}{2}}+ C\|\nabla \bar{u}_I\|^2_{\Camp{2}{2}},\label{uniq:campun}
\end{flalign}
where we denoted $\mathcal{A}_{2,C}(T): =\|\nabla u_C^{(2)}\|_{\Camp{2}{2}}^2+|\nabla u_A^{(2)}\|_{\Camp{2}{2}}^2+|\nabla u_I^{(2)}\|_{\Camp{2}{2}}^2+ |\nabla \hat{u}_V\|_{\Camp{2}{2}}^2.$
By the assumption on $\chi_{11}$ and $\chi_{21}$ and summing up estimates~\eqref{uniq:campuc} and~\eqref{uniq:campun}, we have:
\begin{flalign}
    \|\nabla\bar{u}_C\|^2_{\Camp{2}{2}}+ \|\nabla\bar{u}_N\|^2_{\Camp{2}{2}}\leqslant& C(\|\bar{u}_C\|_{L^2(Q_T)}^2+ \|\bar{u}_N\|_{L^2(Q_T)}^2)+ C\|\nabla \bar{u}_A\|^2_{\Camp{2}{2}}+ C\|\nabla \bar{u}_I\|^2_{\Camp{2}{2}} \notag \\
     &+ C(\|\bar{u}_C\|_{L^{\infty}(Q_T)}^2 + \|\bar{u}_N\|_{L^{\infty}(Q_T)}^2 )\mathcal{A}_2(T) ,&&
\end{flalign}
where we denoted $\mathcal{A}_2(T): =  \sum_{i\in\I\backslash\{V\}}\|\nabla u_i^{(2)}\|_{\Camp{2}{2}}^2+ |\nabla \hat{u}_V\|_{\Camp{2}{2}}^2.$

Also, we have $\Camp{2}{2}$-estimates on $\bar{u}_A$, $\bar{u}_I$ and $\bar{u}_P$, which are:
\begin{flalign}
    &\|\nabla \bar{u}_A\|^2_{\Camp{2}{2}}\leqslant C(\|\bar{u}_A\|^2_{L^2(Q_T)}+\|\bar{u}_I\|^2_{L^2(Q_T)}+\|\bar{u}_C\|^2_{L^2(Q_T)})\\
    &\|\nabla\bar{u}_I\|^2_{\Camp{2}{2}}\leqslant C(\|\bar{u}_A\|^2_{L^2(Q_T)}+\|\bar{u}_I\|^2_{L^2(Q_T)}+\|\bar{u}_P\|^2_{L^2(Q_T)})\\
    &\|\nabla\bar{u}_P\|^2_{\Camp{2}{2}}\leqslant C(\|\bar{u}_A\|^2_{L^2(Q_T)}+\|\bar{u}_C\|^2_{L^2(Q_T)}+\|\bar{u}_I\|^2_{L^2(Q_T)}+\|\bar{u}_P\|^2_{L^2(Q_T)})\label{uniq:campup} && 
\end{flalign}

Therefore, by estimates~\eqref{uniq:campuc} -~\eqref{uniq:campup}, we have:
\begin{equation}\label{uniq:camp:2}
    \sum_{i\in\I\backslash\{V\}}\|\nabla\bar{u}_i\|^2_{\Camp{2}{2}}\leqslant C(\|\bar{u}_C\|_{L^{\infty}(Q_T)}^2 + \|\bar{u}_N\|_{L^{\infty}(Q_T)}^2 ) (\int_{0}^T \mathcal{A}_1 dt + \mathcal{A}_2(T))
\end{equation}

Then, we proceed with $\Camp{2}{2+2\alpha}$-estimates:
\begin{flalign}
    &\|\nabla \bar{u}_A\|^2_{\Camp{2}{2+2\alpha}}\leqslant C(\|\bar{u}_A\|^2_{L^2(Q_T)}+\|\bar{u}_A\|^2_{\Camp{2}{2\alpha}}+\|\bar{u}_I\|^2_{\Camp{2}{2\alpha}}+\|\bar{u}_C\|^2_{\Camp{2}{2\alpha}})\\
    &\|\nabla\bar{u}_I\|^2_{\Camp{2}{2+2\alpha}}\leqslant C(\|\bar{u}_I\|^2_{L^2(Q_T)}+\|\bar{u}_A\|^2_{\Camp{2}{2\alpha}}+\|\bar{u}_I\|^2_{\Camp{2}{2\alpha}}+\|\bar{u}_P\|^2_{\Camp{2}{2\alpha}})\\
    &\|\nabla\bar{u}_P\|^2_{\Camp{2}{2+2\alpha}}\leqslant C(\|\bar{u}_P\|^2_{L^2(Q_T)}+\|\bar{u}_A\|^2_{\Camp{2}{2\alpha}}+\|\bar{u}_C\|^2_{\Camp{2}{2\alpha}}+\|\bar{u}_I\|^2_{\Camp{2}{2\alpha}}+\|\bar{u}_P\|^2_{\Camp{2}{2\alpha}})\label{uniq:betaup} && 
\end{flalign}
\begin{flalign}
    \|\nabla\bar{u}_C\|^2_{\Camp{2}{2+2\alpha}}\leqslant & C(\|\bar{u}_C\|_{L^2(Q_T)}^2 +\|\bar{u}_C\|_{\Camp{2}{2\alpha}}^2+\|\bar{u}_N\|_{\Camp{2}{2\alpha}}^2) + C\|\bar{u}_C\|_{L^{\infty}(Q_T)}^2\mathcal{A}_{3,N}(T) \notag\\
    &+ \tilde{C}_1 \chi_{11}^2\|\nabla \bar{u}_N\|^2_{\Camp{2}{2+2\alpha}} + C\|\nabla \bar{u}_A\|^2_{\Camp{2}{2+2\alpha}}+ C\|\nabla \bar{u}_I\|^2_{\Camp{2}{2+2\alpha}} &&\\
    \|\nabla\bar{u}_N\|^2_{\Camp{2}{2+2\alpha}}\leqslant & C(\|\bar{u}_N\|_{L^2(Q_T)}^2 +\|\bar{u}_C\|_{\Camp{2}{2\alpha}}^2+\|\bar{u}_N\|_{\Camp{2}{2\alpha}}^2) + C\|\bar{u}_N\|_{L^{\infty}(Q_T)}^2\mathcal{A}_{3,C}(T) \notag\\
    &+ \tilde{C}_2 \chi_{21}^2\|\nabla \bar{u}_C\|^2_{\Camp{2}{2+2\alpha}} + C\|\nabla \bar{u}_A\|^2_{\Camp{2}{2+2\alpha}}+ C\|\nabla \bar{u}_I\|^2_{\Camp{2}{2+2\alpha}}\label{uniq:camp:b}&&
\end{flalign}
where we denoted 
\[\mathcal{A}_{3,j}(T):=\|\nabla u_j^{(2)}\|_{\Camp{2}{2+2\alpha}}^2+\|\nabla u_A^{(2)}\|_{\Camp{2}{2+2\alpha}}^2+\|\nabla u_I^{(2)}\|_{\Camp{2}{2+2\alpha}}^2+ \|\nabla \hat{u}_V\|_{\Camp{2}{2+2\alpha}}^2,\;j\in \{C,N\}.\]

Lastly, with estimates~\eqref{uniq:camp:2} -~\eqref{uniq:camp:b} and an embedding as following: $\forall u\in \{u:\|\nabla u \|_{\Camp{2}{2\alpha}}<\infty\}$, when $\sd\leqslant 2$,
\begin{equation}\label{cm:embedding}
    \|u\|_{\Camp{2}{2\alpha}}\leqslant C\|u\|_{\Mor{2}{2\alpha}}\leqslant C\|u\|_{\Mor{2}{2+2\alpha}}\leqslant C\|u\|_{\Camp{2}{2+2\alpha}}\leqslant C\|\nabla u\|_{\Camp{2}{2\alpha}},
\end{equation}

we obtain the following estimate:
\begin{flalign}
    \sum_{i\in\I\backslash\{V\}}\|\nabla\bar{u}_i\|^2_{\Camp{2}{2+2\alpha}}&\leqslant C(\|\bar{u}_C\|_{L^{\infty}(Q_T)}^2 + \|\bar{u}_N\|_{L^{\infty}(Q_T)}^2 ) (\int_{0}^T \mathcal{A}_1 dt + \mathcal{A}_3(T))\\
    &\leqslant C^*(\|\nabla\bar{u}_C\|_{\Camp{2}{2+2\alpha}}^2 + \|\nabla\bar{u}_N\|_{\Camp{2}{2+2\alpha}}^2 ) (\int_{0}^T \mathcal{A}_1 dt + \mathcal{A}_3(T))
\end{flalign}
where we denoted
$\mathcal{A}_3(T): =\sum_{i\in\I\backslash\{V,P\}}\|\nabla u_i^{(2)}\|_{\Camp{2}{2+2\alpha}}^2+ \|\nabla \hat{u}_V\|_{\Camp{2}{2+2\alpha}}^2$, and notably, $A_2(T)\leqslant CA_3(T)$ by Morrey space embedding.

If we choose $T = \tau$, for some $\tau$ small enough, such that $C^*(\int_{0}^{\tau} \mathcal{A}_1 dt + \mathcal{A}_3(\tau)) < 1$, then we have:
\begin{equation}
    \sum_{i\in\I\backslash\{V\}}\|u_i\|_{C^{\alpha,\frac{\alpha}{2}}(Q_{\tau})}\leqslant \sum_{i\in\I\backslash\{V\}}\|\nabla\bar{u}_i\|_{L^{2,2+2\alpha}_C(Q_{\tau})}\leqslant 0,
\end{equation}
which implies the uniqueness of the solution of the problem~\eqref{ori:c} -~\eqref{ori:p} with boundary conditions~\eqref{bc} and initial conditions \eqref{ic} on $\Omega \times (0,\tau)$. Since we have H\"older estimates on $\bar{Q}_{\tau}$, one can use bootstrap argument  to obtain global existence and uniqueness on $Q_T$ for any $T>0$.

This completes the proof of lemma \ref{sepa:lemma1}
\end{proof}

\input{lemma_contraction}

%% file: lemma_contraction.tex
\begin{lem}\label{lemma2}
    For arbitrary $T>0$ and spatial dimension $\sd\leqslant 2$, the following equation:
\begin{flalign}
        &\partial_t u_V= - \alpha_{31}u_V u_P - \alpha_{32}u_V u_I + \alpha_{33}u_A u_I + \mu_V u_V(1-\frac{u_V}{K_V}),\;\forall (x,t)\in Q_{T};&&\label{contraction:equ}\\
        &u_V(x,0) = u_{V_0}(x), \;\forall x \in \Omega,\label{contraction:init}&&
    \end{flalign}
    where $u_{V_0}\in \{u|\|u\|_{C^{\alpha}(\Omega)}<\infty\,\text{and}\,\|\nabla u\|_{\Camp{2}{\sd+2\alpha}}<\infty\}$ and $u_P,u_I,u_A$ are defined as lemma~\ref{sepa:lemma1} depending on $u_V$, has a unique solution.    
\end{lem}
\begin{proof}
    Defined a norm space, for $\sd\leqslant 2$, $X:=\{u|\|\nabla u\|_{\Camp{2}{\sd+2\alpha}}< \infty\}$, and given $\hat{u}\in X$, we construct a mapping, $M_2[\hat{u}_V] = u_V$ as following:
    \begin{flalign}
        &\partial_t u_V= - \alpha_{31}\hat{u}_V u_P - \alpha_{32}\hat{u}_V u_I + \alpha_{33}u_A u_I + \mu_V \hat{u}_V(1-\frac{\hat{u}_V}{K_V}),\;\forall (x,t)\in Q_T; \label{contraction:v}&&\\
        &u_V(x,0) = u_{V_0}(x), \;\forall x \in \Omega,&&
    \end{flalign}
    where $u_P,u_I,u_A$ are defined as lemma~\ref{sepa:lemma1} depends on $\hat{u}_V$. By the standard existence and uniqueness theory of ODE, this mapping $M_2$ is well-defined. With similar process in section 3, we obtain that $\|u_V\|_X\leqslant C$, i.e. $M_2:X\rightarrow X$. We shall show $M_2$ is a contraction mapping for $\sd=2$. Again, similar process would be applied to dimension $n=1$ case which we will omit here.

    Given $\hat{u}_V^{(1)},\hat{u}_V^{(2)}$, we would have corresponding solutions $(u_i^{(1)})_{i\in \I}$ and $(u_i^{(2)})_{i\in \I}$ respectively. Also, we denoted $\bar{u}_i := u_i^{(1)}-u_i^{(2)}$ for $i\in\I\backslash\{V\}$, $\bar{u}_V:=\hat{u}_V^{(1)}-\hat{u}_V^{(2)}$ and $\tilde{u}_V := u_V^{(1)}-u_V^{(1)}$.

    We claim that the solution in lemma~\ref{sepa:lemma1}, $u_I$, $u_P$, $u_A$, are Lipschitz functions on $\hat{u}_V$ in $\|\nabla\hat{u}_V\|_{\Camp{2}{3}}$ sense:
    \begin{claimproof}
    With similar process in the uniqueness proof in lemma~\ref{sepa:lemma1}, we obtain the following estimates:

    An $L^2$ energy estimate:
        \begin{equation}
            \sum_{i\in\I\backslash \{V\}} (\frac{d}{dt}\int_{\Omega} \bar{u}_i^2 dx + \int_{\Omega}|\nabla \bar{u}_i|^2 dx )  \leqslant  C\sum_{i\in\I} \|\bar{u}_i\|_{L^2(\Omega)}^2+C \mathcal{B}_1(\|\bar{u}_C\|^2_{L^{\infty}(\Omega)}+\|\bar{u}_N\|^2_{L^{\infty}(\Omega)}),
        \end{equation}
        where we denoted $\mathcal{B}_1: = \|\nabla u_C^{(2)}\|^2_{L^2(\Omega)}+\|\nabla u_N^{(2)}\|^2_{L^2(\Omega)}+\|\nabla u^{(2)}_I\|^2_{L^{2}(\Omega)}+\|\nabla u_A^{(2)}\|^2_{L^2(\Omega)}+\|\nabla \hat{u}^{(2)}_V\|_{L^2(\Omega)}^2$.
        
        By Gr\"onwall inequality and the initial condition, $\bar{u}_i(x,0) = 0$, $\forall x \in \Omega$, $i\in\I\backslash\{V\}$, we obtain:
        \begin{equation}
        \sum_{i\in\I\backslash \{V\}} (\sup_{0<t<T}\|\bar{u}_i\|^2_{L^2(\Omega)}+\|\nabla \bar{u}_i\|^2_{L^2(Q_T)}) \leqslant C (\|\bar{u}_C\|^2_{L^{\infty}(Q_T)}+\|\bar{u}_N\|^2_{L^{\infty}(Q_T)}) \int_{0}^{T}\mathcal{B}_1\,dt + C \|\bar{u}_V\|_{L^2(Q_T)}^2.
        \end{equation}
    A $\Camp{2}{2}$ estimate: 
    \begin{multline}
        \sum_{i\in\I\backslash\{V\}}\|\nabla\bar{u}_i\|^2_{\Camp{2}{2}}\leqslant C(\|\bar{u}_C\|_{L^{\infty}(Q_T)}^2 + \|\bar{u}_N\|_{L^{\infty}(Q_T)}^2 ) (\int_{0}^T \mathcal{B}_1 dt + \mathcal{B}_2(T)) \\
        + C \|\bar{u}_V\|_{L^2(Q_T)}^2+C\|\nabla \bar{u}_V\|^2_{\Camp{2}{2}},
    \end{multline}
    where we denoted $\mathcal{B}_2(T): =  \sum_{i\in\I\backslash\{V\}}\|\nabla u_i^{(2)}\|_{\Camp{2}{2}}^2+ |\nabla \hat{u}^{(2)}_V\|_{\Camp{2}{2}}^2$.

    A $\Camp{2}{2+2\alpha}$ estimate:
    \begin{multline}
        \sum_{i\in\I\backslash\{V\}}\|\nabla\bar{u}_i\|^2_{\Camp{2}{2+2\alpha}}\leqslant C^*(\|\nabla \bar{u}_C\|_{\Camp{2}{2+2\alpha}}^2 + \|\nabla\bar{u}_N\|_{\Camp{2}{2+2\alpha}}^2 ) (\int_{0}^T \mathcal{B}_1 dt + \mathcal{B}_3(T)) \\
        + C \|\bar{u}_V\|_{\Camp{2}{2\alpha}}^2+C\|\nabla \bar{u}_V\|^2_{\Camp{2}{2+2\alpha}},
    \end{multline}
    where we denoted $\mathcal{B}_3(T): =  \sum_{i\in\I\backslash\{V\}}\|\nabla u_i^{(2)}\|_{\Camp{2}{2+2\alpha}}^2+ |\nabla \hat{u}^{(2)}_V\|_{\Camp{2}{2+2\alpha}}^2$.

    Then we choose $T=\tau_1$ small enough, such that $C^*(\int_{0}^{\tau_1} \mathcal{B}_1 dt + \mathcal{B}_3(\tau_1))<1$ with embedding~\eqref{cm:embedding}, we obtain
    \begin{equation}\label{lip:camp}
         \sum_{i\in\I\backslash\{V\}}\|\nabla\bar{u}_i\|_{L^{2,2+2\alpha}_C(Q_{\tau_1})}\leqslant C\|\nabla \bar{u}_V\|_{L^{2,2+2\alpha}_C(Q_{\tau_1})}
    \end{equation}
    \end{claimproof}
    
    Proceed on the difference equation of~\eqref{contraction:v}, we have the following equation:
    \begin{multline}
        \partial_t \tilde{u}_V = -\alpha_{31}\bar{u}_V u_P^{(1)}-\alpha_{31}\bar{u}_P \hat{u}_V^{(2)}-\alpha_{32}\bar{u}_V u_I^{(1)}-\alpha_{32}\bar{u}_I \hat{u}_V^{(2)}\\
        +\alpha_{33}\bar{u}_A u_I^{(1)}+\alpha_{33}\bar{u}_I u_A^{(2)} + \mu_V \bar{u}_V-\frac{\mu_V}{K_V}\bar{u}_V(\hat{u}_V^{(1)}+\hat{u}_V^{(2)})
    \end{multline}

    Integrating both side on $(0,\tau)$, and with initial condition $\tilde{u}_V(x,0) = 0$, $\forall x \in \Omega$, we have:
    \begin{multline}
        \tilde{u}_V = \int_{0}^{t}-\alpha_{31}\bar{u}_V u_P^{(1)}-\alpha_{31}\bar{u}_P \hat{u}_V^{(2)}-\alpha_{32}\bar{u}_V u_I^{(1)}-\alpha_{32}\bar{u}_I \hat{u}_V^{(2)}\\
        +\alpha_{33}\bar{u}_A u_I^{(1)}+\alpha_{33}\bar{u}_I u_A^{(2)} + \mu_V \bar{u}_V-\frac{\mu_V}{K_V}\bar{u}_V(\hat{u}_V^{(1)}+\hat{u}_V^{(2)}) ds
    \end{multline}

    Then take gradient and apply $L^{2,2+2\alpha}_C(Q_{\tau_2})$ norm, where $\tau_2 \leqslant \tau_1$, we have:
    \begin{flalign}
         \|\nabla \tilde{u}_V\|^2_{L^{2,2+2\alpha}_C(Q_{\tau_2})} &\leqslant \tau_2 C(\|\nabla \bar{u}_V\|^2_{L^{2,2+2\alpha}_C(Q_{\tau_2})}+\|\nabla \bar{u}_P\|^2_{L^{2,2+2\alpha}_C(Q_{\tau_2})}+\|\nabla \bar{u}_A\|^2_{L^{2,2+2\alpha}_C(Q_{\tau_2})}+\|\nabla \bar{u}_I\|^2_{L^{2,2+2\alpha}_C(Q_{\tau_2})})\notag\\
         &\leqslant  \tau_2 C^{**}\|\nabla \bar{u}_V\|^2_{L^{2,2+2\alpha}_C(Q_{\tau_2})},\label{contraction:uv}
    \end{flalign}
    where the last inequality is due to Lipschitz condition estimate~\eqref{lip:camp}. As we pick $\tau_2$ small enough, such that $ \tau_2 C^{**} <1$, the mapping $M_2$ is contracting on $(0,\tau_2)$. Hence, there exists a unique solution of equation~\eqref{contraction:equ} with initial condition~\eqref{contraction:init} on $Q_{\tau_2}$. Lastly, by continuity argument and bootstrap argument,
    we can extend the time interval from $(0,\tau_2)$ to $(0,T)$, $\forall T>0$. This completes the proof of lemma~\ref{lemma2}.
\end{proof}

With lemma~\ref{sepa:lemma1} and~\ref{lemma2}, we proved the existence and uniqueness of the solution to the system~\eqref{equ:1} -~\eqref{equ:6} with boundary condition~\eqref{bc} and the initial condition~\eqref{ic}. This complete the proof of theorem~\ref{exist}.

%% file: 3d.tex
\section{Proof of theorem~\eqref{exist_3d}}
In this section, we show the proof of theorem~\eqref{exist_3d} by decoupling the system into two sub-problems as the following two lemmas:
\input{lemma_3d_1.tex}

\input{lemma_3d_2.tex}

%% file: lemma_3d_1.tex
\begin{lem}\label{3d:lemma1}
    For arbitrary $T>0$, spatial dimension $\sd=3$, given $\hat{u}_C \in \{u:\|u\|_{L^3(Q_T)}\,\text{and}\,u\geqslant 0\}$, the following system: 
    \begin{flalign}
        \partial_t u_V &= -\alpha_{31} u_V u_P - \alpha_{32} u_V u_I +  \alpha_{33} u_A u_I + \mu_V u_V (1 - \frac{u_V}{K_V}),\label{special:term:cts:uv}\\
    \partial_t u_A &- \nabla \cdot (d_A(x,t)\nabla u_A)= - \alpha_{41}u_A u_I -\alpha_{42}u_A \hat{u}_C + \mu_A \hat{u}_C,\label{special:term:cts:ua}\\
        \partial_t u_I &- \nabla \cdot (d_I(x,t)\nabla u_I)= - \alpha_{51}u_I u_A - \alpha_{52}u_I u_V + \mu_I u_P, \label{special:term:cts:ui}\\
        \partial_t u_P &- \nabla \cdot (d_P(x,t)\nabla u_P)= \alpha_{61} \hat{u}_C u_A+ \alpha_{62}u_V u_I - \delta_P u_P,\label{special:term:cts:up}
    \end{flalign}
    with boundary conditions:
    \[\nabla_{\vec{\nu}} u_A (x,t) = \nabla_{\vec{\nu}} u_I (x,t) = \nabla_{\vec{\nu}} u_P (x,t) = 0,\,\forall (x,t)\in\Omega\times(0,T),\]
    and initial conditions:
     \[u_V(x,0) =u_A(x,0) =u_I(x,0) =u_P(x,0) = (u_{V_0}(x),u_{A_0}(x),u_{I_0}(x),u_{P_0}(x)),\, x\in\Omega\]
     satisfying $\eqref{hypo:init}$, has a unique weak solution in $V_2(Q_T)\cap L^{\infty}(Q_T)$.
\end{lem}
\begin{proof}
   We shall use the same strategy as we prove theorem~\ref{exist}. Consider the following two mappings, for spatial dimension $\sd=3$:  

   Given $(\tilde{u}_A,\tilde{u}_I,\tilde{u}_P) \in (L^{\infty}(Q_T))^3$ and $\hat{u}_V\in L^{\infty}(Q_T)$, the system:
   \begin{flalign}
   \partial_t u_A &- \nabla \cdot (d_A(x,t)\nabla u_A)= - \alpha_{41}u_A \tilde{u}_I -\alpha_{42}\tilde{u}_A \hat{u}_C + \mu_A \hat{u}_C,\label{ua:3d}\\
    \partial_t u_I &- \nabla \cdot (d_I(x,t)\nabla u_I)= - \alpha_{51}u_I \tilde{u}_A - \alpha_{52}u_I \hat{u}_V + \mu_I u_P,\label{ui:3d}\\
    \partial_t u_P &- \nabla \cdot (d_P(x,t)\nabla u_P)= \alpha_{61} \hat{u}_C\tilde{u}_A + \alpha_{62}\hat{u}_V u_I - \delta_P u_P,\label{up:3d}
    \end{flalign} 
    has a unique solution by standard linear parabolic PDE theory, which implies we have a well-defined mapping, $M_3[(\tilde{u}_A,\tilde{u}_I,\tilde{u}_P)] = (u_A,u_I,u_P)$. Follow the similar processes in lemma~\ref{sepa:lemma1} and an embedding:
    \[L^3(Q_T)\hookrightarrow \Camp{2}{\frac{5}{3}}\,\text{for dimension}\,d=3,\]
     we know that $M_3$ has a unique fixed point in $V_2(Q_T)\cap \Hold{\alpha}$, where $\alpha = \frac{1}{3}$.

    Then consider the other mapping, where we give  $\hat{u}_V\in L^{\infty}(Q_T)$, we construct:
    \begin{equation}
        \partial_t u_V = -\alpha_{31} \hat{u}_V u_P - \alpha_{32} \hat{u}_V u_I +  \alpha_{32} u_A u_I + \mu_V \hat{u}_V - \frac{\mu_V}{K_V}\hat{u}^2_V,
    \end{equation}
    where $(u_A,u_I,u_P)$ is the unique fix point of $M_3$ with given  $\hat{u}_V\in L^{\infty}(Q_T)$, and denote it as mapping $M_4$. By standard ODE theory, we know this mapping is a well-defined mapping, where $M_4[\hat{u}_V]= u_V$. Similar to the process in~\cref{lemma2}, we shall briefly show some different steps.

    Recall some notations in~\cref{lemma2}: given $\hat{u}_V^{(1)},\hat{u}_V^{(2)}$, we would have corresponding solutions $(u_i^{(1)})_{i\in \I}$ and $(u_i^{(2)})_{i\in \I}$ respectively, and we denoted $\bar{u}_i := u_i^{(1)}-u_i^{(2)}$ for $i\in\I\backslash\{V\}$ and $\bar{u}_V:=\hat{u}_V^{(1)}-\hat{u}_V^{(2)}$. Then a slightly different estimate strategy in the $L^2$ energy estimate is applied on the term $\int_{\Omega}\hat{u}_C \bar{u}_A \bar{u}_P dx$, which is:
    \begin{align}
        \int_{\Omega}\hat{u}_C \bar{u}_A \bar{u}_P dx &\leqslant \|\hat{u}_C\|_{L^3(\Omega)}\|\bar{u}_A \bar{u}_P\|_{L^{\frac{3}{2}}(\Omega)}\notag\\
        &\leqslant \frac{1}{2}\|\hat{u}_C\|_{L^3(\Omega)}(\|\bar{u}_A^2 + \bar{u}_P^2\|_{L^{\frac{3}{2}}(\Omega)})\notag\\
        &\leqslant \frac{1}{2}\|\hat{u}_C\|_{L^3(\Omega)}(\|\bar{u}_A^2\|_{L^{\frac{3}{2}}(\Omega)} + \|\bar{u}_P^2\|_{L^{\frac{3}{2}}(\Omega)})\notag\\
        &\leqslant \frac{1}{2}\sum_{i\in\{A,P\}}\|\hat{u}_C\|_{L^3(\Omega)} \|\bar{u}_i\|^2_{L^{3}(\Omega)}\notag\\
        &\leqslant C \sum_{i\in\{A,P\}}(\|\hat{u}_C\|_{L^3(\Omega)} \|\bar{u}_i\|_{L^{2}(\Omega)}\|\nabla \bar{u}_i\|_{L^2(\Omega)}+\|\hat{u}_C\|_{L^3(\Omega)} \|u_i\|^2_{L^2(\Omega)})\notag\\
        &\leqslant C \sum_{i\in\{A,P\}}((\frac{1}{4\eta}\|\hat{u}_C\|^2_{L^3(\Omega)} +\|\hat{u}_C\|_{L^3(\Omega)} )\|\bar{u}_i\|^2_{L^{2}(\Omega)} + \eta\|\nabla \bar{u}_i\|^2_{L^2(\Omega)})
    \end{align} 
    where $\eta = \min\{\frac{d_P^{(0)}}{3},\frac{d_A^{(0)}}{3}\}$ the first inequality is due to H\"older's Inequality, the second and the last one is due to Cauchy inequality, the third one is due to Minkowski inequality, the fifth one is due to Gagliardo-Nirenberg inequality for dimension $d=3$:
    \[\|u\|_{L^3(\Omega)}\leqslant C \|\nabla u\|^{\frac{1}{2}}_{L^2(\Omega)}\|u\|^{\frac{1}{2}}_{L^2(\Omega)} + C \|u\|_{L^2(\Omega)},\forall u \in H^1(\Omega).\] 
    Hence, one can have the following estimate:
    \begin{equation}
        \sum_{i\in\{A,I,P\}} \|\bar{u}_i\|^2_{V_2(Q_T)}\leqslant C\|\bar{u}_V\|^2_{L^2(Q_T)}.
    \end{equation}
    Then apply $L^{2,3+2\beta}_C(Q_T)$ estimate with $\beta=\min\{\frac{1}{3},\alpha\}$ where $\alpha$ satisfies~\eqref{hypo:init}, we shall use equation~\eqref{special:term:cts:up} as a representative and the rest will follow the same process:
    \begin{align}
        \|u_P\|^2_{\Hold{\beta}}\leqslant &C\|\nabla u_P\|^2_{L^{2,3+2\beta}_C(Q_T)} \notag\\
        \leqslant& C(\|\bar{u}_P\|^2_{L^2(Q_T)}+ \|\hat{u}_C \bar{u}_A\|^2_{L^{2,1+2\beta}_C(Q_T)} +\|u_I^{(1)} \bar{u}_V\|^2_{L^{2,1+2\beta}_C(Q_T)}\notag\\
        &+\|u_V^{(2)} \bar{u}_I\|^2_{L^{2,1+2\beta}_C(Q_T)} + \|\bar{u}_P\|^2_{L^{2,1+2\beta}_C(Q_T)}\notag\\
        \leqslant & C (\|\bar{u}_P\|^2_{L^2(Q_T)} + \|\hat{u}_C\|^2_{L^{2,1+2\beta}_C(Q_T)} \|\bar{u}_A\|^2_{L^{\infty}(Q_T)} + \|\bar{u}_V\|^2_{L^{\infty}(Q_T)}\|u_I^{(1)}\|^2_{L^{2,1+2\beta}_C(Q_T)}\notag\\
        &+\|\bar{u}_I\|^2_{L^{2,1+2\beta}_C(Q_T)}\|u_V^{(2)}\|_{L^{\infty}(Q_T)} + \|\bar{u}_P\|^2_{L^{2,1+2\beta}_C(Q_T)}\notag\\
        \leqslant& C^*\|\hat{u}_C\|^2_{L^3(Q_T)} \|\bar{u}_A\|^2_{L^{\infty}(Q_T)} + C\|\bar{u}_V\|^2_{L^{\infty}(Q_T)}, \label{camp:lemma51:1}
    \end{align}
    where the last inequality is due to the embedding and $L^{2,1+2\beta}_C(Q_T)$ gradient estimates.
    Similarly, we have:
    \begin{align}
        \|u_A\|^2_{\Hold{\beta}}\leqslant& C^*\|\hat{u}_C\|^2_{L^3(Q_T)} \|\bar{u}_A\|^2_{L^{\infty}(Q_T)} + C\|\bar{u}_V\|^2_{L^{\infty}(Q_T)},\\
        \|u_I\|^2_{\Hold{\beta}}\leqslant& C\|\bar{u}_V\|^2_{L^{\infty}(Q_T)}.\label{camp:lemma51:3}
    \end{align}
    We choose $T = \tau_3$ such that $C^*\|\hat{u}_C\|^2_{L^3(Q_{\tau})}<\frac{1}{2}$, sum up estimates~\eqref{camp:lemma51:1} -~\eqref{camp:lemma51:3} and obtain the following estimate:
    \begin{equation}
        \sum_{i\in\{A,I,P\}} \|\bar{u}_i\|_{L^{\infty}(Q_{\tau_3})}\leqslant C \|\bar{u}_V\|_{L^{\infty}(Q_{\tau_3})}.
    \end{equation}

    Hence, similar to estimate~\eqref{contraction:uv}, we pick $\tau_4\leqslant \tau_3$ and obtain that $M_4$ is a contraction mapping on $L^{\infty}(Q_{\tau_4})$. Then by~\cref{holderlemma} and bootstrap argument, one can extend this existence uniqueness result to arbitrary $T>0$.
    \begin{remark*}
        One can see that $u_V$ is actually in $\Hold{\beta}$, where $\beta:=\min\{\frac{1}{3},\alpha\}$ and $\alpha$ satisfies~\eqref{hypo:init}.
    \end{remark*}
    
    This completes the proof of lemma ~\ref{3d:lemma1}
\end{proof}

%% file: lemma_3d_2.tex
\begin{lem}\label{3d:lemma2}
    For arbitrary $T>0$, spatial dimension $\sd=3$, for some $\chi_{11}$ and $\chi_{21}$ satisfying $(\chi_{11}+\chi_{21})^2< 4 d_C^{(0)}d_N^{(0)}$,the following system: 
    \begin{flalign}
        \partial_t u_C - \nabla \cdot (d_C(x,t)\nabla u_C)= -\nabla\cdot (\chi_{11} B(u_C)&\nabla u_N + B(u_C) \nabla g_1)\notag\\
        & + \alpha_{11}u_N u_C + \mu_C u_C (1-\frac{u_C}{K_C})-\delta_C u_C,\label{lemma52:uc}\\
        \partial_t u_N - \nabla \cdot (d_N(x,t)\nabla u_N)= -\nabla\cdot (\chi_{21} B(u_N)&\nabla u_C +B(u_N)\nabla g_2)\notag\\
        &-\alpha_{21} u_N u_C + \mu_N u_N (1-\frac{u_N}{K_N})-\delta_N u_N, 
    \end{flalign}
    where $g_i:= \chi_{i2}u_A+\chi_{i3}u_I+\chi_{i4}u_V,\,\forall i\in\{1,2\}$, and $u_A, u_I, u_V$ are defined as lemma~\ref{lemma2} depending on $u_C$
    with boundary conditions:
    \begin{equation}
    \nabla_{\vec{\nu}} u_C (x,t) = \nabla_{\vec{\nu}} u_N (x,t)= 0,\,\forall (x,t)\in\Omega\times(0,T),
    \end{equation}
    and initial conditions:
     \begin{equation}
     (u_C(x,0), u_N(x,0)) = (u_{C_0}(x),u_{N_0}(x)),\, x\in\Omega\label{lemma52:init}
     \end{equation}
     satisfying $\eqref{hypo:init}$, has a weak solution in $V_2(Q_T)\cap L^3(Q_T)$, of the form of:
\begin{multline}
    \int_{0}^T \left<\partial_t u_C,\psi \right>_H dt + \int_{0}^T \int_{\Omega} d_C(x,t)\nabla u_C \cdot \nabla \psi dxdt + \frac{\mu_C}{K_C}\int_{0}^T\int_{\Omega} (u_C)^2 \psi dx dt\\
    = \int_{0}^T \int_{\Omega} \chi_{11} B(u_C) \nabla u_N \cdot \nabla \psi dxdt + \int_{0}^T \int_{\Omega}  B(u_C) \nabla g_1 \cdot \nabla \psi dxdt \\
    + \alpha_{11}\int_{0}^T \int_{\Omega} u_Nu_C\psi dxdt + (\mu_C-\delta_C)\int_{0}^T\int_{\Omega} u_C \psi dx dt\notag
\end{multline}
\begin{multline}\label{lemma52:weaksolution:un}
    \int_{0}^T \left<\partial_t u_N,\psi \right>_H dt + \int_{0}^T \int_{\Omega} d_N(x,t)\nabla u_N \cdot \nabla \psi dxdt + \frac{\mu_N}{K_N}\int_{0}^T\int_{\Omega} (u_N)^2 \psi dx dt\\
    = \int_{0}^T \int_{\Omega} \chi_{21} B(u_N) \nabla u_C \cdot \nabla \psi dxdt + \int_{0}^T \int_{\Omega}  B(u_N) \nabla g_2\cdot \nabla \psi dxdt \\
    -  \alpha_{21}\int_{0}^T \int_{\Omega} u_Nu_C\psi dxdt + (\mu_N-\delta_N)\int_{0}^T\int_{\Omega} u_N \psi dx dt
\end{multline}
for arbitrary test function, $\psi\in L^2(0,T;H^1(\Omega))\cap L^3(Q_T)$.  
\end{lem}

\begin{proof} We shall first investigate an approximate system. 
    
    $\bullet$ \textbf{An approximate system:}
    
    For arbitrary $\varepsilon\in (0,1)$, we have the following system for $(u_C^{(\varepsilon)},u_N^{(\varepsilon)})$:
\begin{flalign}
    \partial_t u^{(\varepsilon)}_C - \nabla \cdot (d_C(x,t)\nabla u^{(\varepsilon)}_C)=-\nabla\cdot (\chi_{11} B(u^{(\varepsilon)}_C)&\nabla u^{(\varepsilon)}_N +B(u^{(\varepsilon)}_C) \nabla g^{(\varepsilon)}_1)\label{approx:system:uc}\\ 
    &+ \alpha_{11}u^{(\varepsilon)}_N \frac{u^{(\varepsilon)}_C}{1+\varepsilon u^{(\varepsilon)}_C} + \mu_C u^{(\varepsilon)}_C -\frac{\mu_C}{K_C}(u^{(\varepsilon)}_C)^2-\delta_C u^{(\varepsilon)}_C, \notag\\
    \partial_t u^{(\varepsilon)}_N - \nabla \cdot (d_N(x,t)\nabla u^{(\varepsilon)}_N)=-\nabla\cdot (\chi_{21} B(u^{(\varepsilon)}_N)&\nabla u^{(\varepsilon)}_C + B(u^{(\varepsilon)}_N)\nabla g^{(\varepsilon)}_2)\label{approx:system:un}\\ 
    & -\alpha_{21}u^{(\varepsilon)}_N \frac{u^{(\varepsilon)}_C}{1+\varepsilon u^{(\varepsilon)}_C}  + \mu_N u^{(\varepsilon)}_N -\frac{\mu_N}{K_N}(u^{(\varepsilon)}_N)^2-\delta_N u^{(\varepsilon)}_N,\notag
\end{flalign}
where $g_i^{(\varepsilon)}:= \chi_{i2} u_A^{(\varepsilon)}+\chi_{i3} u_I^{(\varepsilon)}+\chi_{i4} u_V^{(\varepsilon)}$, $\forall i\in \{1,2\}$, and $u_A^{(\varepsilon)},u_I^{(\varepsilon)}, u_V^{(\varepsilon)}$ are the solution of lemma~\ref{3d:lemma1} with given $\hat{u}_C = u_C^{(\varepsilon)}$, with boundary conditions:
 \begin{equation}\label{eps:bdry}
 \nabla_{\vec{\nu}} u^{(\varepsilon)}_C (x,t) = \nabla_{\vec{\nu}} u^{(\varepsilon)}_N (x,t)= 0,\,\forall (x,t)\in\Omega\times(0,T),
 \end{equation}
and initial conditions:
\begin{equation}\label{eps:init}
    (u^{(\varepsilon)}_C(x,0), u^{(\varepsilon)}_N(x,0)) = (u_{C_0}(x),u_{N_0}(x)),\, \forall x\in\Omega.
\end{equation}

$\bullet$ \textbf{$\varepsilon$-independent Apriori estimates:}

We establish some apriori estimate on the approximate system which is independent of $\varepsilon$ by the following claims:
\begin{claim}
    \begin{equation}
        \appuc\geqslant 0,\,\text{and}\,\appun \geqslant 0,\,\forall (x,t)\in Q_T.\label{app:estimate:pos}
    \end{equation}
\end{claim}
\begin{claimproof}
    Similar to lemma~\ref{ap:pos}, one can obtain this non-negativity by the maximum principle.
\end{claimproof}

\begin{claim}
    \begin{align}
    \|u_C^{(\varepsilon)}\|_{L^3(Q_T)}+\|u_C^{(\varepsilon)}\|_{V_2(Q_T)}+\|\partial_t u_C^{(\varepsilon)}\|_{L^2(0,T;H^*(\Omega))}\leqslant C \label{app:estimate:uc} \\
    \|u_N^{(\varepsilon)}\|_{L^3(Q_T)}+\|u_N^{(\varepsilon)}\|_{V_2(Q_T)}+\|\partial_t u_N^{(\varepsilon)}\|_{L^2(0,T;H^*(\Omega))}\leqslant C  \label{app:estimate:un}
    \end{align}
    where $C$ is a constant only depending on known data and being independent of $\varepsilon$.
\end{claim}
\begin{claimproof}
    Multiplying equation~\eqref{approx:system:uc} and ~\eqref{approx:system:un} by $\appuc$ and $\appun$, respectively, and integrating over the spatial domain $\Omega$, we obtain the following inequalities:
    \begin{flalign}
        &\frac{1}{2}\frac{d}{dt} \int_{\Omega}(\appuc)^2 dx + d^{(0)}_C\int_{\Omega}|\nabla u^{(\varepsilon)}_C|^2 dx + \frac{\mu_C}{K_C}\int_{\Omega} (u^{(\varepsilon)}_C)^3dx&&\notag\\
        \leqslant& \int_{\Omega} \chi_{11} B(u^{(\varepsilon)}_C)\nabla u^{(\varepsilon)}_N\cdot \nabla \appuc dx + B(u^{(\varepsilon)}_C)\nabla g^{(\varepsilon)}_2 \cdot \nabla \appun dx \notag\\
        &+ \underbrace{\alpha_{11}\int_{\Omega} (u^{(\varepsilon)}_C)^2 \frac{u^{(\varepsilon)}_N}{1+\varepsilon u^{(\varepsilon)}_C} dx}_{(\ref{appkey:epsilon}.1)}  + (\mu_C-\delta_C)\int_{\Omega}(u^{(\varepsilon)}_C)^2 dx,\label{appkey:epsilon}\\
    &\frac{1}{2}\frac{d}{dt} \int_{\Omega}(\appun)^2 dx + d^{(0)}_N\int_{\Omega}|\nabla u^{(\varepsilon)}_N|^2 dx +\alpha_{21} \int_{\Omega} (u^{(\varepsilon)}_N)^2 \frac{u^{(\varepsilon)}_C}{1+\varepsilon u^{(\varepsilon)}_C} dx + \frac{\mu_N}{K_N}\int_{\Omega} (u^{(\varepsilon)}_N)^3dx&&\notag\\
    \leqslant& \int_{\Omega} \chi_{21} B(u^{(\varepsilon)}_N)\nabla u^{(\varepsilon)}_C\cdot \nabla \appun dx + B(u^{(\varepsilon)}_N)\nabla g^{(\varepsilon)}_2 \cdot \nabla \appun dx + (\mu_N-\delta_N)\int_{\Omega}(u^{(\varepsilon)}_N)^2 dx,\label{appkey:epsilon2}
    \end{flalign}

    Notice that
    \begin{flalign}
        (\ref{appkey:epsilon}.1)&\leqslant \int_{\Omega}\alpha_{11}(\appuc)^{\frac{3}{2}}\frac{(\appuc)^{\frac{1}{2}}\appun}{\sqrt{1+\varepsilon \appuc}} dx &&\notag\\
        &\leqslant \alpha_{21} \int_{\Omega}\frac{\appuc}{1+\varepsilon \appuc} (\appun)^2 dx + \frac{\alpha_{11}^2}{4\alpha_{21}}\int_{\Omega}(\appuc)^3 dx\notag
    \end{flalign}
    Since by~\eqref{hypo:reaction}, we have $\alpha_{11}^2<\frac{4\alpha_{21} \mu_C}{K_C}$, then there exists a constant $\theta>0$ such that:
    \[\alpha_{11}^2 \leqslant (\frac{\theta -1}{\theta})\frac{4\alpha_{21}\mu_C}{K_C},\]
    then we have 
    \[(\ref{appkey:epsilon}.1)\leqslant\alpha_{21} \int_{\Omega}\frac{\appuc}{1+\varepsilon \appuc} (\appun)^2 dx + (\frac{\theta -1}{\theta})\frac{\mu_C}{K_C}\int_{\Omega}(\appuc)^3 dx.\]
    Hence, summing up~\eqref{appkey:epsilon} and~\eqref{appkey:epsilon2} with~\eqref{h2}, we obtain:
    \begin{equation}
        \|u_C^{(\varepsilon)}\|_{L^3(Q_T)}+\|u_C^{(\varepsilon)}\|_{V_2(Q_T)}+ \|u_N^{(\varepsilon)}\|_{L^3(Q_T)}+\|u_N^{(\varepsilon)}\|_{V_2(Q_T)}\leqslant C,
    \end{equation}
    where the process will be the same as lemma~\ref{ap:v2}.

    Next, we shall prove $\|\partial_t\appuc\|_{L^2(0,T;H^*(\Omega))}+\|\partial_t\appun\|_{L^2(0,T;H^*(\Omega))}\leqslant C$. 
    
    Given $v\in \{v: \|v\|_{H^1(\Omega)}\leqslant 1\}$ and denoted the dual pairing between $H^*(\Omega)$ and $H^1(\Omega)$ by $\left<u,v \right>_H$, then we have:
    \begin{flalign}
        &\left<\partial_t\appun,v \right>_H \notag\\
        =& \left<\nabla\cdot(d_N(x,t)\nabla \appun),v \right>_H - \left<\chi_{21}\nabla \cdot(B(\appun)\nabla \appuc,v \right>_H - \left<\nabla \cdot(B(\appun)\nabla g_2^{(\varepsilon)}),v \right>_H\notag\\
        & -\left<\alpha_{21}\appun\frac{\appuc}{1+\varepsilon\appuc},v \right>_H +(\mu_N-\delta_N)\left<\appun,v \right>_H +\frac{\mu_N}{K_N}\left<-(\appun)^2,v \right>_H\notag&&\\
        \leqslant& d_N^{(1)}\|\appun\|_{H^1(\Omega)}+ \chi_{21} \|\appuc\|_{H^1(\Omega)}+\|g_2^{(\varepsilon)}\|_{H^1(\Omega)} + \alpha_{21}\|\appuc\|_{L^{3}(\Omega)}\|\appun\|_{L^{3}(\Omega)}\|v\|_{L^{3}(\Omega)} \\
        &+ (\mu_N-\delta_N) \|\appun\|_{L^2(\Omega)} + \frac{\mu_N}{K_N}\|\appun\|^2_{L^3(\Omega)}\|v\|_{L^{3}(\Omega)}\notag\\
        \leqslant& C (\|\appun\|_{H^1(\Omega)}+\|\appuc\|_{H^1(\Omega)}+\|\appuc\|^2_{L^3(\Omega)}+\|\appun\|^2_{L^3(\Omega)})
    \end{flalign}
    where the second to last inequality is due to H\"older inequality and the last one is due to an embedding $H^1(\Omega)\hookrightarrow L^3(\Omega)$ when $\sd =3$. Hence, by the same process in part 7 of~\cref{ap:v2}, integrating over time, and an embedding $V_2(Q_T)\hookrightarrow L^4(0,T;L^3(\Omega))$ for dimension $\sd = 3$ as it is shown in~\cref{v2:embedding}, we obtain the following two estimates :
    \begin{equation}\label{dual:help}
        \|\partial_t \appun\|_{L^2(0,T;H^*(\Omega))} \leqslant C,
    \end{equation}
   Similarly, one can obtain the following estimate:
    \begin{equation}\label{dual:uc}
        \|\partial_t \appuc\|_{L^2(0,T;H^*(\Omega))}\leqslant C,
    \end{equation}
    where $C$ in estimate~\eqref{dual:help} and~\eqref{dual:uc} are constants independent of $\varepsilon$. 
\end{claimproof}

$\bullet$ \textbf{Existence of the approximate system:}

We shall use Leray-Schauder fixed point theorem on the following mapping $M_5$:

Given $(\tilde{u}^{(\varepsilon)}_C,\tilde{u}^{(\varepsilon)}_N) \in \D:=\{(u,v):\|u\|_{L^3(Q_T)}+\|v\|_{L^3(Q_T)}\leqslant K_0\,\text{and}\,u,v\geqslant 0\}$, where $K_0$ will be specified by~\eqref{K:cond:start} -~\eqref{K:cond:end} later, and denote $\tilde{g}^{(\varepsilon)}_i:= \chi_{i2}\tilde{u}^{(\varepsilon)}_A+\chi_{i3}\tilde{u}^{(\varepsilon)}_I+\chi_{i4}\tilde{u}^{(\varepsilon)}_V$ as the solution of lemma~\ref{3d:lemma1} with given $\hat{u}_C = \tilde{u}^{(\varepsilon)}_C$, and using the same notation for the solution, $(\appuc,\appun)$, they satisfy the following system:
\begin{flalign}
    \partial_t u^{(\varepsilon)}_C - \nabla \cdot (d_C(x,t)\nabla u^{(\varepsilon)}_C)=-\nabla\cdot (\chi_{11} B(&\tilde{u}^{(\varepsilon)}_C)\nabla u^{(\varepsilon)}_N +B(u^{(\varepsilon)}_C) \nabla \tilde{g}^{(\varepsilon)}_1) \label{Schaefer:equ1}\\ 
    & + \alpha_{11}\tilde{u}^{(\varepsilon)}_N \frac{u^{(\varepsilon)}_C}{1+\varepsilon u^{(\varepsilon)}_C} + \mu_C u^{(\varepsilon)}_C -\frac{\mu_C}{K_C}(u^{(\varepsilon)}_C)^2-\delta_C u^{(\varepsilon)}_C, \notag\\
    \partial_t u^{(\varepsilon)}_N - \nabla \cdot (d_N(x,t)\nabla u^{(\varepsilon)}_N)=-\nabla\cdot (\chi_{21} B(&\tilde{u}^{(\varepsilon)}_N)\nabla u^{(\varepsilon)}_C + B(u^{(\varepsilon)}_N)\nabla \tilde{g}^{(\varepsilon)}_2)\label{Schaefer:equ2}\\ 
    & -\alpha_{21}\tilde{u}^{(\varepsilon)}_N \frac{u^{(\varepsilon)}_C}{1+\varepsilon u^{(\varepsilon)}_C}  + \mu_N u^{(\varepsilon)}_N -\frac{\mu_N}{K_N}(u^{(\varepsilon)}_N)^2-\delta_N u^{(\varepsilon)}_N,\notag 
\end{flalign}
with boundary condition~\eqref{eps:bdry} and initial condition~\eqref{eps:init}. By Galerkin method and energy estimate, this system has a unique solution, which implies the mapping $M_5$ is a well-defined mapping, i.e. $M_5[(\tilde{u}^{(\varepsilon)}_C,\tilde{u}^{(\varepsilon)}_N)]= (u^{(\varepsilon)}_C,u^{(\varepsilon)}_N)$. We recall that to apply Leray-Schauder fixed point theorem, one need to show the mapping $M_5$ maps from a nonempty, convex, bounded, closed subset $\D$ in a Banach space to a subset of $\D$ and the mapping is continuous and compact. It is obvious that $\D$ is nonempty, convex, bounded and closed, so we shall show that $\|M_5[u]\|_{(L^3(Q_T))^2}\leqslant K_0,\,\text{and}\,M_5[u]\geqslant 0,\,\forall\,u\in\D$ then compactness and continuity of the mapping:

$\circ$ \textbf{$M_5[\D]\subset\D$:}

Apply $L^2$ energy estimate on equation~\eqref{Schaefer:equ1} and~\eqref{Schaefer:equ2},  we have the following estimate:
\begin{align}
   &\frac{d}{dt} \int_{\Omega}(u^{(\varepsilon)}_C)^2+(u^{(\varepsilon)}_N)^2dx + \int_{\Omega}\frac{d^{(0)}_C}{\gamma}|\nabla u^{(\varepsilon)}_C|^2+\frac{d^{(0)}_N}{\gamma}|\nabla u^{(\varepsilon)}_N|^2dx +\frac{2\mu_C}{K_C}\int_{\Omega}(u^{(\varepsilon)}_C)^3 dx +\frac{2\mu_N}{K_N}\int_{\Omega}(u^{(\varepsilon)}_N)^3 dx \notag\\
   \leqslant& \int_{\Omega} \frac{3\gamma}{d_C^{(0)}}(|\chi_{12}\nabla \tilde{u}^{(\varepsilon)}_A|^2+|\chi_{13}\nabla\tilde{u}^{(\varepsilon)}_I|^2+|\chi_{14}\nabla\tilde{u}^{(\varepsilon)}_V|^2) +\frac{3\gamma}{ d_N^{(0)}} (|\chi_{22}\nabla \tilde{u}^{(\varepsilon)}_A|^2+|\chi_{23}\nabla\tilde{u}^{(\varepsilon)}_I|^2+|\chi_{24}\nabla\tilde{u}^{(\varepsilon)}_V|^2) dx \notag\\
   &+ \max\{2(\mu_C -\delta_C),2(\mu_N -\delta_N),\frac{\alpha_{11}^2}{\varepsilon\kappa} ,\frac{\alpha_{21}^2}{\varepsilon\kappa}\}\int_{\Omega}(u^{(\varepsilon)}_C)^2+(u^{(\varepsilon)}_N)^2dx + \kappa \int_{\Omega} (\tilde{u}^{(\varepsilon)}_N)^2 dx,\label{bounded:equ}
\end{align}
where $\gamma$ is defined the same as it in~\eqref{const:gamma} due to $(\chi_{11}+\chi_{21})^2< 4 d_C^{(0)}d_N^{(0)}$ and $\kappa<< 1$ is a small enough constant, which will be specified by~\eqref{cond:kappa} later. And for several estimates on $\int_{\Omega}|\nabla\tilde{g}^{(\varepsilon)}_1|^2dx$ and $\int_{\Omega}|\nabla\tilde{g}^{(\varepsilon)}_2|^2dx$, we shall show the process on equation~\eqref{special:term:cts:ui} and~\eqref{special:term:cts:up} as an example, and the rest estimates will be the same. By the virtue of the unique, bounded solution of system in~\cref{3d:lemma1} whenever given a $\tilde{u}^{(\varepsilon)}_C\in\D$, one can have uniform bounds, denoted by $(\tilde{K}_V,\tilde{K}_A,\tilde{K}_I,\tilde{K}_P)$, for all corresponding unique solutions such that $(\tilde{u}^{(\varepsilon)}_V,\tilde{u}^{(\varepsilon)}_A,\tilde{u}^{(\varepsilon)}_I,\tilde{u}^{(\varepsilon)}_P)\leqslant (\tilde{K}_V,\tilde{K}_A,\tilde{K}_I,\tilde{K}_P)$, uniform in the sense of $\tilde{u}^{(\varepsilon)}_C$. Here, we use $(\tilde{u}_V,\tilde{u}_A,\tilde{u}_I,\tilde{u}_P)$ as the notation of the solution of system in~\cref{3d:lemma1} when given $\hat{u}_C = \tilde{u}_C^{(\varepsilon)}$ and obtain estimates by applying $L^2$ energy estimate as following:
\begin{align}
    &\frac{d}{dt}\int_{\Omega}  \tilde{u}_I^2+\tilde{u}_P^2 dx +2 d_I^{(0)}\int_{\Omega}|\nabla \tilde{u}_I|^2dx+ 2 d_P^{(0)}\int_{\Omega}|\nabla \tilde{u}_P|^2dx\notag\\
     \leqslant &\int_{\Omega} 2\mu_I \tilde{u}_I \tilde{u}_P dx+\int_{\Omega} 2\alpha_{61} \tilde{u}_C^{(\varepsilon)} \tilde{u}_A \tilde{u}_P + 2\alpha_{62}\tilde{u}_V\tilde{u}_I\tilde{u}_P dx\notag\\
     \leqslant& \kappa \int_{\Omega} (\tilde{u}_C^{(\varepsilon)} )^2 dx + (\frac{\alpha_{61}^2 \tilde{K}_A^2}{\kappa}+\alpha_{62}\tilde{K}_V + \mu_I)\int_{\Omega} \tilde{u}_I^2 +\tilde{u}_P^2 dx.
\end{align}
Similarly,
\begin{align}
    &\frac{d}{dt}\int_{\Omega}  \tilde{u}_A^2dx +2 d_A^{(0)}\int_{\Omega}|\nabla \tilde{u}_A|^2dxdx\leqslant  \kappa\int_{\Omega} (\tilde{u}_C^{(\varepsilon)} )^2 dx + \frac{\mu_A^2}{\kappa}\int_{\Omega}  \tilde{u}_A^2dx,\\
    &\frac{d}{dt}|\nabla \tilde{u}_V|^2 \leqslant C_V |\nabla \tilde{u}_V|^2 + \alpha_{31}\tilde{K}_V|\nabla \tilde{u}_P|^2 + (\alpha_{32}\tilde{K}_V+\alpha_{33}\tilde{K}_A)|\nabla \tilde{u}_I|^2 + \alpha_{33}\tilde{K}_I|\nabla \tilde{u}_A|^2,
\end{align}
where $C_V:=(2\alpha_{31}\tilde{K}_P + \alpha_{31}\tilde{K}_V + 2\alpha_{32}\tilde{K}_I+\alpha_{32}\tilde{K}_V+\alpha_{33}\tilde{K}_I+\alpha_{33}\tilde{K}_A+2\mu_V)$. Then, by Gr\"onwall's inequality,
\begin{multline}
    \int_{Q_T} |\nabla \tilde{u}_I|^2 +|\nabla \tilde{u}_P|^2 dxdt\leqslant \frac{\kappa e^{(\frac{\alpha_{61}^2 \tilde{K}_A^2}{\kappa}+\alpha_{62}\tilde{K}_V + \mu_I) T}}{\min\{2d_I^{(0)},2d_P^{(0)}\}}\int_{Q_T}(\tilde{u}_C^{(\varepsilon)} )^2 dxdt\\
     + \frac{e^{(\frac{\alpha_{61}^2 \tilde{K}_A^2}{\kappa}+\alpha_{62}\tilde{K}_V + \mu_I) T}}{\min\{2d_I^{(0)},2d_P^{(0)}\}} (\|\tilde{u}_{I_0}\|^2_{L^2(\Omega)}+\|\tilde{u}_{P_0}\|^2_{L^2(\Omega)}) 
\end{multline}
\begin{flalign}
    \int_{Q_T}|\nabla \tilde{u}_A|^2 dxdt&\leqslant \frac{1}{2d_A^{(0)}}\kappa e^{(\frac{\mu_A^2}{\kappa})T}\int_{Q_T}(\tilde{u}_C^{(\varepsilon)} )^2 dxdt + \frac{e^{(\frac{\mu_A^2}{\kappa})T}}{2d_A^{(0)}}\|\tilde{u}_{A_0}\|_{L^2(\Omega)}^2\\
   \sup_{0<t<T} \int_{\Omega} |\nabla \tilde{u}_V|^2 dx  &\leqslant e^{C_V T}\|\nabla \tilde{u}_{V_0}\|_{L^2(\Omega)}^2 + C_R e^{C_V T}\int_{Q_T}|\nabla \tilde{u}_P|^2 + |\nabla \tilde{u}_I|^2 + |\nabla \tilde{u}_A|^2 dxdt
\end{flalign}
where $C_R:=\max\{\alpha_{31}\tilde{K}_V,\alpha_{32}\tilde{K}_V+\alpha_{33}\tilde{K}_A,\alpha_{33}\tilde{K}_I\} $. Since we shall only discuss the scales of $\kappa$ and $T$, for constants independent of $\kappa$ and $T$ will be denoted as generic $C$. Applying Gr\"onwall's inequality on~\eqref{bounded:equ}, we have the following estimate:

\begin{align}
    \int_{Q_T} (u^{(\varepsilon)}_C)^3dxdt + \int_{Q_T} (u^{(\varepsilon)}_N)^3dxdt &\leqslant (1+T) (C e^{(\frac{C}{\kappa}+C)T})\kappa \int_{Q_T}(\tilde{u}_C)^2 +(\tilde{u}_N)^2 dxdt\notag\\
    &+Ce^{(\frac{C}{\kappa})T}(\|u_{C_0}\|^2_{L^2(\Omega)} +\|u_{N_0}\|^2_{L^2(\Omega)})+ Te^{CT}\|\nabla u_{V_0}\|^2_{L^2(\Omega)}\notag\\ 
    &+ (1+T)C\kappa e^{(\frac{C}{\kappa}+C)T} (\|u_{A_0}\|^2_{L^2(\Omega)} +\|u_{I_0}\|^2_{L^2(\Omega)} +\|u_{P_0}\|^2_{L^2(\Omega)} )\notag\\
    &\leqslant (1+T) (C_c e^{(\frac{C_c}{\kappa}+C_c)T})\kappa (\|\tilde{u}_C\|_{L^3(Q_T)}^2+\|\tilde{u}_N\|_{L^3(Q_T)}^2) + C^{\sharp}(\kappa,T)
\end{align} 
where $C_c$ is only dependent on known data and independent of $\kappa$ and $T$ and $$C^{\sharp}(\kappa,T)\rightarrow C_c(\|u_{C_0}\|^2_{L^2(\Omega)} +\|u_{N_0}\|^2_{L^2(\Omega)}) + C_c\kappa (\|u_{A_0}\|^2_{L^2(\Omega)} +\|u_{I_0}\|^2_{L^2(\Omega)} +\|u_{P_0}\|^2_{L^2(\Omega)} )$$ as $T\rightarrow 0$. Notice that for each $\kappa>0$, $(1+T) (C e^{(\frac{C}{\kappa}+C)T})$ is a monotonic non-decreasing continuous function of $T$, so does $C^{\sharp}(\kappa,T)$. Hence, we shall choose $\kappa << 1$ satisfying:
\begin{equation}\label{cond:kappa}
2 C_c \kappa K_0^2 + C_c(\|u_{C_0}\|^2_{L^2(\Omega)} +\|u_{N_0}\|^2_{L^2(\Omega)}) + C_c\kappa (\|u_{A_0}\|^2_{L^2(\Omega)} +\|u_{I_0}\|^2_{L^2(\Omega)} +\|u_{P_0}\|^2_{L^2(\Omega)} )< \frac{K_0^3}{8}
\end{equation}
with $K_0$ predefined such that all three following conditions are satisfied:
\begin{align}
    \frac{K_0^3}{24}&>C_c(\|u_{C_0}\|^2_{L^2(\Omega)} +\|u_{N_0}\|^2_{L^2(\Omega)});\label{K:cond:start}\\
    \frac{K_0^3}{24}&>C_c (\|u_{A_0}\|^2_{L^2(\Omega)} +\|u_{I_0}\|^2_{L^2(\Omega)} +\|u_{P_0}\|^2_{L^2(\Omega)})\\
    \frac{K_0^3}{24}&> 2 C_c K_0^2.\label{K:cond:end}
\end{align}
Therefore, by continuity, there exists a $T = \tau^*$ such that $\|u^{(\varepsilon)}_C\|^3_{L^3(Q_{\tau^*})}+\|u^{(\varepsilon)}_N\|^3_{L^3(Q_{\tau^*})} < \frac{K_0^3}{8}$ which implies $\|u^{(\varepsilon)}_C\|_{L^3(Q_{\tau^*})}+\|u^{(\varepsilon)}_N\|_{L^3(Q_{\tau^*})} < K_0$.

For non-negativity part, as we derived~\eqref{app:estimate:pos}, one can use the maximum principle to conclude the non-negativity result.

$\circ$ \textbf{Compactness of $M_5$:}
Since we already know the apriori estimate from~\eqref{app:estimate:pos} -~\eqref{app:estimate:un}, with the same process, one can obtain the following estimates:
\begin{align}
    \|u^{(\varepsilon)}_C\|_{V_2(Q_T)}+ \|u^{(\varepsilon)}_C\|_{L^3(Q_T)} +\|\partial_t u_C\|_{L^2(0,T;H^*(\Omega))}\leqslant C(\varepsilon) \label{app:conv:uc}\\
    \|u^{(\varepsilon)}_N\|_{V_2(Q_T)}+ \|u^{(\varepsilon)}_N\|_{L^3(Q_T)} +\|\partial_t u_N\|_{L^2(0,T;H^*(\Omega))}\leqslant C(\varepsilon) \label{app:conv:un}
\end{align}
where $C(\varepsilon)$ are some constants depending on known data and $\varepsilon$. With these estimates, the Aubin-Lion lemma yields that for any bounded sequence $(\tilde{u}_{C,n}^{(\varepsilon)},\tilde{u}_{N,n}^{(\varepsilon)})_{n\in\N}\subset (L^3(Q_T))^2$, there exists a convergent subsequence $(\tilde{u}_{C,n_j}^{(\varepsilon)},\tilde{u}_{N,n_j}^{(\varepsilon)})_{j\in\N}$ converging to $(v_C,v_N)$ in the $(L^2(0,T;L^3(\Omega)))^2$ sense. 
\begin{claim}
    The subsequence $(\tilde{u}_{C,n_j}^{(\varepsilon)},\tilde{u}_{N,n_j}^{(\varepsilon)})_{j\in\N}$ is convergent in $L^3(Q_T)$ sense.
\end{claim}
\begin{claimproof}
    We recall an important embedding to prove this claim:
    \[L^2(0,T;H^1(\Omega))\cap L^{\infty}(0,T;L^2(\Omega))\hookrightarrow L^4(0,T;L^3(\Omega)),\,\text{for dimension}\,d = 3.\]

    Since $(\tilde{u}_{C,n_j}^{(\varepsilon)},\tilde{u}_{N,n_j}^{(\varepsilon)})_{j\in\N}$ in the $(L^2(0,T;L^3(\Omega)))^2$ sense, we denote $\bar{V}_{i,j} = \tilde{u}_{i,n_j}^{(\varepsilon)} - v_i$ for $i\in\{C,N\}$, and by H\"older inequality, we have:
    \begin{equation}
        \|\bar{V}_{i,j}\|^3_{L^3(Q_T)}\leqslant \|\|\bar{V}_{i,j}\|_{L^3(\Omega)}\|_{L^2(0,T)}\|\|\bar{V}_{i,j}\|^2_{L^3(\Omega)}\|_{L^2(0,T)} = \|\bar{V}_{i,j}\|_{L^2(0,T;L^3(\Omega))} \|\bar{V}_{i,j}\|^2_{L^4(0,T;L^3(\Omega))}\label{l3:conv}
    \end{equation}
    
    By \eqref{l3:conv}, the subsequence $(\tilde{u}_{C,n_j}^{(\varepsilon)},\tilde{u}_{N,n_j}^{(\varepsilon)})_{j\in\N}$  is convergent in $L^3(Q_T)$ sense.

\end{claimproof}

$\circ$ \textbf{Continuity of $M_5$:}

Given two sequences that $\tilde{u}_{C,n}^{(\varepsilon)}\rightarrow \tilde{u}_{C}^{(\varepsilon)}$ and $\tilde{u}_{N,n}^{(\varepsilon)}\rightarrow \tilde{u}_{N}^{(\varepsilon)}$ are convergent in $L^3(Q_T)$ sense. We first claim that:
\begin{claim}
    $(\nabla \tilde{g}^{(\varepsilon)}_{i,n})_{n\in\N}$ converges to $\nabla \tilde{g}^{(\varepsilon)}_{i}$ in $(L^2(Q_T))^3$ sense, for $i\in\{1,2\}$, where $\tilde{g}^{(\varepsilon)}_{i,n}: = \chi_{i2} \tilde{u}^{(\varepsilon)}_{A,n}+\chi_{i3} \tilde{u}^{(\varepsilon)}_{I,n}+\chi_{i4} \tilde{u}^{(\varepsilon)}_{V,n}$ and $\tilde{u}^{(\varepsilon)}_{A,n},\tilde{u}^{(\varepsilon)}_{I,n},\tilde{u}^{(\varepsilon)}_{V,n}$ are the solution of lemma~\ref{3d:lemma1} with given $\hat{u}_C = \tilde{u}^{(\varepsilon)}_{C,n}$, and $\tilde{g}^{(\varepsilon)}_{i}: = \chi_{i2} \tilde{u}^{(\varepsilon)}_{A}+\chi_{i3} \tilde{u}^{(\varepsilon)}_{I}+\chi_{i4} \tilde{u}^{(\varepsilon)}_{V}$ and $\tilde{u}^{(\varepsilon)}_{A},\tilde{u}^{(\varepsilon)}_{I},\tilde{u}^{(\varepsilon)}_{V}$ are the solution of lemma~\ref{3d:lemma1} with given $\hat{u}_C = \tilde{u}^{(\varepsilon)}_{C}$. \label{continuity:claim}
\end{claim}
\begin{claimproof}
    Since $\tilde{u}_{C,n}^{(\varepsilon)}\rightarrow \tilde{u}_{C}^{(\varepsilon)}$ convergent in $L^3(Q_T)$ sense, this implies it is convergent in $L^2(Q_T)$ sense and $\{\tilde{u}_{C,n}^{(\varepsilon)}\}_{n\in\N}\cup \{\tilde{u}^{(\varepsilon)}_{C}\}$ is uniformly $L^2(Q_T)$ bounded. Hence, we have:
    \[\sum_{i\in \{V,A,I,P\}}\|u^{(\varepsilon)}_{i,n}\|_{V_2(Q_T)} + \|u^{(\varepsilon)}_{i,n}\|_{L^{\infty}(Q_T)}\leqslant C,\]
    where $C$ is independent of $n$.

    We will prove that for any $n\in \N$, 
    \begin{equation}\label{continuous:v2:uc}
        \sum_{i\in \{V,A,I,P\}} \|\tilde{u}^{(\varepsilon)}_{i,n} - \tilde{u}^{(\varepsilon)}_i\|^2_{V_2(Q_T)} \leqslant C \|\tilde{u}_{C,n}^{(\varepsilon)}-\tilde{u}_{C}^{(\varepsilon)}\|^2_{L^2(Q_T)}.
    \end{equation}
    This estimate holds due to a simple $L^2(Q_T)$ energy estimate. We shall only show two estimates on $-\alpha_{42} u_A \hat{u}_C$ in equation~\eqref{special:term:cts:ua} and $ +\alpha_{61}\hat{u}_C u_A$ in equation~\eqref{special:term:cts:up} as a representative and the rest terms will be similar:
    \begin{align}
        &-\alpha_{42} \int_{\Omega} (\tilde{u}^{(\varepsilon)}_{A,n} \tilde{u}^{(\varepsilon)}_{C,n} - \tilde{u}^{(\varepsilon)}_A \tilde{u}^{(\varepsilon)}_{C}) (\tilde{u}^{(\varepsilon)}_{A,n}- \tilde{u}^{(\varepsilon)}_A)dx\notag \\
        =& -\alpha_{42} \int_{\Omega} (\tilde{u}^{(\varepsilon)}_{A,n} - \tilde{u}^{(\varepsilon)}_A)^2 \tilde{u}^{(\varepsilon)}_{C}+\tilde{u}^{(\varepsilon)}_{A,n}( \tilde{u}^{(\varepsilon)}_{C,n} -\tilde{u}^{(\varepsilon)}_{C}) (\tilde{u}^{(\varepsilon)}_{A,n} - \tilde{u}^{(\varepsilon)}_A)dx \notag\\
        \leqslant& C(\int_{\Omega}( \tilde{u}^{(\varepsilon)}_{C,n} -\tilde{u}^{(\varepsilon)}_{C})^2dx + \int_{\Omega}(\tilde{u}^{(\varepsilon)}_{A,n} - \tilde{u}^{(\varepsilon)}_A)^2dx) \notag\\
        =& C (\|\tilde{u}^{(\varepsilon)}_{C,n} -\tilde{u}^{(\varepsilon)}_{C}\|_{L^2(\Omega)}^2+\|\tilde{u}^{(\varepsilon)}_{A,n} -\tilde{u}^{(\varepsilon)}_{A}\|_{L^2(\Omega)}^2),
    \end{align}
    where the last inequality is due to nonnegativity of $\tilde{u}^{(\varepsilon)}_C$, the uniform $L^{\infty}(Q_T)$ bound of $\tilde{u}^{(\varepsilon)}_{A,n}$ and Cauchy inequality;
    \begin{flalign}
        &\alpha_{61}\int_{\Omega} ( \tilde{u}^{(\varepsilon)}_{C,n}\tilde{u}^{(\varepsilon)}_{A,n} -\tilde{u}^{(\varepsilon)}_{C} \tilde{u}_A ) (\tilde{u}^{(\varepsilon)}_{P,n}- \tilde{u}_P)dx\notag &&\\
        =&\alpha_{61}\int_{\Omega}\tilde{u}^{(\varepsilon)}_A ( \tilde{u}^{(\varepsilon)}_{C,n} -\tilde{u}^{(\varepsilon)}_{C} ) (\tilde{u}^{(\varepsilon)}_{P,n}- \tilde{u}^{(\varepsilon)}_P)dx +\alpha_{61}\int_{\Omega} \tilde{u}^{(\varepsilon)}_{C,n}(\tilde{u}^{(\varepsilon)}_{A,n} - \tilde{u}_A ) (\tilde{u}^{(\varepsilon)}_{P,n}- \tilde{u}^{(\varepsilon)}_P)dx\notag\\
        \leqslant& C(\int_{\Omega} (\tilde{u}^{(\varepsilon)}_{C,n} -\tilde{u}^{(\varepsilon)}_{C} )^2 dx+\int_{\Omega} (\tilde{u}^{(\varepsilon)}_{P,n}- \tilde{u}^{(\varepsilon)}_P)^2dx) + C\|\tilde{u}^{(\varepsilon)}_{C,n}\|_{L^2(\Omega)}\|\tilde{u}^{(\varepsilon)}_{A,n} - \tilde{u}^{(\varepsilon)}_A \|_{L^4(\Omega)}\|\tilde{u}^{(\varepsilon)}_{P,n} - \tilde{u}^{(\varepsilon)}_P \|_{L^4(\Omega)} \notag\\
        \leqslant& C\|\tilde{u}^{(\varepsilon)}_{C,n} -\tilde{u}^{(\varepsilon)}_{C}\|^2_{L^2(\Omega)}+C\|\tilde{u}^{(\varepsilon)}_{C,n}\|^2_{L^2(\Omega)}\|\tilde{u}^{(\varepsilon)}_{A,n}- \tilde{u}^{(\varepsilon)}_A \|^2_{L^2(\Omega)}\notag\\
        &+ C\|\tilde{u}^{(\varepsilon)}_{P,n} - \tilde{u}^{(\varepsilon)}_P \|_{L^2(\Omega)}^2 + \frac{d_P^{(0)}}{2}\|\nabla(\tilde{u}^{(\varepsilon)}_{P,n} - \tilde{u}^{(\varepsilon)}_P) \|_{L^2(\Omega)}^2
    \end{flalign}
    where the first inequality is due to Cauchy inequality and general H\"older inequality, the second inequality is due to the uniform $L^{\infty}(Q_T)$ bound of $\tilde{u}^{(\varepsilon)}_{A,n}$ Gagliardo-Nirenberg inequality when dimension $n = 3$ and Young's equality. Hence, we have:
  \begin{multline}
    \sum_{i\in \{V,A,I,P\}} \frac{1}{2}\frac{d}{dt}\|\tilde{u}^{(\varepsilon)}_{i,n} - \tilde{u}^{(\varepsilon)}_i\|^2_{L^2(\Omega)} + \sum_{i\in \{A,I,P\}}d_i^{(0)}\|\nabla (\tilde{u}^{(\varepsilon)}_{i,n} - \tilde{u}^{(\varepsilon)}_i)\|^2_{L^2(\Omega)}\\
    \leqslant C \|\tilde{u}_{C,n}^{(\varepsilon)}-\tilde{u}_{C}^{(\varepsilon)}\|^2_{L^2(Q_T)} +  \sum_{i\in \{V,A,I,P\}}C\|\tilde{u}^{(\varepsilon)}_{i,n} - \tilde{u}^{(\varepsilon)}_i\|^2_{L^2(\Omega)},
  \end{multline}
    by Gr\"onwall's inequality, we get estimate~\eqref{continuous:v2:uc}.
\end{claimproof}
Proceed on showing the continuity of $M_5$, for any $n\in \N$, and define the differences $\bar{u}_C:= u^{(\varepsilon)}_{C,n}-u^{(\varepsilon)}_{C}$, $\bar{\tilde{u}}_C:= \tilde{u}^{(\varepsilon)}_{C,n}-\tilde{u}^{(\varepsilon)}_{C}$, $\bar{u}_N:= u^{(\varepsilon)}_{N,n}-u^{(\varepsilon)}_{N}$ and $\bar{\tilde{u}}_N:= \tilde{u}^{(\varepsilon)}_{N,n}-\tilde{u}^{(\varepsilon)}_{N}$. Apply the $L^2$ energy estimate on equations~\eqref{Schaefer:equ1} and~\eqref{Schaefer:equ2} for $\bar{u}_C$ and $\bar{u}_N$:
\begin{flalign}
    &\frac{1}{2}\frac{d}{dt} \int_{\Omega}\bar{u}_C^2 dx + d^{(0)}_C\int_{\Omega} |\nabla \bar{u}_C|^2 dx\notag&&\\
    =&-\int_{\Omega}\nabla\cdot (\chi_{11} (B(\tilde{u}^{(\varepsilon)}_{C,n})-B(\tilde{u}^{(\varepsilon)}_{C}))\nabla u^{(\varepsilon)}_{N} + B(\tilde{u}^{(\varepsilon)}_{C,n})\nabla \bar{u}_N) \bar{u}_C dx\notag\\
    &-\int_{\Omega}\nabla\cdot ((B(\tilde{u}^{(\varepsilon)}_{C,n})-B(\tilde{u}^{(\varepsilon)}_{C}))\nabla \tilde{g}^{(\varepsilon)}_{1,n}+B(u^{(\varepsilon)}_{C}) \nabla (\tilde{g}^{(\varepsilon)}_{1,n} - \tilde{g}^{(\varepsilon)}_{1}))\bar{u}_C dx\notag \\ 
    & + \int_{\Omega}\alpha_{11}(\bar{\tilde{u}}_{N} \frac{u^{(\varepsilon)}_{C,n}}{1+\varepsilon u^{(\varepsilon)}_{C,n}} +\tilde{u}^{(\varepsilon)}_{N} (\frac{u^{(\varepsilon)}_{C,n}}{1+\varepsilon u^{(\varepsilon)}_{C,n}} -\frac{u^{(\varepsilon)}_{C}}{1+\varepsilon u^{(\varepsilon)}_{C}}))\bar{u}_C+ (\mu_C -\delta_C) \bar{u}_C^2 -\frac{\mu_C}{K_C}(u^{(\varepsilon)}_{C,n}+u^{(\varepsilon)}_{C})\bar{u}_C^2 dx, \notag \\
    \leqslant & C\int_{\Omega} (B(\tilde{u}^{(\varepsilon)}_{C,n})-B(\tilde{u}^{(\varepsilon)}_{C}))^2 |\nabla u_N^{(\varepsilon)}|^2dx +  \eta \int_{\Omega} |\nabla \bar{u}_C|^2 dx + \int_{\Omega}\chi_{11}\nabla \bar{u}_N\cdot \nabla\bar{u}_C dx \notag\\
    & + C\int_{\Omega}(B(\tilde{u}^{(\varepsilon)}_{C,n})-B(\tilde{u}^{(\varepsilon)}_{C}))^2 |\nabla \tilde{g}^{(\varepsilon)}_{1,n}|^2 dx+ C\int_{\Omega}|\nabla (\tilde{g}^{(\varepsilon)}_{1,n} - \tilde{g}^{(\varepsilon)}_{1})|^2 dx \notag\\
    &+ C(\int_{\Omega} \bar{\tilde{u}}_N^2 dx + \int_{\Omega} \bar{u}_C^2 dx) + C \int_{\Omega}(\tilde{u}^{(\varepsilon)}_{N})^2 (\frac{u^{(\varepsilon)}_{C,n}}{1+\varepsilon u^{(\varepsilon)}_{C,n}} -\frac{u^{(\varepsilon)}_{C}}{1+\varepsilon u^{(\varepsilon)}_{C}})^2 dx \label{convergence:guild:uc}
\end{flalign}
where the first inequality is due to Cauchy inequality and non-negativity of $u^{(\varepsilon)}_{C,n}$ and $u^{(\varepsilon)}_{C}$ where $\eta$ is a small constant which will be defined later. Similar to estimate~\eqref{convergence:guild:uc}, $\bar{u}_N$ has the following estimate:
\begin{flalign}
    &\frac{1}{2}\frac{d}{dt}\int_{\Omega}\bar{u}_N^2 dx + d^{(0)}_C\int_{\Omega} |\nabla \bar{u}_N|^2 dx\notag\\
    \leqslant&  C\int_{\Omega} (B(\tilde{u}^{(\varepsilon)}_{N,n})-B(\tilde{u}^{(\varepsilon)}_{N}))^2 |\nabla u_C^{(\varepsilon)}|^2dx +  \eta \int_{\Omega} |\nabla \bar{u}_N|^2 dx + \int_{\Omega}\chi_{21}\nabla \bar{u}_C\cdot \nabla\bar{u}_N dx \notag\\
     &+ C\int_{\Omega}(B(\tilde{u}^{(\varepsilon)}_{N,n})-B(\tilde{u}^{(\varepsilon)}_{N}))^2 |\nabla \tilde{g}^{(\varepsilon)}_{2,n}|^2 dx+ C\int_{\Omega}|\nabla (\tilde{g}^{(\varepsilon)}_{2,n} - \tilde{g}^{(\varepsilon)}_{2})|^2 dx \notag\\
    &+ C(\int_{\Omega} \bar{\tilde{u}}_N^2 dx + \int_{\Omega} \bar{u}_N^2 dx) + C \int_{\Omega}(\tilde{u}^{(\varepsilon)}_{N})^2 (\frac{u^{(\varepsilon)}_{C,n}}{1+\varepsilon u^{(\varepsilon)}_{C,n}} -\frac{u^{(\varepsilon)}_{C}}{1+\varepsilon u^{(\varepsilon)}_{C}})^2 dx,&& \label{convergence:guild:un}
\end{flalign}
Summing up estimates~\eqref{convergence:guild:uc} and~\eqref{convergence:guild:un} and applying Gr\"onwall's inequality, with $\chi_{11}$ and $\chi_{21}$ satisfying $(\chi_{11}+\chi_{21})^2< 4 d_C^{(0)}d_N^{(0)}$ and $\gamma$ in inequality~\eqref{const:gamma}, we pick $\eta = \frac{1}{2\gamma}$ we have the following estimate:
\begin{flalign}
    &\sup_{0<t<T}\frac{1}{2}\int_{\Omega}\bar{u}_C^2 +\bar{u}_N^2 dx + \frac{d^{(0)}_C}{2\gamma}\int_{Q_T} |\nabla \bar{u}_C|^2 dx+ \frac{d^{(0)}_N}{2\gamma}\int_{Q_T} |\nabla \bar{u}_N|^2 dx&&\notag\\
    \leqslant &C\int_{Q_T} (B(\tilde{u}^{(\varepsilon)}_{C,n})-B(\tilde{u}^{(\varepsilon)}_{C}))^2 |\nabla u_N^{(\varepsilon)}|^2 +C\int_{Q_T} (B(\tilde{u}^{(\varepsilon)}_{N,n})-B(\tilde{u}^{(\varepsilon)}_{N}))^2 |\nabla u_C^{(\varepsilon)}|^2 dx\notag\\
    & + C\int_{Q_T}(B(\tilde{u}^{(\varepsilon)}_{C,n})-B(\tilde{u}^{(\varepsilon)}_{C}))^2 |\nabla \tilde{g}^{(\varepsilon)}_{1,n}|^2 dx + C\int_{Q_T}(B(\tilde{u}^{(\varepsilon)}_{N,n})-B(\tilde{u}^{(\varepsilon)}_{N}))^2 |\nabla \tilde{g}^{(\varepsilon)}_{2,n}|^2 dx  \notag\\
    &+ C \int_{Q_T}(\tilde{u}^{(\varepsilon)}_{N})^2 (\frac{u^{(\varepsilon)}_{C,n}}{1+\varepsilon u^{(\varepsilon)}_{C,n}} -\frac{u^{(\varepsilon)}_{C}}{1+\varepsilon u^{(\varepsilon)}_{C}})^2 dx  \notag\\
    &+ C\int_{Q_T}|\nabla (\tilde{g}^{(\varepsilon)}_{1,n} - \tilde{g}^{(\varepsilon)}_{1})|^2 dx+ C\int_{Q_T}|\nabla (\tilde{g}^{(\varepsilon)}_{2,n} - \tilde{g}^{(\varepsilon)}_{2})|^2 dx + C \int_{Q_T} \bar{\tilde{u}}_N^2 dx.\label{ultimate:converge}
\end{flalign}
Since $\tilde{u}_{C,n}^{(\varepsilon)}\rightarrow \tilde{u}_{C}^{(\varepsilon)}$ and $\tilde{u}_{N,n}^{(\varepsilon)}\rightarrow \tilde{u}_{N}^{(\varepsilon)}$ are convergent in $L^2(Q_T)$ sense, 
\begin{itemize}[label=$\diamond$]
    \item the first five terms on the right-hand side of estimate~\eqref{ultimate:converge} are convergent to zero:\\
    $B(z)$ and $\frac{z}{1+\varepsilon z}$ are bounded, Lipschitz continuous functions about $z$. Hence $B(\tilde{u}^{(\varepsilon)}_{C,n})$, $B(\tilde{u}^{(\varepsilon)}_{N,n})$, $\frac{u^{(\varepsilon)}_{C,n}}{1+\varepsilon u^{(\varepsilon)}_{C,n}}$ converge to $B(\tilde{u}^{(\varepsilon)}_{C})$, $B(\tilde{u}^{(\varepsilon)}_{N})$, $\frac{u^{(\varepsilon)}_{C}}{1+\varepsilon u^{(\varepsilon)}_{C}}$ strongly in $L^2(Q_T)$, respectively. Then any subsequence of those sequences must be convergent in $L^2(Q_T)$ sense. Also, there exists a subsequence of $L^2(Q_T)$ convergent sequence, which is convergent almost everywhere. Lastly, due to $\nabla u_C^{(\varepsilon)}$, $\nabla u_N^{(\varepsilon)}$, $\nabla g_{1,n}^{(\varepsilon)}$,$\nabla g_{2,n}^{(\varepsilon)}$ are $(L^2(Q_T))^3$, $\tilde{u}_N^{(\varepsilon)}$ is $L^2(Q_T)$ and the upper bounds of $B(z)$ and $\frac{z}{1+\varepsilon z}$, dominated convergent theorem yields the result;
    \item the sixth and seventh term is convergent to zero:\\
    By claim~\ref{continuity:claim}, we have shown this result.
\end{itemize}

Therefore, $(M_5[\tilde{u}^{(\varepsilon)}_{C,n}])_{N,n},\,(M_5[\tilde{u}^{(\varepsilon)}_{N,n}])_{n\in\N}$ is convergent in $V_2(Q_T)$ sense. By the continuous embedding of $V_2(Q_T)\hookrightarrow L^4(0,T;L^3(\Omega))\hookrightarrow L^3(Q_T)$ as it is shown in~\cref{v2:embedding}, then $M_5[\tilde{u}^{(\varepsilon)}_{C,n}],\,M_5[\tilde{u}^{(\varepsilon)}_{N,n}]$ is convergent in $L^3(Q_T)$ sense. This completes the proof of $M_5$ is continuous.

By Leray-Schauder fixed point theorem, we know there exists a fixed point on $Q_{\tau^*}$. Then by continuity argument and bootstrap argument, one can extend this result to any $T>0$, which solves the approximate system~\eqref{approx:system:uc} -~\eqref{eps:init}.

$\bullet$ \textbf{Existence of a solution of lemma~\ref{3d:lemma2}:}

Lastly, we shall let $\varepsilon\rightarrow 0$. By the virtue of estimates~\eqref{app:estimate:pos} -~\eqref{app:estimate:un}, we have the following convergent subsequence $\{u_C^{(\varepsilon_j)}\}_{j},\,\text{and}\,\{u_N^{(\varepsilon_j)}\}_{j}$, with $\varepsilon_j\rightarrow 0$ as $j\rightarrow \infty$, satisfying:
\[\begin{cases}
    u_C^{(\varepsilon_j)}\rightarrow u_C,\,u_N^{(\varepsilon_j)}\rightarrow u_N\,&\text{strongly in}\,L^2(Q_T)\,\text{sense};\\
    B(u_C^{(\varepsilon_j)}) \rightarrow B(u_C), \,B(u_N^{(\varepsilon_j)}) \rightarrow B(u_N)&\text{strongly in}\,L^2(Q_T)\,\text{sense};\\
    g_1^{(\varepsilon_j)} \rightarrow g_1,\,g_2^{(\varepsilon_j)} \rightarrow g_2&\text{strongly in}\,L^2(Q_T)\,\text{sense};\\
    \partial_t u_C^{(\varepsilon_j)}\rightharpoonup \partial_t u_C,\,\partial_t u_N^{(\varepsilon_j)}\rightharpoonup \partial_t u_N &\text{weakly in}\,L^2(0,T;H^*(\Omega))\,\text{sense};\\
    u_C^{(\varepsilon_j)}\rightharpoonup u_C,\,u_N^{(\varepsilon_j)}\rightharpoonup u_N &\text{weakly in}\,L^2(0,T;H^1(\Omega))\,\text{sense};\\
    u_C^{(\varepsilon_j)}\rightharpoonup u_C,\,u_N^{(\varepsilon_j)}\rightharpoonup u_N&\text{weakly in}\,L^3(Q_T)\,\text{sense};\\
    B(u_C^{(\varepsilon_j)})\nabla u_N^{(\varepsilon_j)}\rightharpoonup B(u_C) \nabla u_N &\text{weakly in}\,(L^2(Q_T))^3\,\text{sense};\\
    B(u_N^{(\varepsilon_j)})\nabla u_C^{(\varepsilon_j)}\rightharpoonup B(u_N) \nabla u_C &\text{weakly in}\,(L^2(Q_T))^3\,\text{sense};\\
    \frac{u_C^{(\varepsilon_j)}}{1+\varepsilon_j u_C^{(\varepsilon_j)}} \rightarrow u_C &\text{almost everywhere on}\,Q_T;\\
    \frac{u_C^{(\varepsilon_j)}}{1+\varepsilon_j u_C^{(\varepsilon_j)}}u_N^{(\varepsilon_j)} \rightharpoonup u_C u_N &\text{weakly in}\,L^2(Q_T)\,\text{sense};\\
\end{cases}\]
Hence, a weak solution of the form~\eqref{lemma52:weaksolution:un} is found.

This completes the proof of lemma~\ref{3d:lemma2}.
\end{proof}

Combining lemma~\ref{3d:lemma1} and lemma~\ref{3d:lemma2}, we proved the existence of the solution of the system~\eqref{equ:1} -~\eqref{equ:6} with boundary condition~\eqref{bc} and the initial condition~\eqref{ic} when spatial dimension $n=3$. This complete the proof of theorem~\ref{exist_3d}.

%% file: Conclusion.tex
\section{Conclusion}
In this work, we extend a cancer invasion model to explicitly incorporate host normal cells, a modification that introduces cross-diffusion and results in a cross-diffusion hybrid differential equations system. The model integrates diffusion, chemotaxis, haptotaxis, recruitment by cancer cells, and the logistic growth and natural degradation of the normal cell population. A key feature of this system is that the diffusion coefficients for all components (except non-diffusive vitronectin) are dependent on both space and time. We prove the global existence and uniqueness of a solution for dimension $\sd\leqslant 2$ under mild and practical conditions on the given data. Notably, selecting an appropriate range for the constants $\chi_{11}$ and $\chi_{21}$ is essential to ensure ellipticity, while obtaining bounded solutions requires a technical restriction on these constants. This result shows the mathematical model is well-defined for dimension $\sd\leqslant 2$. The proof strategy for showing well-posedness can be applied to other parabolic and ODE hybrid model. 

For dimension $\sd=3$, we establish $V_2(Q_T)$ apriori estimate for all components and show the existence of a weaker solution of the model. However, some open questions remain, such as the uniqueness of the solution when dimension $\sd=3$. These questions require a further investigation.

%% file: Appendix.tex
\section{Appendix}
\subsection{Proof of \cref{holderlemma}}
\begin{proof}
    This proof will be combined by two parts. First we will show the local estimate for some $0<\tau\leqslant T$. Then using the principle of induction, we extend this estimate to $T$.

    Since we know that the solution $u$ is $L^{\infty}(Q_T)$ and $f,g$ are in $C^{\alpha,\frac{\alpha}{2}}(\bar{Q}_T)$, there exists $M,M_f,M_g>0$ such that $\|u\|_{L^{\infty}(Q_T)}\leqslant M,\|f\|_{L^{\infty}(Q_T)}\leqslant M_f,\|g\|_{L^{\infty}(Q_T)} \leqslant M_g$.

    \textbf{Part 1: $u\in C^{\alpha,\frac{\alpha}{2}}(\bar{Q}_\tau)$, for some $0<\tau\leqslant T$.}
    \begin{claim}\label{localholder}
        $\forall u_0\in C^{\alpha}(\bar{\Omega})$, there exists $0<\tau\leqslant T$, such that $\|u\|_{C^{\alpha,\frac{\alpha}{2}}(\bar{Q}_{\tau})}\leqslant C$, where $C$ is independent of $\tau$.
    \end{claim}
    \begin{claimproof}
        Integrate \eqref{holdereq} over $(0,t)$, we have
    \begin{equation}
        u(x,t) - u_0(x) = \int_0^t a_1 f (x,\xi) u(x,\xi) + a_2 g(x,\xi) - a_3 u(x,\xi)^2 d \xi,\;\forall x\in \Omega.
    \end{equation}

    Consider arbitrary two distinct points $(x,t),(y,s) \in \bar{Q}_\tau:= \bar{\Omega}\times [0,\tau]$ for some $\tau>0$, then taking the difference of $u(x,t)$ and $u(y,s)$, we have:
    \begin{align}
        u(x,t)-u(y,s) &= u_0(x)-u_0(y) + \int_0^t a_1 f (x,\xi) u(x,\xi) + a_2 g(x,\xi) - a_3 u(x,\xi)^2 d \xi \notag\\
        & - \int_0^s a_1 f (y,\xi) u(y,\xi) + a_2 g(y,\xi) - a_3 u(y,\xi)^2 d \xi \notag\\
        & = u_0(x)-u_0(y) +  a_1 \underbrace{\int_0^t (f (x,\xi) u(x,\xi)- f (y,\xi) u(y,\xi)) d\xi}_{(\ref{holderlong}.1)} \notag\\
        & + a_2 \underbrace{\int_0^t g(x,\xi)-g(y,\xi) d\xi}_{(\ref{holderlong}.2)}  - a_3 \underbrace{\int_0^t u(x,\xi)^2 - u(y,\xi)^2  d \xi}_{(\ref{holderlong}.3)} \notag\\
        & +\underbrace{\int_s^t a_1 f (y,\xi) u(y,\xi) + a_2 g(y,\xi) - a_3 u(y,\xi)^2 d \xi}_{(\ref{holderlong}.4)} \label{holderlong}
    \end{align}
    Then we state the estimate for each term in equation \eqref{holderlong} divided by $\delta := \max\{|x-y|^{\alpha},|t-s|^{\frac{\alpha}{2}}\}$:
    \begin{flalign}
        \frac{(\ref{holderlong}.1)}{\delta} &\leqslant \int_0^t u(x,\xi)\frac{(f(x,\xi)-f(y,\xi))}{\max\{|x-y|^{\alpha},|t-s|^{\frac{\alpha}{2}}\}}d\xi + \int_0^t f(y,\xi)\frac{(u(x,\xi)-u(y,\xi))}{\max\{|x-y|^{\alpha},|t-s|^{\frac{\alpha}{2}}\}} d\xi \notag &&\\
        &\leqslant \tau M \|f\|_{C^{\alpha,\frac{\alpha}{2}}(\bar{Q}_\tau)} + \tau M_f \|u\|_{C^{\alpha,\frac{\alpha}{2}}(\bar{Q}_\tau)};
    \end{flalign}

    Similarly,
    \begin{flalign}
        \frac{(\ref{holderlong}.2)}{\delta}&\leqslant\tau M_g \|g\|_{C^{\alpha,\frac{\alpha}{2}}(\bar{Q}_\tau)}; && \notag \\
        \frac{(\ref{holderlong}.3)}{\delta}&\leqslant \int_{0}^t (u(x,\xi)+u(y,\xi)) \frac{(u(x,\xi)-u(y,\xi))}{\max\{|x-y|^{\alpha},|t-s|^{\frac{\alpha}{2}}\}} d\xi\leqslant 2\tau M \|u\|_{C^{\alpha,\frac{\alpha}{2}}(\bar{Q}_\tau)}; \notag
    \end{flalign}

    for the last term,
    \begin{flalign}
        \frac{(\ref{holderlong}.4)}{\delta}&\leqslant (a_1M_f M + a_2M_g +a_3 M^2) \frac{(t-s)}{\max\{|x-y|^{\alpha},|t-s|^{\frac{\alpha}{2}}\}}\notag\\
        &\leqslant (a_1M_f M + a_2M_g +a_3 M^2) (t-s)^{1-\frac{\alpha}{2}}&&\notag\\
        &\leqslant (a_1M_f M + a_2M_g +a_3 M^2) \tau^{1-\frac{\alpha}{2}}\leqslant C,\notag
    \end{flalign} 
    as long as $\tau\leqslant T$.

    Hence, combining all previous estimates on each term, we have:
    \begin{align}
        \frac{(u(x,t)-u(y,s))}{\max\{|x-y|^{\alpha},|t-s|^{\frac{\alpha}{2}}\}} &\leqslant 2\tau M \|u\|_{C^{\alpha,\frac{\alpha}{2}}(\bar{Q}_\tau)} + \tau M_f \|u\|_{C^{\alpha,\frac{\alpha}{2}}(\bar{Q}_\tau)}  + \|u_0\|_{C^{\alpha}(\bar{\Omega})} \notag \\
        &+\tau M \|f\|_{C^{\alpha,\frac{\alpha}{2}}(\bar{Q}_\tau)}+\tau M_g \|g\|_{C^{\alpha,\frac{\alpha}{2}}(\bar{Q}_\tau)} +C\notag\\
        &\leqslant \tau (2M+M_f)\|u\|_{C^{\alpha,\frac{\alpha}{2}}(\bar{Q}_\tau)} +C
    \end{align}

    Lastly, we take the supremum by running over arbitrary distinct two points $(x,t), (y,s) \in Q_\tau$ and pick $\tau = \min\{\frac{1}{2}\frac{1}{(2M+M_f)},T\}$, i.e. we have 
    \[\|u\|_{C^{\alpha,\frac{\alpha}{2}}(\bar{Q}_\tau)} = \sup_{(x,t)\neq(y,s)\in \bar{Q}_\tau}\frac{(u(x,t)-u(y,s))}{\max\{|x-y|^{\alpha},|t-s|^{\frac{\alpha}{2}}\}}\leqslant 2C\]
    \end{claimproof}
    \begin{claim} (Bootstrap argument)\label{bootstrap:arg}
        Given $u(\cdot,0) = u_0\in C^{\alpha}(\bar{\Omega})$ and $u$ satisfies equation~\eqref{holdereq}, suppose $T^*:=\sup\{\tau: \|u\|_{C^{\alpha,\frac{\alpha}{2}}(\bar{Q}_\tau)}\leqslant C\}$, then $T^* \equiv T$.
    \end{claim}
    \begin{claimproof}
        Option 1: (Direct proof) By the principle of induction, we need to show that if $\|u\|_{C^{\alpha,\frac{\alpha}{2}}(\bar{\Omega}\times [0,\tau))}\leqslant C$, where $C$ is independent of $\tau$, then $\|u(\cdot,\tau)\|_{C^{\alpha}(\bar{\Omega})}\leqslant C $. 
        
        Notice that
        \begin{equation}
            \|u(\cdot,\tau)\|_{C^{\alpha}(\bar{\Omega})}\leqslant \|u\|_{C^{\alpha,\frac{\alpha}{2}}(\bar{\Omega}\times [0,\tau])} =\|u\|_{C^{\alpha,\frac{\alpha}{2}}(\bar{\Omega}\times [0,\tau))}\leqslant C
        \end{equation}
        Combined with claim~\ref{localholder} and given $u_0\in C^{\alpha}(\bar{\Omega})$, we can conclude $T^*=T$ by induction.

        Option 2: (Proof by contradiction) Suppose $T^* < T$, otherwise there is nothing needs to be proved. Notice that
        \begin{equation}
        \|u(\cdot,T^*)\|_{C^{\alpha}(\Omega)}\leqslant \|u\|_{C^{\alpha,\frac{\alpha}{2}}(\bar{Q}_{T^*})} = \|u\|_{C^{\alpha,\frac{\alpha}{2}}(\bar{\Omega}\times [0,T^*))}\leqslant C.
        \end{equation} 
        Hence, by claim~\ref{localholder}, we know there exists $0<\tau\leqslant T-T^*$, which contradicts to the maximality of $T^*$. So it is impossible to have $T^*<T$.
    \end{claimproof}
\end{proof}